\documentclass[11pt, reqno]{amsart}

\usepackage{amsmath}
\usepackage{fullpage}
\usepackage{amssymb}
\usepackage{amsthm}
\usepackage{bbm}
\usepackage{subfigure}
\usepackage{float}
\usepackage{graphicx}
\usepackage{color}
\usepackage{xcolor}
\usepackage{epstopdf}
\usepackage{enumerate}
\usepackage{enumitem}
\graphicspath{{./figs/}}
\newcommand{\e}{\mathrm{e}}
\newcommand{\dd}{\,\mathrm{d}}
\newcommand{\R}{\mathbb{R}}
\newcommand{\Z}{\mathbb{Z}}
\newcommand{\T}{\mathbb{T}}

\newcommand{\lb}{\langle}
\newcommand{\rb}{\rangle}
\newtheorem{lemma}{Lemma}[section]
\newtheorem{theorem}[lemma]{Theorem}
\newtheorem{proposition}[lemma]{Proposition}

\newtheorem{definition}[lemma]{Definition}
\allowdisplaybreaks
\newcommand{\p}{\partial}
\DeclareMathOperator{\supp}{supp}

\DeclareMathOperator{\Imag}{Im}

\newcommand{\fe}{{\rm e}}
\newcommand{\eps}{\varepsilon}

\numberwithin{equation}{section}
\newcommand{\msc}[2][2000]{%
	\let\@oldtitle\@title%
	\gdef\@title{\@oldtitle\footnotetext{#1 \emph{Mathematics subject
				classification.} #2}}%
}

\title[A filtered scheme for dNLS]{Explicit Fourier Integrator for the Periodic dNLS via Gauge Transformation: Low‐Regularity Estimates in Discrete Bourgain Spaces}

\author{Lun Ji}
\address{Department of Applied Mathematics, The Hong Kong Polytechnic University, Hung Hom, Hong Kong (L.~Ji)}
\email{lun-422.ji@polyu.edu.hk}

\author{Hang Li}
\address{Laboratoire Jacques-Louis Lions, Sorbonne Université, UPMC, 4 place Jussieu, Paris 75005, France (H.~Li)}
\email{hang.li.1@sorbonne-universite.fr}

\author{Alexander Ostermann}
\address{Digital Science Center\\Universit\"at Innsbruck\\Innrain~15\\6020 Innsbruck\\Austria (A.~Ostermann)}
\email{alexander.ostermann@uibk.ac.at}

\author{Gangfan Zhong}
\address{School of Mathematics, Statistics and Mechanics, Beijing University of Technology, 100 Pingleyuan, Chaoyang District, Beijing 100124, China (G.~Zhong)}
\email{gfzhong@emails.bjut.edu.cn}

\subjclass[2020]{65M12, 65M15, 65M70}

\keywords{Derivative nonlinear Schr\"odinger equation; Bourgain space; low-regularity; error estimate}
\thanks{L. Ji is partially supported by the Research Grants Council of Hong Kong (grant No.~15306123). H. Li is supported by the European Research Council (ERC) under the European Union's Horizon 2020 research and innovation program (grant agreement No.~850941), as well as by a postdoctoral fellowship from the Foundation Sciences Math\'ematiques de Paris (FSMP)}

\begin{document}
	
	\begin{abstract}
		The derivative nonlinear Schr\"odinger equation is a fundamental model for the propagation of nonlinear dispersive waves in, for example, plasma physics and nonlinear optics. In this work, we consider this model on the one-dimensional torus and study a filtered explicit Fourier integrator for the corresponding periodic problem. After applying a periodic gauge transformation, we consider a frequency-truncated model and its filtered exponential-Euler discretization. The main difficulty comes from the derivative cubic nonlinearity in the periodic setting, since local smoothing is unavailable and resonant interactions are stronger than in the non-periodic case. To address this issue, we develop a discrete Bourgain-space framework adapted to the gauge-transformed equation. For initial data $u_0 \in H^s(\mathbb{T})$ with $1/2 < s \le 5/2$, we prove that the numerical error is of order $\mathcal{O}(\tau^{s/2-1/4})$  in $H^{1/2}(\mathbb{T})$, where $\tau$ denotes the employed time step size. Numerical experiments confirm the predicted convergence behavior and demonstrate the effectiveness of the filtered scheme for rough solutions.
	\end{abstract}
	
	\maketitle
	
	%%%%%%%%%%%%%%%%%%%%%%%%%%%
	\section{Introduction}
	%%%%%%%%%%%%%%%%%%%%%%%%%%%
	In recent years, the low-regularity analysis of numerical methods for dispersive equations has developed rapidly \cite{ostfocm,caolly,caonew,lima2022ns,lischratzzivcovich2023kg,liwunls,WuandYao2022,feng2025explicit,feng_maierhofer_schratz_2023,fengschratz2024sg}, with most results devoted to semilinear equations whose nonlinearities do not contain derivatives. For some derivative models, such as KdV equation \cite{Roupaa,liwu2026kdv}, the apparent derivative loss can still be compensated by refined phase expansions and resonance structures. However, the derivative nonlinear Schr\"odinger equation has a different structure: the derivative loss in the original variable cannot be compensated by resonance-based arguments in the low-regularity regime. This makes gauge transformations essential for rough data, as they recast the equation into a form with a more favorable multilinear structure and thereby allow Bourgain-space estimates to be applied. The purpose of this work is to address this issue and to develop an explicit filtered scheme suitable for rough solutions of the derivative nonlinear Schr\"odinger (dNLS) equation on the one-dimensional torus $\mathbb{T}=\mathbb{R}/2\pi\mathbb{Z}$:
	\begin{equation}\label{dnls}
		\left\{
		\begin{aligned}
			&\partial_t u -i \partial_x^2 u = \lambda \partial_x\bigl(|u|^2 u\bigr), \quad t>0, \ x \in \mathbb{T},\\
			&u(0,x) = u_0(x),
		\end{aligned}
		\right.
	\end{equation}
	where $u(t,x):[0,T]\times\mathbb{T}\to\mathbb{C}$ denotes the complex-valued  field. Throughout the paper, we assume without loss of generality that $\lambda=1$, since the Cauchy problem~\eqref{dnls} can be reduced to this case by the transformation
	\[
	u(t,x)\mapsto \frac{1}{\sqrt{|\lambda|}}\,u\bigl(t,\operatorname{sign}(\lambda)x\bigr), \qquad \lambda\in \mathbb{R} \backslash \{0\}.
	\]
	The dNLS equation was introduced by Mj{\o}lhus~\cite{mjolhus1976modulational} as a model for the propagation of Alfv\'en waves~\cite{mjolhus1989nonlinear} in magnetized plasmas~\cite{mio1976modified}, and it has since found broad applications in several areas of physics, including nonlinear optics and water-wave theory~\cite{Dysthe1979-waterNLS,mosesself2007}. From the analytical point of view, it is considerably more delicate than the standard cubic nonlinear Schr\"odinger (NLS) equation, since the derivative affects directly the cubic nonlinearity and thus gives rise to a more singular multilinear structure, especially in low-regularity regimes where classical energy estimates are no longer sufficient. 
	
	The derivative in the nonlinearity also fundamentally changes the structural features of dNLS~\eqref{dnls} compared with the classical cubic NLS. Although the mass
	\begin{equation}\label{Eqn:mass}
		M(u)=\int_{\mathbb{T}} |u|^2 \dd x
	\end{equation}
	remains conserved, the classical energy structure is substantially modified. For sufficiently smooth solutions $u\in C([0,T],H^3(\mathbb{T}))\cap C^1([0,T],H^1(\mathbb{T}))$, the conserved energy is given by
	\begin{equation}\label{Eqn:energy}
		E(u)=\int_{\mathbb{T}} |u_x|^2\dd x+\frac{3}{2}\Imag\int_{\mathbb{T}} |u|^2u\overline{u}_x\dd x+\frac{1}{2}\int_{\mathbb{T}} |u|^6\dd x,
	\end{equation}
	which contains both a derivative quartic term and a sextic term, and is not sign-definite. Moreover, the presence of $\partial_x(|u|^2u)$ destroys the classical Galilean invariance and leads to resonance mechanisms that are absent in the standard cubic NLS. These structural differences make dNLS substantially more difficult to analyze, both at the PDE level and for the design of numerical methods~\cite{bahouri2022global,colliander2001global,colliander2002refined,deng2021optimal,li2018numerical,wu2014global,wu2015global}.
	
	A major breakthrough in the analysis of dNLS came from the introduction of gauge transformations, which recast the equation into a form with a more favorable multilinear structure and better analytical features; see, for instance, \cite{Yang1974-intFormalism,YangMills1954}. In the non-periodic setting, Hayashi and Ozawa~\cite{Hayashi1993,HayashiOzawa1992} introduced such a transformation and thereby opened the way to the sharp local theory established by Takaoka~\cite{Takaoka1999}, who proved local well-posedness of \eqref{dnls} in $H^s(\mathbb{R})$ for $s\ge \frac12$. Although the transformed equation still contains derivative-type nonlinear terms, such as $u^2\partial_x\overline{u}$, it becomes accessible to Bourgain's Fourier restriction norm method \cite{bourgain1993} together with local smoothing and Strichartz estimates. The threshold $s=\frac12$ is sharp, since Biagioni and Linares~\cite{BiagioniLinares} showed that below this regularity the flow map is no longer uniformly continuous. Subsequent progress extended the theory from local to global regimes. In particular, Tao \emph{et al.}~\cite{colliander2001global,colliander2002refined} proved global well-posedness for $s>\frac12$ under a smallness assumption on the $L^2$ norm by means of the $I$-method, while Bahouri and Perelman~\cite{bahouri2022global} later established global well-posedness for arbitrary data in $H^{1/2}(\mathbb{R})$ together with uniform-in-time bounds on the $H^{1/2}$ norm of the solution. Furthermore, the dNLS also possesses a rich class of soliton solutions, which are closely connected with its integrable structure and long-time dynamics~\cite{kaup1978exact,mio1976modified,mjolhus1989nonlinear,mjolhus1986alfven}.
	
	However, the analysis becomes more challenging in the periodic setting because one no longer has the local smoothing available on $\mathbb{R}$ and the periodic Strichartz estimates exhibit derivative loss beyond the $L^4$ level. In order to treat the derivative nonlinearity within the framework of periodic Bourgain spaces, it is essential to apply a gauge transformation $v=\mathcal{G}(u)$ (see Definition~\ref{def:gauge_trafo}), through which the problematic term $|u|^2\partial_x u$ in the nonlinearity is rewritten in a more favorable form involving $v^2\partial_x\overline{v}$. The fact that the derivative falls on the complex conjugate is crucial for the resonance analysis and ultimately makes it possible to establish the required trilinear estimates. This approach was carried out by Herr~\cite{Herr2006}, who proved a sharp trilinear estimate on the torus by combining refined multiplier estimates with a periodic $L^4$ Strichartz estimate. Together with the inverse transformation, this yields local well-posedness in $H^s(\mathbb{T})$ for $s\ge \frac12$, as well as global well-posedness in $H^1(\mathbb{T})$ for small $L^2$ data. More recently, Deng \textit{et al.}~\cite{deng2021optimal} established optimal local well-posedness for the $L^2$-critical equation in Fourier--Lebesgue spaces whose scaling corresponds to $H^s(\mathbb{T})$ with $0<s<\frac12$. By comparison, the numerical analysis of \eqref{dnls} remains rather limited. The presence of the derivative in the nonlinearity makes the construction and analysis of particular schemes difficult, while the available convergence theory has so far been largely restricted to smooth solutions. Existing works by Li \textit{et al.}~\cite{li2018numerical} developed implicit finite difference methods for \eqref{dnls}, but their convergence analysis relies on high-regularity assumptions. Consequently, there remains a substantial gap between the low-regularity well-posedness theory and the rigorous numerical approximation of \eqref{dnls}, and in particular the design of efficient explicit integrators for rough initial data is still open.
	
	In this paper, we introduce a \emph{discrete Bourgain framework} to analyze the error of the filtered exponential integrator \eqref{filteredsch}. The discrete Bourgain spaces, initially developed by Ostermann \emph{et al.}~\cite{Ostjems} for the nonlinear Schrödinger equation, have since been extended to a broader class of dispersive equations~\cite{Jimcom,JiZhao,Roupaa} and to higher-dimensional settings~\cite{Jisiam,Jiima}. By incorporating a temporal regularity parameter~$b$, it enables the effective exploitation of the interaction between spatial and temporal frequencies, thus capturing resonances with greater precision. Our key innovations include:
	\begin{itemize}
		\item[(i)] We develop a discrete Bourgain framework adapted to \eqref{dnls}, within which we establish multilinear estimates under low-regularity assumptions.
		\item[(ii)] We construct a filtered exponential integrator tailored to the gauge-transformed dNLS equation and prove its stability and convergence for rough initial data.
		\item[(iii)] We incorporate resonance analysis into the numerical treatment of derivative-type cubic nonlinearities, which leads to a robust and accurate scheme in low-regularity regimes.
	\end{itemize}
	
	Our main result gives a low-regularity error bound for the scheme \eqref{filteredsch}. Let $v^n$ be the numerical approximation generated by \eqref{filteredsch} for the gauge-transformed variable $v=\mathcal{G}(u)$, see \eqref{Eqn:gauge}. A key feature of this scheme is that it filters out high frequencies from the solution. This is achieved via the projection operator \eqref{Eqn:projection}, which removes all frequencies $k$ satisfying $|k| > K$. The maximum retained frequency $K$ is coupled to the time step size $\tau$ through the CFL-type condition $\tau K^2\le 1$. The main result of this paper is the following theorem.
	
	\begin{theorem}\label{Thm:main}
		Let $u_0\in H^s(\mathbb{T})$ for $s\in(\frac{1}{2},\frac{5}{2}]$, and let $[0,T_{\max})$ be the maximal  interval of existence of the corresponding solution to \eqref{dnls}. Then, for any fixed $T<T_{\max}$, there exists $\tau_0>0$ such that for all $\tau\le \tau_0$, the numerical solution $v^n$ produced by \eqref{filteredsch} satisfies
		\begin{equation}\label{Eqn:main-result}
			\|u(t_n)-\mathcal{G}^{-1}(v^n)\|_{H^{1\over2}}
			\lesssim
			\|v(t_n)-v^n\|_{H^{1\over2}}
			\leq C_T
			\tau^{\frac{s}{2}-\frac14},
			\qquad 0<t_n=n\tau\le T.
		\end{equation}
		where $C_T$ is a constant depending on $u_0$ and $T$.
	\end{theorem}

	\subsection*{Outline of the paper.}
	The rest of the paper is organized as follows. Section~\ref{sectionmethod} introduces the gauge transformation for the periodic dNLS equation. Section~\ref{Sec:bourgain} recalls the Bourgain space framework and establishes the corresponding well-posedness and approximation results for the projected problem. Section~\ref{Sec:integrator} introduces the filtered exponential integrator \eqref{filteredsch}. Section~\ref{Sec:discrete-bourgain} develops discrete Bourgain spaces and the discrete multilinear estimates used in the numerical error analysis. The local and global error analyses are carried out in Sections~\ref{sectionlocal} and~\ref{sectionglobal}, respectively. Finally, numerical experiments illustrating the theoretical results are presented in Section~\ref{sectionnumerexp}. The required estimates for the trilinear term with derivative are proved in Appendix~\ref{multilinear}.

	\subsection*{Notations.}
	\begin{itemize}[leftmargin=1em]
		\item For $a, b\ge 0$, we write \( a \lesssim b \) to indicate that there exists a constant \( C > 0 \), independent of \( \tau \in (0,1] \), such that \( a \le Cb \). If the constant depends on a parameter \( \gamma \), we write \( \lesssim_\gamma \). The notation \( a \sim b \) means both \( a \lesssim b \) and \( b \lesssim a \).
		
		\item For $y\in \mathbb{R}$, the Japanese bracket is $\langle y\rangle=(1+|y|^2)^{\frac12}$.
		
		\item For a sequence of functions \( \{u^n\}_{n\in \mathbb{Z}} \), with each \( u^n \) taking values in a Banach space \( B \), we define the following discrete-in-time norms  
		\[
		\|u^n\|_{\ell_{\tau}^p B} = \left( \tau \sum_n \|u^n\|_B^p \right)^{1/p}, 
		\quad 
		\|u^n\|_{\ell^{\infty} B} = \sup_{n \in \mathbb{Z}} \|u^n\|_B.
		\]
		Similarly, for a function \( u(t,x) \in B \), we define the corresponding continuous-in-time norms by  
		\[
		\|u\|_{L^p_tB} = \left( \int_{\mathbb{R}} \|u(t)\|_B^p   \dd t \right)^{1/p}, 
		\quad 
		\|u\|_{L^\infty_t B} = \sup_{t \in \mathbb{R}} \|u(t)\|_B.
		\]
		
		\item For a function $f(x)=\sum\limits_{k\in\mathbb{Z}}\widehat{f}_k\fe^{i  k x}$, we define the operator $\langle \partial_x \rangle^s$ via Fourier multipliers:
		\begin{align*}
			\langle \p_x \rangle^s f=\sum\limits_{k\in\Z}\langle k\rangle^s\widehat{f}_k\fe^{i  k  x },\quad s \in\mathbb{R}.
		\end{align*}
		We define the Sobolev norm for $H^s(\mathbb{T})$ as 
		$$\|f\|_{H^s}^2=2\pi\sum_{k\in \mathbb{Z}}(1+k^2)^s|\widehat{f}_k|^2.$$
		Across this paper, for the sake of brevity, we use $H^s$ to denote the functions defined on $\mathbb{T}$ with the above norm $\| \cdot\|_{H^s}$. Moreover, for $s=0$, the space reduces to $L^2$ and the corresponding norm denoted as $\|\cdot \|_{L^2}$, which is derived from the inner product
		\begin{equation}
			\langle f, g \rangle =2\pi \sum_{k \in \mathbb{Z}} \widehat{f}_k \overline{\widehat{g}}_k = \int_\mathbb{T} f(x) \overline{g(x)} \dd x.
		\end{equation}
		\item For $u \in C([-T,T],L^2)$, we define $\mu(u)=\frac{1}{2\pi}\|u(0)\|^2_{}$.
		
		\item For $\mu \in \R$, we define the translation $\tau_{\mu} u(t,x):=u(t,x+2\mu t)$ for $u \in C([-T,T],L^2)$.
		
		\item \( \eta \colon \mathbb{R} \to [0,1] \) denotes a smooth, nonnegative, and even function that equals 1 on the interval \([-1,1]\) and is supported in \([-2,2]\).

	\end{itemize}

	%%%%%%%%%%%%%%%%%%%%%%%%%%%
	\section{The gauge transformation}\label{sectionmethod}
	%%%%%%%%%%%%%%%%%%%%%%%%%%%
	\begin{definition}[{\cite[Definition~2.1]{Herr2006}}]\label{def:gauge_trafo}
		For $f \in L^2$, we let
		$$\mathcal{G}_0(f)(x)=\e^{-i\mathcal{I}(f)}f(x)
		$$
		where
		$$
		\mathcal{I}(f)(x)=\frac{1}{2\pi}\int_0^{2\pi}\int_{\xi}^x |f(y)|^2
		-\frac{1}{2\pi} \|f\|_{L^2}^2 \dd y \dd \xi.
		$$
		Then, for $u \in C([0,T],L^2)$, we can define the gauge transformation
		\begin{equation}\label{Eqn:gauge}
			\mathcal{G}(u)(t,x)=\tau_{-\mu(u)}\mathcal{G}_0(u(t))(x).
		\end{equation}
	\end{definition}
	Since $|f(y)|^2 - \frac{1}{2\pi} \|f\|_{}^2$ is a zero-mean function on $\mathbb{T}$, its primitive $\mathcal{I}(f)(x)$ is $2\pi$-periodic. As a consequence, the gauge-transformed function $\mathcal{G}(f)(x)= \e^{-i \mathcal{I}(f)(x)} f(x)$ is also $2\pi$-periodic.
	
	The gauge transformation introduced above plays a central role in the analysis of periodic dNLS. Its main purpose is to rewrite the derivative nonlinearity in a more favorable form, in which the most troublesome derivative interaction can be handled more effectively within the Bourgain space framework. Before deriving the transformed equation, we first record the basic mapping properties of $\mathcal{G}$, which show that the transformation is well defined, invertible, and stable in Sobolev spaces.
	\begin{lemma}[{\cite[Lemma~2.3]{Herr2006}}]\label{lem:gauge_trafo_est}
		For $s \geq 0$ the map
		$$\mathcal{G}: C([0,T],H^s)\to
		C([0,T],H^s)$$
		is a homeomorphism, and its inverse is given by
		$$
		\mathcal{G}^{-1}(v)(t,x)=\mathcal{G}^{-1}_0 ( \tau_{\mu(v)} v(t)) (x)  = \e^{i\mathcal{I}(\tau_{\mu(v)}v)}\tau_{\mu(v)} v(t,x).
		$$
		Moreover, the map $\mathcal{G}_0$ is locally Lipschitz continuous on $H^s$:
		\begin{equation}\label{eq:gauge_trafo_est}
			\|\mathcal{G}_0(f_1) -\mathcal{G}_0(f_2) \|_{H^s}\leq C
			\|f_1 - f_2 \|_{H^s}, 
		\end{equation}
		where $C$ depends on $\|f_1\|_{H^{s}}$ and $\|f_2\|_{H^{s}}$.
		Furthermore, the inverse map $\mathcal{G}_0^{-1}$ is also locally Lipschitz continuous on $H^s$.
	\end{lemma}
	
	Based on the stability and invertibility of the gauge transformation, the original problem can be reformulated in an equivalent form for the transformed variable $v=\mathcal{G}(u)$.
	
	\begin{lemma}[{\cite[Theorem~1.1]{Herr2006}}]\label{eq-gaugelemma}
		For $u_0 \in H^s$ with $s\geq \frac{1}{2}$, the dNLS equation \eqref{dnls} admits a unique solution $u\in C([0,T];H^s)$ for some $T>0$, and its transformed variable $v=\mathcal{G}(u)$ satisfies the following problem
		\begin{equation}\label{gdnls}
			\left\{
			\begin{aligned}
				&\partial_t v-i\partial^2_{x}v=F(v) \quad\text{in}~ [0, T]\times \T,\\
				&v(0)=\mathcal{G}(u_0)=v_0,
			\end{aligned}
			\right.
		\end{equation}
		where
		\begin{align}
			F(v)&=-v^2\overline{v}_x+\frac{i}{2}|v|^4v-i\mu|v|^2v+i\psi(v)v, \label{Eqn:def-Fv}\\
			\psi(v)&=\frac{1}{2\pi}\int_0^{2\pi}2\Imag (\overline{v}_x v)(t,\theta)-\frac{1}{2}|v|^4(t,\theta)\dd\theta+\mu^2 \quad \text{with}~ \mu=\frac{1}{2\pi}\|v(0)\|^2. \label{Eqn:def:psiv}
		\end{align}
	\end{lemma}

	\section{Bourgain spaces}\label{Sec:bourgain}
	Our approach relies on the well-known Bourgain spaces, which have been extensively utilized in the numerical analysis of dispersive equations \cite{Jimcom,Jisiam,Jiima,JiZhao,Ostjems,Roupaa}.
	
	Let us recall the Bourgain spaces for the Schr\"odinger operator $\p_t-i\p_x^2$. A tempered distribution $v(t,x)$ on $\R\times\T$ belongs to the Bourgain space $X^{s,b}$ if the following norm is finite:
	\begin{equation*}
		\label{eq:X_s,b}
		\|v\|^2_{X^{s,b}}= \sum_{k \in \Z}\int_\R  \lb k \rb^{2s} \lb \sigma +
		k^2\rb^{2b}  | \widetilde{v} (\sigma,k)|^2 \dd \sigma,
	\end{equation*} where $\widetilde{v}$ is the space-time Fourier transform of $v$: 
	$$
	\widetilde{v}(\sigma, k)=\int_{\R\times\T}{\rm e}^{-i\sigma t-ikx}v(t,x)\dd t\dd x.
	$$ 
	%Then we can define $X_{-}^{s,b}$ by replacing $\lb \sigma + k^2\rb$ with $\lb \sigma -k^2\rb$. 
	Furthermore, we define the space $Y^{s,b}$ by using the (stronger) $L^1$ norm for the time-frequency variable $\sigma$
	\begin{equation*}
		\label{eq:Y_s}
		\|v\|^2_{Y^{s,b}}= \sum_{k \in \Z} \left(\int_\R \lb \sigma +
		k^2\rb^{b} \lb k \rb^{s} |\widetilde{v} (\sigma,k)|\dd \sigma\right)^2,
	\end{equation*}
	and we consider the space $Z^{s,b}=X^{s,b} \cap Y^{s,b-\frac{1}{2}}$ with the norm
	$$
	\|v\|_{Z^{s,b}}= \|v\|_{X^{s,b}} + \|v\|_{Y^{s,b-\frac{1}{2}}}. 
	$$
	Next, we give a localized version of the space $Z^{s,b}$. Let $I\subset \R$ be a closed interval. We say that $v\in Z^{s,b}(I)$ if the following norm is finite:
	$$\|v\|_{Z^{s,b}(I)}=\inf\{\|v^{\star}\|_{Z^{s,b}} \mid v^{\star}|_{I}=v, \, v^{\star}\in Z^{s,b}\}.$$
	When $I=[0, T]$, we will often simply use the notation $Z^{s,b}(T)$. 
	
	Now, we start with frequently used lemmas.
	\begin{lemma}[{\cite[Lemma~3.3]{Herr2006}}]\label{sobolevcontinu}
		Let $s\in\mathbb{R}$. Then the following embeddings hold:
		\begin{align*}
			\|v\|_{L_t^p H^s}
			&\lesssim \|v\|_{X^{s,b}},
			\quad 2\le p<\infty, ~ b\ge \tfrac12-\tfrac1p,\\
			\|v\|_{L^{\infty}_t H^s}
			&\lesssim \|v\|_{Z^{s,\frac{1}{2}}}, \\
			\|v\|_{Y^{s,b}} &\lesssim \|  v\|_{X^{s,b^\prime}},\quad b^\prime>b+\tfrac{1}{2}.
		\end{align*}
	\end{lemma}

	\begin{lemma}[\cite{Herr2006,Tao2006}]\label{intoft}
		Let $s\in\R$. We have 
		\begin{align*}
			\|\eta(t){\rm e}^{it\p_x^2}f\|_{Z^{s,\frac{1}{2}}}+\|\eta(t){\rm e}^{it\p_x^2}f\|_{X^{s,1}}&\lesssim_{\eta}\|f\|_{H^s},\quad f\in H^s.
		\end{align*}
		For $\supp(F)\subset \{(t,x)\mid|t|\leq 2\}$, we have
		\begin{align*}
			\bigg\|\eta(t)\int_0^t{\rm e}^{i(t-{\xi})\p_x^2}F({\xi} )\dd {\xi} \bigg\|_{Z^{s,\frac{1}{2}}}&\lesssim_{\eta}\|F\|_{Y^{s,-1}}+\|F\|_{X^{s,-\frac{1}{2}}},\quad F\in Y^{s,-1}\cap X^{s,-\frac{1}{2}}, \\
			\bigg\|\eta(t)\int_0^t{\rm e}^{i(t-{ \xi})\p_x^2}F({\xi} )\dd {\xi} \bigg\|_{X^{s,1}}&\lesssim_{\eta}\|F\|_{X^{s,0}},\quad F\in X^{s,0}.
		\end{align*}
	\end{lemma}
	
	Next, we state the multilinear estimates.
	\begin{lemma}[{\cite[Corollary 4.6]{Herr2006}}]\label{multilinearcontinuous}
		Let $s\geq \frac{1}{2}$ and $\varepsilon_0>0$. For $T_1 \in(0,1]$ and $\supp(v_j) \subset \{(t,x) \mid
		|t|\leq 2T_1\}$, $j=1,\ldots, 5$, there exists $\varepsilon>0$ such that 
		\begin{align*}
			\|v_1 v_2 \, \partial_x \overline{v}_3 \|_{Y^{s,-1} \cap X^{s,-\frac{1}{2}}}
			&\lesssim T_1^{\eps}\sum\limits_{k=1}^3\|v_k\|_{X^{s,\frac{1}{2}}}
			\prod_{\genfrac{}{}{0pt}{}{j=1}{j\not=k}}^3\|v_j\|_{X^{\frac{1}{2},\frac{1}{2}}},	\\%\label{congen_tri}\\
			\bigg\|\overline{v}_1\overline{v}_2\prod_{j=3}^5 v_j\bigg\|_{X^{s,-\frac{3}{8}-\eps_0}}&\lesssim T_1^{\eps}
			\sum\limits_{k=1}^5\|v_k\|_{X^{s,\frac{1}{2}}}
			\prod_{\genfrac{}{}{0pt}{}{j=1}{j\not=k}}^5\|v_j\|_{X^{\frac{1}{2},\frac{1}{2}}},	\\%\label{congen_quint_pol_est}\\
			\|\overline{v}_3v_4v_5 \|_{X^{s,-\frac{3}{8}-\eps_0}}
			&\lesssim T_1^{\eps} 
			\sum\limits_{k=3}^5\|v_k\|_{X^{s,\frac{1}{2}}} \prod_{\genfrac{}{}{0pt}{}{j=3}{j\not=k}}^5\|v_j\|_{X^{\frac{1}{2},\frac{1}{2}}},  \\%\label{congen_tri_pol_est}\\
			\|(\psi(v_1)-\psi(v_2))v_3 \|_{X^{s,0}} 
			&\lesssim T_1^{\eps}\Big(1+\|v_1\|_{X^{\frac{1}{2},\frac{1}{2}}\cap Z^{0,\frac{1}{2}}}+\|v_2\|_{X^{\frac{1}{2},\frac{1}{2}}\cap Z^{0,\frac{1}{2}}}\Big)^3 \notag \\
			&\qquad\quad  \times \|v_1-v_2\|_{X^{\frac{1}{2},\frac{1}{2}}\cap Z^{0,\frac{1}{2}}}\|v_3\|_{X^{s,\frac{1}{2}}}, %\label{congen_psi_est}
		\end{align*}
		and 
		\begin{equation}\label{con_psi}
			\|\psi(v)(t)\|_{L^4_t}\leq \|v\|_{L_t^8 H^{\frac{1}{2}}}^2+\|v\|_{L_t^{16} L_x^4}^4\lesssim \|v\|_{X^{\frac{1}{2}, \frac{3}{8}}}^2+\|v\|_{X^{\frac{1}{4}, \frac{7}{16}}}^4.
		\end{equation}
	\end{lemma}
	By using Lemmas \ref{intoft} and \ref{multilinearcontinuous} for the map
	\begin{align*}
		G_{v_0}^{\delta}(w):=\eta(t){\rm e}^{it\p_x^2}v_0+\eta(t)\int_0^t{\rm e}^{i(t- {\xi})\p_x^2}\eta \bigg(\frac{\xi}{\delta}\bigg)F(w({\xi} ))\dd {\xi} ,
	\end{align*}
	we obtain the well-posedness result in $H^s$ ($s\geq\frac{1}{2}$) for the equation \eqref{gdnls}. %Finally, by Lemma \ref{eq-gaugelemma}, we  can get the well-posedness result in $H^s(s\geq \frac{1}{2})$ for the equation \eqref{dnls}. 

	\section{A filtered exponential integrator for dNLS}\label{Sec:integrator}
	The projected equation plays an important role in the numerical solution of dispersive problems; see, e.g.,~\cite{Jimcom,Jisiam,Jiima,Ostjems,Roupaa}. For the dNLS equation \eqref{gdnls}, the projected version is
	\begin{equation}\label{pjgdnls}
		\left\{
		\begin{aligned}
			&\partial_t v_\tau -i\partial^2_{x}v_\tau=F_\tau(v_\tau)  \quad\text{in } [0, T]\times \T,\\
			&v_\tau(0)=\Pi_{\tau}\mathcal{G}(u_0)=\Pi_{\tau}v_0,
		\end{aligned}
		\right.
	\end{equation}
	where
	\begin{equation}\label{pjFpsi}
		F_\tau(v) =-\Pi_\tau [(\Pi_\tau v)^2\partial_x \Pi_\tau\overline{v}]+\frac{i}{2}\Pi_\tau(|\Pi_\tau v|^4\Pi_\tau v)-i\mu \Pi_\tau(|\Pi_\tau v|^2\Pi_\tau v)+i\psi(\Pi_\tau v)\Pi_\tau v
	\end{equation}
	and $\psi$ is defined in \eqref{Eqn:def:psiv}.
	Here, the projection operator $\Pi_\tau$ for $\tau>0$ is defined by the Fourier multiplier
	\begin{equation}\label{Eqn:projection}
		\Pi_\tau = \chi \bigg(\frac{- i \partial_x}{\tau^{-\frac{1}{2}}}\bigg)  = \overline{\Pi}_\tau,
	\end{equation}
	where $\chi$ is the characteristic function $\chi= {1}_{[-1,1]}$. The projected equation \eqref{pjgdnls} also conserves mass. Taking the $L^2$-inner product of \eqref{pjgdnls} with ${v}_\tau$ and taking the real part, we obtain
	$$ 
	\frac{1}{2} \frac{\mathrm{d}}{\mathrm{d}t} \|v_\tau(t)\|_{L^2}^2 =\mathrm{Re}\, \langle \partial_t v_\tau -i\partial^2_{x}v_\tau,{v}_\tau \rangle= \mathrm{Re}\,\langle F_\tau(v_\tau) ,{v}_\tau \rangle. 
	$$
	Noting that $\Pi_\tau$ is self-adjoint and $v_\tau = \Pi_\tau v_\tau$, we have
	$$
	\begin{aligned}
		\mathrm{Re}\,\langle -\Pi_\tau [(\Pi_\tau v_\tau)^2\partial_x \Pi_\tau\overline{v}_\tau] ,{v}_\tau \rangle &= -\frac{1}{4}\int_{\T} \partial_x (|v_\tau|^4)\dd x =0, \\
		\mathrm{Re}\, \langle i\Pi_\tau(|\Pi_\tau v_\tau|^4\Pi_\tau v_\tau),{v}_\tau \rangle&= \mathrm{Re}\, \int_{\T} i|{v}_\tau|^6\dd x =0, \\
		\mathrm{Re}\,\langle i\mu \Pi_\tau(|\Pi_\tau v_\tau|^2\Pi_\tau v_\tau),{v}_\tau \rangle& = \mathrm{Re}\, \int_{\T} i\mu |v_\tau|^4\dd x =0,\\
		\mathrm{Re}\,\langle i\psi(\Pi_\tau v_\tau)\Pi_\tau v,{v}_\tau \rangle& = \mathrm{Re}\, \int_{\T} i\psi(\Pi_\tau v_\tau) |v_\tau|^2\dd x =0.
	\end{aligned}
	$$
	Therefore, we arrive at $\frac{\mathrm{d}}{\mathrm{d}t} \|v_\tau(t)\|_{L^2}^2 = 0$, which implies 
	\begin{equation}\label{Eqn:mass-vtau}
		\|v_\tau(t)\|_{L^2} = \|v_\tau(0)\|_{L^2},\quad t \in [0,T]. 
	\end{equation}

	Our filtered exponential integrator is constructed by directly applying the exponential Euler method (see \cite{Hochacta}) to the projected dNLS equation \eqref{pjgdnls}. To this end, we first write Duhamel's formula of \eqref{pjgdnls} as follows
	\begin{align*}
		v_{\tau}(t_n+\tau)={\rm e}^{i\tau\p_x^2}v_{\tau}(t_n)+\int_0^\tau{\rm e}^{i(\tau-{\xi} )\p_x^2}F_\tau(v_{\tau}(t_n+{ \xi} ))\dd {\xi} .
	\end{align*}
	Approximating $v_{\tau}(t_n+\xi)$ by $v_{\tau}(t_n)$ in the nonlinear term leads to the explicit exponential Euler integrator
	\begin{align}\label{filteredsch}
		v^{n+1}=\Phi^\tau(v^n):={\rm e}^{i\tau\p_x^2}v^n+\tau {\rm e}^{i\tau\p_x^2}\varphi_1(-i\tau\p_x^2)F_\tau(v^n), \quad v^0=\Pi_\tau v_0,
	\end{align}
	where $\varphi_1(z)=\tfrac{{\rm e}^z -1}{z}$. {This is the numerical method we will analyze.} We note that the explicit exponential Euler integrator for \eqref{pjgdnls} can also be viewed as a filtered scheme for \eqref{dnls}.  Additionally, a CFL-type condition $K^2\tau \leq 1$ arises from the projection operator $\Pi_\tau$, where $K$ is the number of Fourier modes in space.

	Note that we will analyze our integrator for the projected equation. Therefore, we need to study the regularity of \eqref{pjgdnls} and estimate the difference between \eqref{pjgdnls} and \eqref{gdnls}. To do this, we require the following two results: Propositions~\ref{diffpj} and~\ref{highregbound}. 
	
	Following the ideas presented in \cite{Jimcom}, by using Lemmas \ref{sobolevcontinu}--\ref{multilinearcontinuous}  and the straightforward estimates
	$$
	\begin{aligned}
		\| (1 - \Pi_{\tau}) f \|_{H^{\frac{1}{2}}}  & \leq \tau^{\frac{s}{2}-\frac{1}{4}}\| f \|_{H^{s_{}}}, \quad \forall f \in H^{s}, \\
		\| (1 - \Pi_{\tau}) f \|_{Z^{\frac{1}{2},b}} & \leq \tau^{\frac{s}{2}-\frac{1}{4}} \|f \|_{Z^{s,b}}, \quad \forall f \in Z^{s,b},
	\end{aligned}
	$$
	we obtain:
	\begin{proposition}\label{diffpj}
		For $v_0\in H^s$ with $s>\frac{1}{2}$, let $[0, T_{max})$ denote the maximal interval of existence of \eqref{dnls}. Then for any $T<T_{max}$, there exists $\tau_0>0$ such that for $\tau\leq \tau_0$, the dNLS equation \eqref{pjgdnls} admits a unique solution $v_\tau\in Z^{s,\frac{1}{2}}(T)\subset  C([0, T], H^s)$ satisfying 
		\begin{align*}
			\|v_\tau\|_{Z^{s,\frac{1}{2}}(T)}&\leq C_T,\\
			\|v-v_\tau\|_{L^{\infty}_t H^{\frac{1}{2}}}\lesssim \|v-v_\tau\|_{Z^{\frac{1}{2},\frac{1}{2}}(T)}& \leq C_T\tau^{\frac{s}{2}-\frac{1}{4}},
		\end{align*}
		where $C_T$ is a constant depending on $v_0$ and $T$.
	\end{proposition}
	
	The following proposition concerns the propagation of higher regularity with respect to time, which is essential for the subsequent error analysis.
	\begin{proposition}\label{highregbound}
		For $s>\frac{1}{2}$, $\tau\in(0,1]$ and every fixed $T<T_{max}$, we have the bound
		\begin{align}
			\|v_\tau\|_{X^{s,1}(T)}  \lesssim_{v_0,T} \tau^{-\frac{1}{2}}.
		\end{align}
	\end{proposition}
	\begin{proof}
		For $\delta>0$ small enough, where the smallness depends only on $\|v_\tau\|_{Z^{s,\frac{1}{2}}(T)}$, we consider the localized fixed-point problem
		%	\begin{align*}
			%		V_\tau(t)=\eta(t){\rm e}^{it\p_x^2}\Pi_{\tau}v_0+\eta(t)\int_0^t{\rm e}^{i(t-s)\p_x^2}\eta \bigg(\frac{t}{\delta}\bigg)F_{\tau}(V_{\tau}(s))\dd s,
			%	\end{align*}
		$$
		V_\tau=\eta(t){\rm e}^{it\p_x^2}\Pi_{\tau}v_0+\eta(t)\int_0^t{\rm e}^{i(t-{\xi})\p_x^2}\eta \bigg(\frac{\xi}{\delta}\bigg)F_{\tau}(V_{\tau}({\xi}))\dd \xi.
		$$
		By the local well-posedness theory, this equation admits a unique fixed point $V_\tau\in Z^{s,\frac{1}{2}}$ 
		and the corresponding solution $v_\tau$ of \eqref{pjgdnls} satisfies
		\[
		v_\tau(t)=V_\tau(t), \quad t\in[0,\delta].
		\]
		By applying Lemma~\ref{intoft}, we obtain 
		\begin{align*}
			\|V_\tau\|_{X^{s,1}}&\lesssim \|v_0\|_{H^s}+\|\Pi_\tau [(\Pi_\tau V_\tau)^2\partial_x \Pi_\tau\overline{V}_\tau]\|_{X^{s,0}}+\|\Pi_\tau(|\Pi_\tau V_\tau|^4\Pi_\tau V_\tau)\|_{X^{s,0}}\\
			&\quad+\mu\| \Pi_\tau(|\Pi_\tau V_\tau|^2\Pi_\tau V_\tau)\|_{X^{s,0}}+\|\psi(\Pi_\tau V_\tau)\Pi_\tau V_\tau\|_{X^{s,0}}.
		\end{align*}
		By applying the bilinear estimate $\|fg\|_{H^s}\lesssim \|f\|_{H^s}\|g\|_{H^s}$ and Lemma~\ref{sobolevcontinu}, we arrive at
		\begin{align*}
			& \|\Pi_\tau(|\Pi_\tau V_\tau|^4\Pi_\tau V_\tau)\|_{X^{s,0}}+\mu\| \Pi_\tau(|\Pi_\tau V_\tau|^2\Pi_\tau V_\tau)\|_{X^{s,0}}\\
			&\qquad \lesssim \|V_\tau\|_{L^\infty_t H^s}^5+\mu \|V_\tau\|_{L^\infty_t H^s}^3 
			\lesssim \|V_\tau\|_{Z^{s,\frac{1}{2}}}^5+\mu \|V_\tau\|_{Z^{s,\frac{1}{2}}}^3.
		\end{align*}
		Since $\Pi_\tau$ projects on frequencies $|k|\leq \tau^{-\frac{1}{2}}$, by using H\"older's inequality, Lemma~\ref{sobolevcontinu} and \eqref{con_psi}, we have
		\begin{align*}
			&\|\Pi_\tau [(\Pi_\tau V_\tau)^2\partial_x \Pi_\tau\overline{V}_\tau]\|_{X^{s,0}}+\|\psi(\Pi_\tau V_\tau)\Pi_\tau V_\tau\|_{X^{s,0}}\\
			&\qquad\lesssim \tau^{-\frac{1}{2}}\|\Pi_\tau V_{\tau}\|_{L_t^6 H^s}^3+\|\psi(\Pi_\tau V_\tau)\|_{L_t^4}\|\Pi_\tau V_\tau\|_{L_t^4 H^s}\\
			&\qquad\lesssim \tau^{-\frac{1}{2}}\|V_{\tau}\|_{X^{s,\frac{1}{3}}}^3+(\|V_{\tau}\|_{X^{\frac{1}{2},\frac{3}{8}}}^2+\|V_{\tau}\|_{X^{\frac{1}{4},\frac{7}{16}}}^4)\|V_{\tau}\|_{X^{s,\frac{1}{4}}}.
		\end{align*}
		Finally, we have that
		\begin{align*}
			\|V_\tau\|_{X^{s,1}}&\lesssim \|v_0\|_{H^s}+\|V_\tau\|_{Z^{s,\frac{1}{2}}}^5+ \mu\|V_\tau\|_{Z^{s,\frac{1}{2}}}^3+\tau^{-\frac{1}{2}}\|V_{\tau}\|_{X^{s,\frac{1}{3}}}^3+(\|V_{\tau}\|_{X^{\frac{1}{2},\frac{3}{8}}}^2+\|V_{\tau}\|_{X^{\frac{1}{4},\frac{7}{16}}}^4)\|V_{\tau}\|_{X^{s,\frac{1}{4}}} \notag
			\\&\lesssim \|v_0\|_{H^s}+\|V_\tau\|_{Z^{s,\frac{1}{2}}}^5+ \mu\|V_\tau\|_{Z^{s,\frac{1}{2}}}^3+\tau^{-\frac{1}{2}}\|V_{\tau}\|_{X^{s,\frac{1}{2}}}^3+(\|V_{\tau}\|_{X^{s,\frac{1}{2}}}^2+\|V_{\tau}\|_{X^{s,\frac{1}{2}}}^4)\|V_{\tau}\|_{X^{s,\frac{1}{2}}},
		\end{align*}
		which together with Proposition~\ref{diffpj} implies $\|V_\tau\|_{X^{s,1}} \leq C_T \tau^{-\frac12}$.
		Since $v_\tau$ coincides locally with the fixed point $V_\tau$, iterating this argument over finitely many subintervals covering $[0,T]$ gives
		$$
		\|v_\tau\|_{X^{s,1}(T)}\lesssim_{v_0,T} \tau^{-\frac12},
		$$
		which is the desired result.
	\end{proof}

	\section{Discrete Bourgain spaces}\label{Sec:discrete-bourgain}
	In order to perform error estimates at low regularity, we shall develop the harmonic analysis tools at the discrete level. Definitions and properties of discrete Bourgain spaces were introduced (in the context of the NLS equation) in \cite{Ostjems}. Nevertheless, as in the continuous case, we need additional results in order to handle the dNLS equation. To this end, we shall introduce the discrete counterparts $X^{s,b}_\tau$, $Y^{s,b}_\tau$ and $Z^{s,b}_\tau$ of the spaces $X^{s,b}$, $Y^{s,b}$ and $Z^{s,b}$.
	
	For sequences of functions $\{v^{n}(x)\}_{n \in \mathbb{Z}},$ we define the Fourier transform $\widetilde{v^{n}}(\sigma, k)$ by
	$$
	\mathcal{F}_{\tau,x}(v^n)(\sigma,k)=\widetilde{v^{n}} (\sigma, k)= \tau \sum_{m \in \mathbb{Z}} \widehat{v^{m}}(k) \e^{-i m \tau \sigma}, \quad \widehat{v^{m}}(k)= {1 \over 2\pi} \int_{-\pi}^\pi v^{m}(x) \e^{-i k x}\dd x.
	$$
	Parseval's identity then reads
	$$
	\| \widetilde{v^{n}}\|_{L^2l^2}= \|v^{n}\|_{l^2_{\tau}L^2},
	$$
	where
	$$
	\| \widetilde{v^{n}}\|_{L^2l^2}^2 = \int_{-{\pi \over \tau}}^{\pi\over \tau} \sum_{k \in \mathbb{Z}}
	|\widetilde{v^{n}}(\sigma, k)|^2 \dd \sigma, \quad
	\|v^{n}\|_{l^2_{\tau}L^2}^2 = \tau \sum_{m \in \mathbb{Z}} \int_{-\pi}^\pi  |v^{m}(x)|^2 \dd x.
	$$
	As in \cite{Jimcom}, we define the discrete Bourgain spaces $X^{s,b}_\tau$ for $s\ge 0$, $b\in\mathbb R$, $\tau>0$ by
	\begin{equation}\label{Xtausb}
		\| v^n \|_{X^{s,b}_{\tau}} =  \| \langle k \rangle^s \langle  d_{\tau}( \sigma + k^2)  \rangle^b \widetilde{v^n}(\sigma, k)   \|_{L^2l^2}\sim %\| \lb D_\tau\rb^b\lb \p_x\rb^s({\rm e}^{-in\tau\p_x^2} v^n)\|_{l_\tau^2 L^2}=
		\|{\rm e}^{-in\tau\p_x^2}v^n\|_{H_\tau^b H^s},
	\end{equation}
	where  $d_{\tau}(\sigma)=\frac{\e^{i \tau \sigma} - 1}\tau$. %and {\color{blue} $(D_\tau (v^n))_n=(\tfrac{v^{n+1}-v^n}{\tau})_n$}.
	Note that $d_{\tau}$ is $2\pi/\tau$ periodic and that uniformly in $\tau$, we have $|d_{\tau}(\sigma)| \sim | \sigma |$ for $|\tau \sigma | \leq \pi$. 
	We define $X_{\tau,-}^{s,b}$ by replacing $\langle d_{\tau}( \sigma + k^2)  \rangle$ with $\langle d_{\tau}( \sigma - k^2)  \rangle$ and observe that 
	\begin{equation}\label{Eqn:equivalent-conjugate}
		\|\overline{v}^n\|_{X_\tau^{s,b}} = \|v^n\|_{X_{\tau,-}^{s,b}}.
	\end{equation}
	This relation will be used later. Furthermore, we define the space $Y_{\tau}^{s,b}$ with respect to the norm
	\begin{equation}\label{Ytausb}
		\|v^n\|_{Y_\tau^{s,b}}^2 =\|\lb k\rb^s\lb d_\tau(\sigma+k^2)\rb^b \widetilde{v^n}(\sigma, k) \|^2_{L^1 l^2} =
		\sum_{k\in \Z}\left(\int_{-\frac{\pi}{\tau}}^{\frac{\pi}{\tau}} \lb k\rb^s\lb d_\tau(\sigma+k^2)\rb^b |\widetilde{v^n}(\sigma, k)| \dd \sigma \right)^2
	\end{equation}
	and the space $Z_{\tau}^{s,b} =X_{\tau}^{s,b} \cap Y_{\tau}^{s,b-\frac{1}{2}}$ with the norm
	\begin{equation*}
		\|v^n\|_{Z_\tau^{s,b}} = \|v^n\|_{X_\tau^{s,b}} + \|v^n\|_{Y_\tau^{s,b-\frac{1}{2}}}.
	\end{equation*}
	We can directly obtain from the definition that for $s\geq s^{\prime}$ and $b\geq b^{\prime}$, 
	\begin{align}
		\|v^n\|_{X^{s,b}_\tau} &\lesssim \tau^{b^{\prime}-b}\|v^n\|_{X^{s,b^\prime}_\tau},\label{normdecr-b} \\
		\|\Pi_\tau v^n\|_{X^{s,b}_\tau} &\lesssim \tau^{\frac{s^{\prime}-s}{2}}\|\Pi_\tau v^n\|_{X^{s^\prime,b}_\tau}.\label{normdecr-s} 
	\end{align}
	
	Some useful technical properties are gathered in the following lemma.
	\begin{lemma}\label{discretelinearest}
		For $s\in \mathbb R$ and $\tau\in(0,1]$, we have
		\begin{align}
			\|\eta(\tfrac{n\tau}{T})v^n\|_{X^{s,b^\prime}_\tau}&\lesssim_{\eta, b, b^{\prime}} T^{b-b^\prime}\|v^n\|_{X^{s,b}_\tau}, \quad -\tfrac{1}{2}<b^\prime\leq b<\tfrac{1}{2},\, 0<T=N\tau\leq 1, \, N\geq 1,\label{Eqn:T-inverse}\\
			\sup\limits_{\delta\in[-4,4]}\|{\rm e}^{i\tau\delta\p_x^2}v^n\|_{X^{s,b}_\tau}&\lesssim \|v^n\|_{X^{s,b}_\tau},\quad \sup\limits_{\delta\in[-4,4]}\|{\rm e}^{i\tau\delta\p_x^2}v^n\|_{Z^{s,b}_\tau}\lesssim \|v^n\|_{Z^{s,b}_\tau}, \quad b\in\R. \label{perturb}
		\end{align}
		For sequences $\{v^n\}_{n\in\mathbb{Z}}$ supported on the time grid $n\tau \in [-2, 2]$, we have
		\begin{align}\label{Eqn:discrete-convolution}
			%\|\eta(t){\rm e}^{it\p_x^2}f\|_{Z_\tau^{s,\frac{1}{2}}\cap X_\tau^{s,1}}&\lesssim_{\eta}\|f\|_{H^s},\\
			\|V^n\|_{Z_\tau^{s,\frac{1}{2}}}&\lesssim_{\eta,b}\|v^n\|_{Z_\tau^{s,-\frac{1}{2}}},\quad v^n \in Z_\tau^{s,-\frac{1}{2}},
		\end{align}
		where 
		$$
		V^n(x)= \eta(n\tau)\tau\sum\limits_{m=0}^{n}{\rm e}^{i(n-m)\tau\p_x^2}v^m(x).
		$$
	\end{lemma}
	
	Inequality \eqref{Eqn:T-inverse} is given in \cite[Lemma~3.4]{Ostjems}. The first inequality in \eqref{perturb} is given in \cite[Remark~3.2]{Ostjems} and the second one can be similarly derived. Inequality \eqref{Eqn:discrete-convolution} can be verified following the ideas present in the proof of \cite[Lemma 3.2]{Roupaa}.

	The following inequalities from \cite[Lemma~3.6]{Ostjems} are frequently used in our analysis:
	\begin{align}
		\|\Pi_\tau v^n\|_{l_\tau^4L^4}& \lesssim \|v^n\|_{X_{\tau}^{0,\frac{3}{8}}}, \label{Eqn:discreteL4bound}\\
		\|\Pi_\tau v^n\|_{X_{\tau}^{0,-\frac{3}{8}}} &\lesssim \|v^n\|_{l_\tau^{\frac{4}{3}}L^{\frac{4}{3}}}.\label{Eqn:discreteX-38bound} 
	\end{align}
	
	Next, we  state some multilinear estimates that will be used in the proof of convergence. For notational simplicity, we shall use the convention that
	$$
	\Pi_\tau f \Pi_\tau g: =  (\Pi_\tau f)( \Pi_\tau g).
	$$
	\begin{theorem}\label{multilineardisc}
		Let $s\geq \frac{1}{2}$. For all sequences $\{v^n_j\}_{n \in \mathbb{Z}}$ $(j = 1,\ldots,5)$ supported on the time grid $n\tau \in [-2T_1, 2T_1]$ for $T_1  \in (0,1]$, there exists $\varepsilon >0$ such that the  following estimates hold:
		\begin{align}
			\|\Pi_{\tau}(\Pi_{\tau}v_{1}^n \Pi_{\tau}v_{2}^n \, \partial_x \Pi_{\tau}\overline{v}_{3}^n)\|_{Z_\tau^{s,-\frac{1}{2}}}
			&\lesssim  T_1^{\varepsilon} \sum\limits_{k=1}^3\|v^n_{k}\|_{X_\tau^{s,\frac{1}{2}}}\prod_{\genfrac{}{}{0pt}{}{j=1}{j\not=k}}^3\|v_j^{n}\|_{X_\tau^{\frac{1}{2},\frac{1}{2}}},	\label{eq:gen_tri}\\ 
			\Bigg\|\Pi_\tau \Bigg(\Pi_{\tau}\overline{v}_1^{n}\Pi_{\tau}\overline{v}_2^{n}\prod_{j=3}^5 \Pi_{\tau}v_j^{n}\Bigg)\Bigg\|_{X_\tau^{s,-\frac{3}{8}}}&\lesssim T_1^{\varepsilon}\sum\limits_{k=1}^5\|v_k^{n}\|_{X_\tau^{s,\frac{1}{2}}}\prod_{\genfrac{}{}{0pt}{}{j=1}{j\not=k}}^5\|v_j^{n}\|_{X_\tau^{\frac{1}{2},\frac{1}{2}}},	\label{eq:gen_quint_pol_est} \\
			\|\Pi_\tau (\Pi_{\tau}\overline{v}_3^{n}\Pi_{\tau}v_4^{n}\Pi_{\tau}v_5^{n})\|_{X_\tau^{s,-\frac{3}{8}}}&\lesssim 
			T_1^{\varepsilon}\sum\limits_{k=3}^5\|v_k^{n}\|_{X_\tau^{s,\frac{1}{2}}}\prod_{\genfrac{}{}{0pt}{}{j=3}{j\not=k}}^5\|v_j^{n}\|_{X_\tau^{\frac{1}{2},\frac{1}{2}}},\label{eq:gen_tri_pol_est} \\
			\|(\psi(\Pi_{\tau}v_1^{n})-\psi(\Pi_{\tau}v_2^{n}))\Pi_{\tau}v_3^{n}\|_{X_\tau^{s,0}}
			&\lesssim T_1^{\varepsilon}\Big(1+\|v_1^{n}\|_{X_\tau^{\frac{1}{2},\frac{1}{2}}\cap Z_\tau^{0,\frac{1}{2}}}+\|v_2^{n}\|_{X_\tau^{\frac{1}{2},\frac{1}{2}}\cap Z_\tau^{0,\frac{1}{2}}}\Big)^3\notag\\
			&\qquad\quad \times \|v_1^{n}-v_2^{n}\|_{X_\tau^{\frac{1}{2},\frac{1}{2}}\cap Z_\tau^{0,\frac{1}{2}}}\|v_3^{n}\|_{X_\tau^{s,\frac{1}{2}}}.\label{eq:gen_psi_est}
		\end{align}
	\end{theorem}
	\begin{proof}
		The estimate \eqref{eq:gen_tri} with $s=\frac{1}{2}$ follows from Theorems~\ref{Thm:multiplier1} and \ref{Thm:multiplier1-Y}. Its extension to the case $s\geq\frac{1}{2}$ can be established as below. Noting that $\langle k \rangle^{s-\frac{1}{2}} \lesssim \sum_{j=1}^3 \langle k_j \rangle^{s-\frac{1}{2}}$ for $k=k_1+k_2+k_3$, we have 
		\begin{align}
			&\|\Pi_\tau (\Pi_\tau v_1^n \Pi_\tau v_2^n \partial_x\Pi_\tau \overline{v}_3^n)\|_{X_\tau^{s,-\frac{1}{2}}}    \notag\\
			&\qquad = \big\|  \langle k \rangle^{\frac{1}{2}}\langle k \rangle^{s-\frac{1}{2}} \langle d_\tau(\sigma+k^2)\rangle^{-\frac{1}{2}}  \mathcal{F}_{\tau,x}\big[ \Pi_\tau (\Pi_\tau v_1^n \Pi_\tau v_2^n \partial_x\Pi_\tau \overline{v}_3^n) \big]\big\|_{L^2l^2} \notag\\
			&\qquad\lesssim   \big\| \langle k \rangle^{1\over2} \langle d_\tau(\sigma+k^2)\rangle^{-\frac{1}{2}} \mathcal{F}_{\tau,x}\big[ \Pi_\tau \big(  (\langle \partial_x \rangle^{s-\frac{1}{2}}\Pi_\tau v_1^n ) \Pi_\tau v_2^n \partial_x\Pi_\tau \overline{v}_3^n  \big) \big] \big\|_{L^2l^2}\notag\\
			&\qquad\quad+\big\|  \langle k \rangle^{1\over2} \langle d_\tau(\sigma+k^2)\rangle^{-\frac{1}{2}} \mathcal{F}_{\tau,x}\big[ \Pi_\tau \big (\Pi_\tau v_1^n (\langle \partial_x \rangle^{s-\frac{1}{2}} \Pi_\tau v_2^n) \partial_x\Pi_\tau \overline{v}_3^n  \big) \big]  \big\|_{L^2l^2}\notag\\
			&\qquad\quad+\big\| \langle k \rangle^{1\over2} \langle d_\tau(\sigma+k^2)\rangle^{-\frac{1}{2}} \mathcal{F}_{\tau,x}\big[  \Pi_\tau \big(\Pi_\tau v_1^n \Pi_\tau v_2^n (\partial_x \langle \partial_x \rangle^{s-\frac{1}{2}} \Pi_\tau \overline{v}_3^n )\big) \big] \big\|_{L^2l^2}\notag \\
			&\qquad = \big\|\Pi_\tau \big(  (\langle \partial_x\rangle^{s-\frac{1}{2}} \Pi_\tau v_1^n )\Pi_\tau v_2^n \partial_x\Pi_\tau \overline{v}_3^n\big)\big\|_{X_\tau^{\frac{1}{2},-\frac{1}{2}}}+\big\|\Pi_\tau \big(\Pi_\tau v_1^n (\langle \partial_x\rangle^{s-\frac{1}{2}}\Pi_\tau v_2^n ) \partial_x\Pi_\tau \overline{v}_3^n\big)\big\|_{X_\tau^{\frac{1}{2},-\frac{1}{2}}} \notag\\
			&\qquad\quad+\big\|\Pi_\tau \big(\Pi_\tau v_1^n \Pi_\tau v_2^n (\langle \partial_x\rangle^{s-\frac{1}{2}}\partial_x\Pi_\tau \overline{v}_3^n) \big)\big\|_{X_\tau^{\frac{1}{2},-\frac{1}{2}}}.\notag
		\end{align}
		Applying \eqref{eq:gen_tri} with $s=\frac{1}{2}$ to each term on the right-hand side gives the desired result.
		
		For proving \eqref{eq:gen_quint_pol_est}, we focus on $s=\frac{1}{2}$ as the extension $s\geq \frac{1}{2}$ is immediate. The proof is obtained by showing that
		$$
		\Bigg\|\Pi_\tau \Bigg(\Pi_{\tau}{v}_1^{n}\Pi_{\tau}{v}_2^{n}\prod_{j=3}^5 \Pi_{\tau}v_j^{n}\Bigg)\Bigg\|_{X_\tau^{\frac{1}{2},-\frac{3}{8}}}\lesssim T_1^{\varepsilon} \|\overline{v}_1^n\|_{X_\tau^{\frac{1}{2},\frac{1}{2}}}\|\overline{v}_2^n\|_{X_\tau^{\frac{1}{2},\frac{1}{2}}}
		\prod_{j=3}^5 \|{v}_j^n\|_{X_\tau^{\frac{1}{2},\frac{1}{2}}}.
		$$
		By noting that $\langle k \rangle^{\frac{1}{2}}\lesssim \sum_{l=1}^5 \langle k_l \rangle^{\frac{1}{2}}$ for $k=\sum_{l=1}^5 k_l$, and using \eqref{Eqn:discreteX-38bound}, H\"older's inequality, \eqref{eq:sob2dis} and \eqref{Eqn:T-inverse}, we have
		\begin{align}
			\Bigg\|\Pi_\tau \Bigg( \Pi_{\tau}{v}_1^{n}\Pi_{\tau}{v}_2^{n}\prod_{j=3}^5 \Pi_{\tau}v_j^{n} \Bigg)\Bigg\|_{X_\tau^{\frac{1}{2},-\frac{3}{8}}} & \lesssim \sum_{l=1}^5 \Bigg\| \Pi_\tau \Bigg( \mathcal{F}_{\tau,x}^{-1} (\langle k_l \rangle^{\frac{1}{2}} |\widetilde{\Pi_\tau v_l^n}| ) \prod_{\genfrac{}{}{0pt}{}{j=1}{j\not=l}}^5   \mathcal{F}_{\tau,x}^{-1} ( |\widetilde{\Pi_\tau v_j^n}| )\Bigg) \Bigg\|_{X_\tau^{0,-\frac{3}{8}}} \notag\\
			& \lesssim \sum_{l=1}^5 \Bigg\|\Pi_\tau\bigg( \mathcal{F}_{\tau,x}^{-1} (\langle k_l \rangle^{\frac{1}{2}} |\widetilde{\Pi_\tau v_l^n}| ) \prod_{\genfrac{}{}{0pt}{}{j=1}{j\not=l}}^5   \mathcal{F}_{\tau,x}^{-1} ( |\widetilde{\Pi_\tau v_j^n}| )\bigg) \Bigg\|_{l_\tau^{\frac{4}{3}}L^{\frac{4}{3}}} \notag\\
			&\lesssim \sum_{l=1}^5  \|\mathcal{F}_{\tau,x}^{-1} (\langle k_l \rangle^{\frac{1}{2}} |\widetilde{\Pi_\tau v_l^n}| )\|_{l_\tau^2L^2}   \prod_{\genfrac{}{}{0pt}{}{j=1}{j\not=l}}^5  \| \mathcal{F}_{\tau,x}^{-1} ( |\widetilde{\Pi_\tau v_j^n}| ) \|_{l_\tau^{16}L^{16}}.\notag
		\end{align}
		It suffices to prove the estimate for the case $l=1$, since the other cases can be treated similarly. By \eqref{eq:sob2dis} and \eqref{Eqn:T-inverse}, we have
		\begin{align}
			\|\mathcal{F}_{\tau,x}^{-1} (\langle k_l \rangle^{\frac{1}{2}} |\widetilde{\Pi_\tau v_l^n}| )\|_{l_\tau^2L^2}   \prod_{j=2}^5  \| \mathcal{F}_{\tau,x}^{-1} ( |\widetilde{\Pi_\tau v_j^n}| ) \|_{l_\tau^{16}L^{16}} 
			&\lesssim  \|\overline{v}_1^n\|_{X_\tau^{\frac{1}{2},0}}\|\overline{v}_2^n\|_{X_\tau^{\frac{15}{32},\frac{15}{32}}} \prod_{j=3}^5 \|{v}_j^n\|_{X_\tau^{\frac{15}{32},\frac{15}{32}}} \notag \\
			&\lesssim T_1^{\varepsilon} \|\overline{v}_1^n\|_{X_\tau^{\frac{1}{2},\frac{1}{2}}}\|\overline{v}_2^n\|_{X_\tau^{\frac{1}{2},\frac{1}{2}}}
			\prod_{j=3}^5 \|{v}_j^n\|_{X_\tau^{\frac{1}{2},\frac{1}{2}}}.\notag
		\end{align}
		The estimate \eqref{eq:gen_tri_pol_est} can be derived in a similar manner.
		
		For the proof of \eqref{eq:gen_psi_est}, by recalling from \cite[Eq.~(2.33)]{Herr2006} that
		\begin{equation}\label{Eeqn:recall-psi}
			\begin{aligned}
				|\psi(u)(t) - \psi(v)(t)| &\lesssim   ( 1 + \|u(t)\|_{H^{\frac{1}{2}}} + \|v(t)\|_{H^{\frac{1}{2}}} )^3 \|u(t)-v(t)\|_{H^{\frac{1}{2}}} \\
				&\quad + \big|  \|u(0)\|_{L^2}^4 - \|v(0)\|_{L^2}^4  \big| \\
				&\lesssim   ( 1 + \|u(t)\|_{H^{\frac{1}{2}}} + \|v(t)\|_{H^{\frac{1}{2}}} )^3 \|u(t)-v(t)\|_{H^{\frac{1}{2}}} \\
				&\quad + ( \|u(0)\|_{L^2}^3 + \|v(0)\|_{L^2}^3 ) \|u(0) - v(0)\|_{L^2},
			\end{aligned}
		\end{equation}
		we obtain
		\begin{align}
			\|\psi(u^n) - \psi(v^n)\|_{l^4_\tau} &\lesssim  \Big\| \Big( 1 + \|u^n\|_{H^{\frac{1}{2}}} + \|v^n\|_{H^{\frac{1}{2}}} \Big)^3\Big\|_{l_\tau^8} \| u^n-v^n\|_{l_\tau^8 H^{\frac{1}{2}}} \notag \\
			&\quad + ( \|u(0)\|_{L^2}^3 + \|v(0)\|_{L^2}^3 ) \|u(0) - v(0)\|_{L^2} \notag \\
			&\lesssim  \Big(1  + \|u^n\|_{l_\tau^{24} H^{\frac{1}{2}}} + \|v^n\|_{l_\tau^{24} H^{\frac{1}{2}}} \Big)^3 \| u^n-v^n\|_{l_\tau^8 H^{\frac{1}{2}}}\notag \\
			&\quad + \Big( \|u^n\|_{Z_\tau^{0,\frac{1}{2}}}^3 + \|v^n\|_{Z_\tau^{0,\frac{1}{2}}}^3 \Big) \|u^n - v^n\|_{Z_\tau^{0,\frac{1}{2}}}.\notag
		\end{align}
		Using \eqref{eq:sobdis} and the embedding \eqref{eq:linfty_embdis}, we arrive at 
		\begin{equation}\label{Eqn:psil4}
			\|\psi(u^n) - \psi(v^n)\|_{l^4_\tau} \lesssim \Big(1+\|u^{n}\|_{X_\tau^{\frac{1}{2},\frac{1}{2}}\cap Z_\tau^{0,\frac{1}{2}}}+\|v^{n}\|_{X_\tau^{\frac{1}{2},\frac{1}{2}}\cap Z_\tau^{0,\frac{1}{2}}}\Big)^3  \|u^{n}-v^{n}\|_{X_\tau^{\frac{1}{2},\frac{1}{2}}\cap Z_\tau^{0,\frac{1}{2}}}.
		\end{equation}
		Finally, by H\"older's inequality and \eqref{eq:sobdis}, we obtain
		$$
		\begin{aligned}
			&\|(\psi(\Pi_{\tau}v_1^{n})-\psi(\Pi_{\tau}v_2^{n}))\Pi_{\tau}v_3^{n}\|_{X_\tau^{s,0}} 
			\lesssim \|\psi(\Pi_{\tau}v_1^{n})-\psi(\Pi_{\tau}v_2^{n})\|_{l_\tau^4} \|\Pi_{\tau}v_3^{n}\|_{l_\tau^4H^s}  \\
			&\qquad \lesssim   \|\psi(\Pi_{\tau}v_1^{n})-\psi(\Pi_{\tau}v_2^{n})\|_{l_\tau^4} \|v_3^{n}\|_{X_\tau^{s,\frac{3}{8}}} \lesssim  T_1^{\varepsilon} \|\psi(\Pi_{\tau}v_1^{n})-\psi(\Pi_{\tau}v_2^{n})\|_{l_\tau^4} \|v_3^{n}\|_{X_\tau^{s,\frac{1}{2}}},
		\end{aligned}
		$$
		where the last inequality follows by applying \eqref{Eqn:T-inverse}, using the fact that each $\{v_j^n\}_{n\in\Z}$ is supported on the time grid $n\tau\in[-2T_1,2T_1]$.
		Plugging \eqref{Eqn:psil4} into the above inequality completes the proof.
	\end{proof}
	
	The following lemma establishes estimates of the exact solution in discrete Bourgain spaces. The proof can be obtained by following the ideas in \cite[Lemma~3.4]{Roupaa}, making use of Proposition~\ref{highregbound}.
	
	\begin{lemma}\label{lemmadisc-cont}
		Denote $v_{\tau}$ as an extension of the solution of the projected dNLS \eqref{pjgdnls}. Then for $s>\frac{1}{2}$ and every $\tau \in (0, 1]$, 
		we have the estimate
		\begin{align}\label{bounddis}
			\sup_{\theta \in [-4 \tau, 4 \tau]}   \|v_{\tau}(t_{n}+ \theta)  \|_{X^{s, {1 \over 2}}_{\tau}} + \sup_{\theta \in [-4 \tau, 4 \tau]}   \|v_{\tau}(t_{n}+ \theta)  \|_{Y^{s,0}_{\tau}} 
			\lesssim_T 1.\end{align}
		
	\end{lemma}

	%%%%%%%%%%%%%%%%%%%%%%%%%%%
	\section{Local error analysis}\label{sectionlocal}
	%%%%%%%%%%%%%%%%%%%%%%%%%%%
	
	In this section, we examine the local error of the explicit exponential Euler method~\eqref{filteredsch} applied to the system~\eqref{pjgdnls}. The local error is given by	
	\begin{align}
		\mathcal{E}_{loc}(t_{n+1})&=v_\tau(t_n+\tau)-\Phi^\tau(v_\tau(t_n))\nonumber\\
		&=\fe^{i\tau\p_x^2}\int_0^\tau\fe^{-i\zeta\p_x^2} \Big[F_\tau(v_\tau(t_n+\zeta))-F_\tau(v_\tau(t_n))\Big] \dd\zeta\nonumber\\
		&=-\fe^{i\tau\p_x^2}\int_0^\tau\fe^{-i\zeta\p_x^2}\Pi_{\tau}\Big[(\Pi_\tau v_\tau(t_n+\zeta))^2\p_x\Pi_\tau\overline{v}_\tau(t_n+\zeta)-(\Pi_\tau v_\tau(t_n))^2\p_x\Pi_\tau\overline{v}_\tau(t_n)\Big]\dd\zeta\nonumber\\
		&\quad +\frac{i}{2}\fe^{i\tau\p_x^2}\int_0^\tau\fe^{-i\zeta\p_x^2}\Pi_{\tau} \Big[|\Pi_\tau v_\tau(t_n+\zeta)|^4\Pi_\tau v_\tau(t_n+\zeta)-|\Pi_\tau v_\tau(t_n)|^4\Pi_\tau v_\tau(t_n)\Big]\dd\zeta\nonumber\\
		&\quad -i\mu\fe^{i\tau\p_x^2}\int_0^\tau\fe^{-i\zeta\p_x^2}\Pi_{\tau}\Big[|\Pi_\tau v_\tau(t_n+\zeta)|^2\Pi_\tau v_\tau(t_n+\zeta)-|\Pi_\tau v_\tau(t_n)|^2\Pi_\tau v_\tau(t_n)\Big]\dd\zeta\nonumber\\
		&\quad +i\fe^{i\tau\p_x^2}\int_0^\tau\fe^{-i\zeta\p_x^2}\Big[\psi(\Pi_\tau v_\tau(t_n+\zeta))\Pi_\tau v_\tau(t_n+\zeta)-\psi(\Pi_\tau v_\tau(t_n))\Pi_\tau v_\tau(t_n)\Big]\dd\zeta\nonumber\\
		&=:\fe^{i\tau\p_x^2}(\mathcal{E}_1(t_{n+1})+\mathcal{E}_2(t_{n+1})+\mathcal{E}_3(t_{n+1})+\mathcal{E}_4(t_{n+1})).
		\label{vloc}
	\end{align}
	The following lemma presents an estimate for $\mathcal{E}_{loc}$. In the proof, we frequently use the embedding
	\begin{equation}\label{eq:XtautoZtao}
		\|v^n\|_{Z_\tau^{\frac{1}{2},-\frac{1}{2}}} = \|v^n\|_{X_\tau^{\frac{1}{2},-\frac{1}{2}}}+\|v^n\|_{Y_\tau^{\frac{1}{2},-1}}
		\lesssim  \|v^n\|_{X_\tau^{\frac{1}{2},-\frac{1}{2}+\varepsilon}},\quad\varepsilon>0,
	\end{equation}
	where the last inequality follows from \eqref{eq:XtautoYtao}.
	
	\begin{lemma}\label{Lem:local-error}
		For $s\in (\frac{1}{2},\frac{5}{2}]$ and $\tau,T_1\in(0,1]$, we have 
		$$
		\tau^{-1}\|\eta(\tfrac{t_{n+1}}{T_1})\mathcal{E}_{{loc}}(t_{n+1})\|_{Z_{\tau}^{\frac{1}{2},-\frac{1}{2}}} \leq C_{T} \tau^{\tfrac{s}{2}-\frac{1}{4}},
		$$
		where $C_T$ is a constant depending on $v_0$ and $T$.
	\end{lemma}
	\begin{proof}
		By the triangle inequality and \eqref{perturb}, we have
		\begin{equation*}%\label{Eqn:Eloc-E1234}
			\tau^{-1}\|\eta(\tfrac{t_{n+1}}{T_1})\mathcal{E}_{loc}(t_{n+1})\|_{Z_\tau^{\frac{1}{2},-\frac{1}{2}}} \lesssim \tau^{-1}\sum_{i=1}^4 \|\eta(\tfrac{t_{n+1}}{T_1})\mathcal{E}_i(t_{n+1})\|_{Z_\tau^{\frac{1}{2},-\frac{1}{2}}}.
		\end{equation*}
		Then it suffices to estimate the four terms on the right-hand side  separately.

		\medskip \noindent  {(a)~Estimate of} $\tau^{-1}\|\eta(\frac{t_{n+1}}{T_1})\mathcal{E}_1(t_{n+1})\|_{Z_\tau^{\frac{1}{2},-\frac{1}{2}}}$.

		We rewrite $\mathcal{E}_1$ as
		\begin{align}
			\mathcal{E}_1(t_{n+1})&= -\int_0^\tau\fe^{-i\zeta\p_x^2}\Pi_{\tau}\Big[(\Pi_\tau v_\tau(t_n+\zeta) - \Pi_\tau v_\tau(t_n) )
			\Pi_\tau v_\tau(t_n+\zeta)\p_x\Pi_\tau\overline{v}_\tau(t_n+\zeta) \Big]\dd\zeta \notag \\
			&\quad -\int_0^\tau\fe^{-i\zeta\p_x^2}\Pi_{\tau}\Big[\Pi_\tau v_\tau(t_n)(\Pi_\tau v_\tau(t_n+\zeta) - \Pi_\tau v_\tau(t_n) )
			\p_x\Pi_\tau\overline{v}_\tau(t_n+\zeta) \Big]\dd\zeta \notag \\
			&\quad - \int_0^\tau\fe^{-i\zeta\p_x^2}\Pi_{\tau}\Big[(\Pi_\tau v_\tau(t_n))^2 \p_x\Pi_\tau( \overline{v}_\tau(t_n+\zeta)-\overline{v}_\tau(t_n) ) \Big]\dd\zeta \notag \\
			& =: \mathcal{E}_{1,1}(t_{n+1})+ \mathcal{E}_{1,2}(t_{n+1})+ \mathcal{E}_{1,3}(t_{n+1}).\notag
		\end{align}
		With the help of Duhamel's formula, we have
		$$
		v_\tau(t_n+\zeta)-v_\tau(t_n) =(\fe^{i\zeta\p_x^2}-1)v_\tau(t_n)+\int_0^{\zeta}\fe^{i(\zeta-\xi)\p_x^2}F_\tau(v_\tau(t_n+\xi))\dd\xi.
		$$
		Then $\mathcal{E}_{1,1}$ can be decomposed as
		\begin{align}
			&\mathcal{E}_{1,1}(t_{n+1}) \notag\\
			&=-\int_0^\tau\fe^{-i\zeta\p_x^2}\Pi_{\tau}\bigg[ \Pi_\tau(\fe^{i\zeta\p_x^2}-1)v_\tau(t_n)
			\Pi_\tau v_\tau(t_n+\zeta)\p_x\Pi_\tau\overline{v}_\tau(t_n+\zeta) \bigg]\dd\zeta \notag\\
			&~+\int_0^\tau\fe^{-i\zeta\p_x^2}\Pi_{\tau}\bigg[ \Pi_\tau \int_0^{\zeta}\fe^{i(\zeta-\xi)\p_x^2}(\Pi_\tau v_\tau(t_n+\xi) )^2\partial_x \Pi_\tau\overline{v}_\tau(t_n+\xi) \dd\xi \, 
			\Pi_\tau v_\tau(t_n+\zeta)\p_x\Pi_\tau\overline{v}_\tau(t_n+\zeta) \bigg]\dd\zeta \notag \\
			&~-\int_0^\tau\fe^{-i\zeta\p_x^2}\Pi_{\tau}\bigg[ \Pi_\tau \int_0^{\zeta}\fe^{i(\zeta-\xi)\p_x^2}\frac{i}{2}|\Pi_\tau v_\tau(t_n+\xi)|^4\Pi_\tau v_\tau(t_n+\xi)\dd\xi \, 
			\Pi_\tau v_\tau(t_n+\zeta)\p_x\Pi_\tau\overline{v}_\tau(t_n+\zeta) \bigg]\dd\zeta \notag \\
			&~+\int_0^\tau\fe^{-i\zeta\p_x^2}\Pi_{\tau}\bigg[ \Pi_\tau \int_0^{\zeta}\fe^{i(\zeta-\xi)\p_x^2}i\mu  |\Pi_\tau v_\tau(t_n+\xi)|^2\Pi_\tau v_\tau(t_n+\xi)\dd\xi \, 
			\Pi_\tau v_\tau(t_n+\zeta)\p_x\Pi_\tau\overline{v}_\tau(t_n+\zeta) \bigg]\dd\zeta \notag \\
			&~-\int_0^\tau\fe^{-i\zeta\p_x^2}\Pi_{\tau}\bigg[ \Pi_\tau \int_0^{\zeta}\fe^{i(\zeta-\xi)\p_x^2}i\psi(\Pi_\tau v_\tau(t_n+\xi))\Pi_\tau v_\tau(t_n+\xi)\dd\xi \, 
			\Pi_\tau v_\tau(t_n+\zeta)\p_x\Pi_\tau\overline{v}_\tau(t_n+\zeta) \bigg]\dd\zeta \notag \\
			&=: \mathcal{E}_{1,1}^{(1)}(t_{n+1})+\mathcal{E}_{1,1}^{(2)}(t_{n+1})+\mathcal{E}_{1,1}^{(3)}(t_{n+1})+\mathcal{E}_{1,1}^{(4)}(t_{n+1})+\mathcal{E}_{1,1}^{(5)}(t_{n+1}). \label{Eqn:E11-12345}
		\end{align}
		
		Recall that $\eta$ is a smooth and compactly supported function, which is one on $[-1,1]$ and supported in $[-2,2]$. For $\tau,T_1\in (0,1]$, by \eqref{perturb} and \eqref{eq:gen_tri}, we derive that
		$$
		\begin{aligned}
			\tau^{-1} \|\eta(\tfrac{t_{n+1}}{T_1})\mathcal{E}_{1,1}^{(1)}(t_{n+1})\|_{Z_\tau^{\frac{1}{2},-\frac{1}{2}}} & \lesssim \sup_{\zeta\in[0,\tau]} \Big( \|(\fe^{i\zeta\p_x^2}-1)\Pi_\tau v_\tau(t_n)\|_{X_\tau^{\frac{1}{2},\frac{1}{2}}}  \|
			\Pi_\tau v_\tau(t_n+\zeta)\|_{X_\tau^{\frac{1}{2},\frac{1}{2}}}^2 \Big).
		\end{aligned}
		$$
		Since both $\fe^{i\zeta\p_x^2}-1$ and $\Pi_\tau$ are Fourier multipliers for the spatial variable, by noting that $|\mathrm{e}^{i\theta} - 1| \lesssim |\theta|^\alpha$ for $\alpha\in[0,1]$, we have
		$$
		\begin{aligned}
			\tau^{-1} \|\eta(\tfrac{t_{n+1}}{T_1})\mathcal{E}_{1,1}^{(1)}(t_{n+1})\|_{Z_\tau^{\frac{1}{2},-\frac{1}{2}}} &   \lesssim 
			\|(\fe^{i\zeta\p_x^2}-1)\Pi_\tau v_\tau(t_n)\|_{X_\tau^{\frac{1}{2},\frac{1}{2}}}  \sup_{\zeta\in[0,\tau]}\|v_\tau(t_n+\zeta)\|_{X_\tau^{\frac{1}{2},\frac{1}{2}}}^2 \\
			&  \lesssim \tau^{\frac{s}{2}-\frac{1}{4}} \|v_\tau(t_n)\|_{X_\tau^{s,\frac{1}{2}}}\sup_{\zeta\in[0,\tau]}\|
			v_\tau(t_n+\zeta)\|_{X_\tau^{\frac{1}{2},\frac{1}{2}}}^2\leq C_T \tau^{\frac{s}{2}-\frac{1}{4}} ,
		\end{aligned}
		$$
		where we chose $\alpha = \frac{s}2-\frac14$ and where \eqref{bounddis} was used in the last inequality.
		
		By using \eqref{perturb} and \eqref{eq:gen_tri}, we first bound $\mathcal{E}_{1,1}^{(2)}$ by
		\begin{align}
			&\tau^{-1} \|\eta(\tfrac{t_{n+1}}{T_1})\mathcal{E}_{1,1}^{(2)}(t_{n+1})\|_{Z_\tau^{\frac{1}{2},-\frac{1}{2}}}\notag  \\
			&\qquad\lesssim
			\sup_{\zeta\in[0,\tau]}\bigg( \bigg\|\Pi_\tau \int_0^{\zeta}\fe^{i(\zeta-\xi)\p_x^2}(\Pi_\tau v_\tau(t_n+\xi) )^2\partial_x \Pi_\tau\overline{v}_\tau(t_n+\xi) \dd\xi \bigg\|_{X_\tau^{\frac{1}{2},\frac{1}{2}}} \|
			\Pi_\tau v_\tau(t_n+\zeta)\|_{X_\tau^{\frac{1}{2},\frac{1}{2}}}^2\bigg).\notag
		\end{align}
		For the case $s\in (\frac{1}{2},\frac{3}{2}]$, by using \eqref{perturb} and \eqref{normdecr-b}, we have
		$$
		\begin{aligned}
			&
			\bigg\|\Pi_\tau \int_0^{\zeta}\fe^{i(\zeta-\xi)\p_x^2}(\Pi_\tau v_\tau(t_n+\xi) )^2\partial_x \Pi_\tau\overline{v}_\tau(t_n+\xi) \dd\xi \bigg\|_{X_\tau^{\frac{1}{2},\frac{1}{2}}}\\
			&\qquad \lesssim
			\tau \sup_{\xi \in[0,\tau]}  \Big\| \Pi_\tau \Big[ (\Pi_\tau v_\tau(t_n+\xi) )^2\partial_x \Pi_\tau\overline{v}_\tau(t_n+\xi) \Big] \Big\|_{X_\tau^{\frac{1}{2},\frac{1}{2}}} \\
			&\qquad
			\lesssim \tau^{1\over2} \sup_{\xi \in[0,\tau]}  \Big\| \Pi_\tau \Big[ (\Pi_\tau v_\tau(t_n+\xi) )^2\partial_x \Pi_\tau\overline{v}_\tau(t_n+\xi) \Big] \Big\|_{X_\tau^{\frac{1}{2},0}} . 
		\end{aligned}
		$$
		By applying the bilinear estimate $\|fg\|_{H^\frac{1}{2}}\lesssim \|f\|_{H^{\frac{1}{2}}}\|g\|_{H^s}$, H\"older's inequality, \eqref{eq:sobdis} and \eqref{normdecr-s}, we obtain that
		\begin{align}
			&\tau^{-1} \|\eta(\tfrac{t_{n+1}}{T_1})\mathcal{E}_{1,1}^{(2)}(t_{n+1})\|_{Z_\tau^{\frac{1}{2},-\frac{1}{2}}} \notag\\
			&\qquad\lesssim 
			\tau^{\frac{1}{2}} \sup_{\zeta,\xi\in[0,\tau]} \Big(  \Big\|\Pi_\tau \Big[(\Pi_\tau v_\tau(t_n+\xi) )^2\partial_x \Pi_\tau\overline{v}_\tau(t_n+\xi) \Big] \Big\|_{l_\tau^2 H^{1\over2}}\|
			v_\tau(t_n+\zeta)\|_{X_\tau^{\frac{1}{2},\frac{1}{2}}}^2\Big) \notag\\
			&\qquad\lesssim   \tau^{\frac{1}{2}} \sup_{\zeta,\xi\in[0,\tau]} \Big(  \|\Pi_\tau v_\tau(t_n+\xi) \|^2_{l_\tau^6 H^{s}} \|\Pi_\tau\overline{v}_\tau(t_n+\xi) \|_{l_\tau^6 H^{\frac{3}{2}}} \|v_\tau(t_n+\zeta)\|_{X_\tau^{\frac{1}{2},\frac{1}{2}}}^2\Big)\notag \\  
			&\qquad\lesssim   \tau^{\frac{1}{2}} \sup_{\zeta,\xi\in[0,\tau]} \Big(  \|\Pi_\tau v_\tau(t_n+\xi) \|^2_{X_\tau^{s,\frac{1}{2}}} \|\Pi_\tau {v}_\tau(t_n+\xi) \|_{X_\tau^{\frac{3}{2},\frac{1}{2}}}\|
			v_\tau(t_n+\zeta)\|_{X_\tau^{\frac{1}{2},\frac{1}{2}}}^2\Big)\notag \\
			&\qquad\lesssim   \tau^{\frac{s}{2}-\frac{1}{4}} \sup_{\zeta,\xi\in[0,\tau]} \Big(  \| v_\tau(t_n+\xi) \|^2_{X_\tau^{s,\frac{1}{2}}} \|\Pi_\tau {v}_\tau(t_n+\xi) \|_{X_\tau^{s,\frac{1}{2}}}\|
			v_\tau(t_n+\zeta)\|_{X_\tau^{\frac{1}{2},\frac{1}{2}}}^2\Big)     \notag\\
			&\qquad\leq C_T  \tau^{\frac{s}{2}-\frac{1}{4}}.\notag
		\end{align}
		For $s> \frac{3}{2}$, noting that $H^{s}\subset W^{1,\infty}$ and $\Pi_\tau \mathcal{E}_{1,1}^{(2)}=\mathcal{E}_{1,1}^{(2)}$, by using \eqref{eq:XtautoZtao} and \eqref{normdecr-s} we get 
		\begin{align}
			& \tau^{-1} \|\eta(\tfrac{t_{n+1}}{T_1})\mathcal{E}_{1,1}^{(2)}(t_{n+1})\|_{Z_\tau^{\frac{1}{2},-\frac{1}{2}}} \lesssim \tau^{-1} \|\eta(\tfrac{t_{n+1}}{T_1})\mathcal{E}_{1,1}^{(2)}(t_{n+1})\|_{X_\tau^{\frac{1}{2},0}}\lesssim \tau^{-\frac{5}{4}} \|\eta(\tfrac{t_{n+1}}{T_1})\mathcal{E}_{1,1}^{(2)}(t_{n+1})\|_{X_\tau^{0,0}} \notag \\
			&\qquad \lesssim \tau^{\frac{3}{4}} \sup_{\zeta,\xi\in[0,\tau]} \Big(\| v_\tau(t_n+\xi)^2\|_{l_\tau^2 L^2} \|\partial_x\overline{v}_\tau(t_n+\xi)\|_{l^\infty L^\infty}
			\|v_\tau(t_n+\zeta)\|_{l^\infty L^\infty} \|\partial_x \overline{v}_\tau(t_n+\zeta)\|_{l^\infty L^\infty}\Big) \notag \\
			&\qquad\leq C_T\tau^{\frac{3}{4}}. \notag
		\end{align}
		For $s>2$, by using \eqref{eq:XtautoZtao} and the bilinear estimate we have
		\begin{align}
			&\tau^{-1} \|\eta(\tfrac{t_{n+1}}{T_1})\mathcal{E}_{1,1}^{(2)}(t_{n+1})\|_{Z_\tau^{\frac{1}{2},-\frac{1}{2}}} \lesssim \tau^{-1} \|\eta(\tfrac{t_{n+1}}{T_1})\mathcal{E}_{1,1}^{(2)}(t_{n+1})\|_{X_\tau^{\frac{1}{2},0}} \notag \\
			& \qquad\lesssim  \tau \sup_{\zeta,\xi\in[0,\tau]} \|(\Pi_\tau v_\tau(t_n+\xi) )^2\partial_x \Pi_\tau\overline{v}_\tau(t_n+\xi) \Pi_\tau v_\tau(t_n+\zeta)\p_x\Pi_\tau\overline{v}_\tau(t_n+\zeta)\|_{l^2_\tau H^{1\over2}} \notag \\
			&\qquad \lesssim  \tau  \sup_{\zeta,\xi\in[0,\tau]} \Big( \|v_\tau(t_n+\xi)^2\|_{l_\tau^2 H^{ {1\over2}}} \|   \partial _x \overline{v}_\tau(t_n+\xi)\|_{l^\infty {H^1}}  \|   v_\tau(t_n+\zeta)\|_{l^\infty {H^1}} \|  \partial _x  \overline{v}_\tau(t_n+\zeta)\|_{l^\infty {H^1}}\Big)\notag \\
			&\qquad\leq  C_T \tau. \notag
		\end{align}
		In summary, we have 
		$$
		\tau^{-1} \|\eta(\tfrac{t_{n+1}}{T_1})\mathcal{E}_{1,1}^{(2)}(t_{n+1})\|_{Z_\tau^{\frac{1}{2},-\frac{1}{2}}} \leq C_T 
		\tau^{\frac{s}{2}-\frac{1}{4}},\quad s \in (\tfrac{1}{2},\tfrac{5}{2}].
		$$
		
		For $\mathcal{E}_{1,1}^{(3)}$, by \eqref{perturb}, \eqref{eq:gen_tri}, \eqref{normdecr-b}  and \eqref{normdecr-s}, we obtain that, for $s>\frac{1}{2}$,
		\begin{align}
			&\tau^{-1} \|\eta(\tfrac{t_{n+1}}{T_1})\mathcal{E}_{1,1}^{(3)}(t_{n+1})\|_{Z_\tau^{\frac{1}{2},-\frac{1}{2}}}\notag  \\
			&\qquad\lesssim
			\sup_{\zeta\in[0,\tau]}\bigg( \bigg\|\Pi_\tau \int_0^{\zeta}\fe^{i(\zeta-\xi)\p_x^2} (|\Pi_\tau v_\tau(t_n+\xi)|^4\Pi_\tau v_\tau(t_n+\xi)) \dd\xi \bigg\|_{X_\tau^{\frac{1}{2},\frac{1}{2}}} \|
			\Pi_\tau v_\tau(t_n+\zeta)\|_{X_\tau^{\frac{1}{2},\frac{1}{2}}}^2\bigg) \notag \\
			&\qquad \lesssim  \tau^{1\over4} \sup_{\zeta,\xi\in[0,\tau]}\Big( \|\Pi_\tau (|\Pi_\tau v_\tau(t_n+\xi)|^4\Pi_\tau v_\tau(t_n+\xi)) \|_{X_\tau^{0,0}}  \|\Pi_\tau v_\tau(t_n+\zeta)\|_{X_\tau^{\frac{1}{2},\frac{1}{2}}}^2\Big) 
			\notag \\
			&\qquad \lesssim  \tau^{1\over4}\sup_{\zeta,\xi\in[0,\tau]} \Big( \|  v_\tau(t_n+\xi)\|^4_{l^\infty L^\infty} \| v_\tau(t_n+\xi) \|_{l_\tau^2 L^2}\| v_\tau(t_n+\zeta)\|_{X_\tau^{\frac{1}{2},\frac{1}{2}}}^2 \Big)
			\leq C_T \tau^{1\over4}. \notag 
		\end{align}
		For $s> 1$, by using \eqref{eq:XtautoZtao} and \eqref{normdecr-s} we get that
		\begin{align}
			&\tau^{-1} \|\eta(\tfrac{t_{n+1}}{T_1})\mathcal{E}_{1,1}^{(3)}(t_{n+1})\|_{Z_\tau^{\frac{1}{2},-\frac{1}{2}}}\lesssim \tau^{- 1} \|\eta(\tfrac{t_{n+1}}{T_1})\mathcal{E}_{1,1}^{(3)}(t_{n+1})\|_{X_\tau^{\frac{1}{2},0}} \lesssim \tau^{- {5\over4}} \|\eta(\tfrac{t_{n+1}}{T_1})\mathcal{E}_{1,1}^{(3)}(t_{n+1})\|_{X_\tau^{0,0}}\notag  \\
			&\qquad \lesssim \tau^{3\over4} \sup_{\zeta,\xi\in[0,\tau]} \Big(  \| v_\tau(t_n+\xi)\|_{l^\infty L^\infty}^5  \| v_\tau(t_n+\zeta)\|_{l^\infty L^\infty} \|\overline{v}_\tau(t_n+\zeta)\|_{X_\tau^{1,0}} \Big) \leq C_T \tau^{3\over4} . \notag
		\end{align}
		For $s>2$, by using \eqref{eq:XtautoZtao} and the bilinear estimate we have
		\begin{align}
			& \tau^{-1} \|\eta(\tfrac{t_{n+1}}{T_1})\mathcal{E}_{1,1}^{(3)}(t_{n+1})\|_{Z_\tau^{\frac{1}{2},-\frac{1}{2}}}\lesssim \tau^{-1} \|\eta(\tfrac{t_{n+1}}{T_1})\mathcal{E}_{1,1}^{(3)}(t_{n+1})\|_{X_\tau^{{1\over2},0}} \notag  \\
			&\qquad \lesssim  \tau  \sup_{\zeta,\xi\in[0,\tau]} \Big( \| |v_\tau(t_n+\xi)|^4\|_{l_\tau^2  H^{ {1\over2}}} \| {v}_\tau(t_n+\xi)\|_{l^\infty {H^1} }   \| v_\tau(t_n+\zeta)\|_{l^\infty {H^1} } \| \partial_x \overline{v}_\tau(t_n+\zeta)\|_{l^\infty {H^1} }\Big)  \notag \\
			&\qquad\leq  C_T \tau. \notag
		\end{align}
		Combining the above estimates, we have
		$$
		\tau^{-1} \|\eta(\tfrac{t_{n+1}}{T_1})\mathcal{E}_{1,1}^{(3)}(t_{n+1})\|_{Z_\tau^{\frac{1}{2},-\frac{1}{2}}} \leq C_T 
		\tau^{\frac{s}{2}-\frac{1}{4}},\quad s \in (\tfrac{1}{2},\tfrac{5}{2}].
		$$
		Note that $\mathcal{E}_{1,1}^{(4)}$ can be similarly treated. 
		%Thus we have
		%$$
		% \tau^{-1} \|\eta(\tfrac{t_{n+1}}{T_1})\mathcal{E}_{1,1}^{(4)}(t_{n+1})\|_{Z_\tau^{\frac{1}{2},-\frac{1}{2}}} \leq C_T 
		%  \tau^{\frac{s}{2}-\frac{1}{4}},\quad s \in (\tfrac{1}{2},\tfrac{5}{2}].
		%$$
		
		For $\mathcal{E}_{1,1}^{(5)}$, by \eqref{perturb}, \eqref{eq:gen_tri}, \eqref{normdecr-b}  and \eqref{eq:gen_psi_est} and noting that $\psi(0)=0$, we obtain that, for $s>\frac{1}{2}$,
		\begin{align}
			&\tau^{-1} \|\eta(\tfrac{t_{n+1}}{T_1})\mathcal{E}_{1,1}^{(5)}(t_{n+1})\|_{Z_\tau^{\frac{1}{2},-\frac{1}{2}}}\notag  \\
			&\qquad\lesssim
			\sup_{\zeta\in[0,\tau]}\bigg( \bigg\|\Pi_\tau \int_0^{\zeta}\fe^{i(\zeta-\xi)\p_x^2} i\psi(\Pi_\tau v_\tau(t_n+\xi)) v_\tau(t_n+\xi) \dd\xi \bigg\|_{X_\tau^{\frac{1}{2},\frac{1}{2}}} \|
			\Pi_\tau v_\tau(t_n+\zeta)\|_{X_\tau^{\frac{1}{2},\frac{1}{2}}}^2\bigg) \notag \\
			&\qquad \lesssim \tau^{1\over2}\sup_{\zeta,\xi\in[0,\tau]} \Big( \|\psi(\Pi_\tau v_\tau(t_n+\xi)) v_\tau(t_n+\xi)  \|_{X_\tau^{{1\over2},0}}  \|
			v_\tau(t_n+\zeta)\|_{X_\tau^{\frac{1}{2},\frac{1}{2}}}^2\Big)
			\notag \\
			&\qquad \lesssim   \tau^{1\over2} \sup_{\zeta,\xi\in[0,\tau]} \bigg(\big(1+\|v_\tau(t_n+\xi)\|_{X_\tau^{\frac{1}{2},\frac{1}{2}}\cap Z_\tau^{0,\frac{1}{2}}}\big)^3  \|v_\tau(t_n+\xi)\|_{X_\tau^{\frac{1}{2},\frac{1}{2}}\cap Z_\tau^{0,\frac{1}{2}}}\|v_\tau(t_n+\xi)\|_{X_\tau^{\frac{1}{2},\frac{1}{2}}}  \notag \\
			&\qquad\qquad\qquad\qquad\quad\times \|
			v_\tau(t_n+\zeta)\|_{X_\tau^{\frac{1}{2},\frac{1}{2}}}^2\bigg)\notag \\
			&\qquad \leq C_T \tau^{1\over2}. \notag  
		\end{align}
		For $s>\frac{3}{2}$, by \eqref{eq:XtautoZtao}, \eqref{normdecr-s} and \eqref{Eqn:psil4} we have
		\begin{align}
			&\tau^{-1} \|\eta(\tfrac{t_{n+1}}{T_1})\mathcal{E}_{1,1}^{(5)}(t_{n+1})\|_{Z_\tau^{\frac{1}{2},-\frac{1}{2}}}\lesssim \tau^{-1} \|\eta(\tfrac{t_{n+1}}{T_1})\mathcal{E}_{1,1}^{(5)}(t_{n+1})\|_{X_\tau^{{1\over2},0}} \lesssim \tau^{- \frac{5}{4}} \|\eta(\tfrac{t_{n+1}}{T_1})\mathcal{E}_{1,1}^{(5)}(t_{n+1})\|_{X_\tau^{0,0}}  \notag  \\
			&\qquad \lesssim \tau^{\frac{3}{4}}  \sup_{\zeta,\xi\in[0,\tau]} \big( \|\psi(\Pi_\tau v_\tau(t_n+\xi))\|_{l_\tau^4} \|v_\tau(t_n+\xi)\|_{l_\tau^4L^2}
			\| v_\tau(t_n+\zeta)\|_{l^\infty L^\infty} \|\p_x \overline{v}_\tau(t_n+\zeta) \|_{l^\infty L^\infty} \big) \notag \\
			&\qquad\leq  C_T \tau^{\frac{3}{4}}. \notag
		\end{align}
		For $s>2$, by using \eqref{eq:XtautoZtao} and the bilinear estimate we have
		\begin{align}
			&\tau^{-1} \|\eta(\tfrac{t_{n+1}}{T_1})\mathcal{E}_{1,1}^{(5)}(t_{n+1})\|_{Z_\tau^{\frac{1}{2},-\frac{1}{2}}}\lesssim \tau^{-1} \|\eta(\tfrac{t_{n+1}}{T_1})\mathcal{E}_{1,1}^{(5)}(t_{n+1})\|_{X_\tau^{{1\over2},0}}  \notag  \\
			&\qquad \lesssim \tau \sup_{\zeta,\xi\in[0,\tau]} \big( \|\psi(\Pi_\tau v_\tau(t_n+\xi))\|_{l_\tau^4} \|v_\tau(t_n+\xi)\|_{l_\tau^4 H^{{1\over2}}} 
			\|   v_\tau(t_n+\zeta)\|_{l^\infty  {H^1}} \| \p_x \overline{v}_\tau(t_n+\zeta) \|_{l^\infty {H^1}} \big)  \notag \\
			&\qquad\leq  C_T \tau. \notag
		\end{align}
		These estimates give
		$$
		\tau^{-1} \|\eta(\tfrac{t_{n+1}}{T_1})\mathcal{E}_{1,1}^{(5)}(t_{n+1})\|_{Z_\tau^{\frac{1}{2},-\frac{1}{2}}} \leq C_T 
		\tau^{\frac{s}{2}-\frac{1}{4}},\quad s \in (\tfrac{1}{2},\tfrac{5}{2}].
		$$
		Note that the above estimates are easily extended for $\mathcal{E}_{1,2}$. %Consequently, we arrive at
		%$$
		% \tau^{-1} \|\eta(\tfrac{t_{n+1}}{T_1})\mathcal{E}_{1,1}(t_{n+`})\|_{Z_\tau^{\frac{1}{2},-\frac{1}{2}}}+ \tau^{-1} \|\eta(\tfrac{t_{n+1}}{T_1})\mathcal{E}_{1,2}(t_{n+1})\|_{Z_\tau^{\frac{1}{2},-\frac{1}{2}}} \leq C_T 
		%  \tau^{\frac{s}{2}-\frac{1}{4}},\quad s \in  (\tfrac{1}{2},\tfrac{5}{2}].
		%$$
		
		We now consider $\mathcal{E}_{1,3}$, given as
		\begin{align}
			\mathcal{E}_{1,3}(t_{n+1})
			&=-\int_0^\tau\fe^{-i\zeta\p_x^2}\Pi_{\tau}\bigg[ 
			(\Pi_\tau v_\tau(t_n) )^2\p_x\Pi_\tau (\e^{-i\zeta\partial_x^2}-1)\overline{v}_\tau(t_n)  \bigg]\dd\zeta \notag\\
			&\quad+\int_0^\tau\fe^{-i\zeta\p_x^2}\Pi_{\tau}\bigg[ (\Pi_\tau v_\tau(t_n) )^2\p_x\Pi_\tau \int_0^{\zeta}\fe^{-i(\zeta-\xi)\p_x^2}\overline{(\Pi_\tau v_\tau(t_n+\xi) )^2\partial_x \Pi_\tau\overline{v}_\tau(t_n+\xi)} \dd\xi   \bigg]\dd\zeta \notag \\
			&\quad+\int_0^\tau\fe^{-i\zeta\p_x^2}\Pi_{\tau}\bigg[ (\Pi_\tau v_\tau(t_n))^2 \p_x\Pi_\tau \int_0^{\zeta}\fe^{-i(\zeta-\xi)\p_x^2}\frac{i}{2}\overline{(|\Pi_\tau v_\tau(t_n+\xi)|^4\Pi_\tau v_\tau(t_n+\xi))}\dd\xi 
			\bigg]\dd\zeta \notag \\
			&\quad-\int_0^\tau\fe^{-i\zeta\p_x^2}\Pi_{\tau}\bigg[(\Pi_\tau v_\tau(t_n))^2 \partial_x \Pi_\tau \int_0^{\zeta}\fe^{-i(\zeta-\xi)\p_x^2}i\mu  \overline{(|\Pi_\tau v_\tau(t_n+\xi)|^2\Pi_\tau v_\tau(t_n+\xi))}\dd\xi  \bigg]\dd\zeta \notag \\
			&\quad+\int_0^\tau\fe^{-i\zeta\p_x^2}\Pi_{\tau}\bigg[ (\Pi_\tau v_\tau(t_n))^2 \partial_x \Pi_\tau \int_0^{\zeta}\fe^{-i(\zeta-\xi)\p_x^2}i\overline{\psi(\Pi_\tau v_\tau(t_n+\xi)) \Pi_\tau v_\tau(t_n+\xi)}\dd\xi   \bigg]\dd\zeta \notag \\
			&=: \mathcal{E}_{1,3}^{(1)}(t_{n+1})+\mathcal{E}_{1,3}^{(2)}(t_{n+1})+\mathcal{E}_{1,3}^{(3)}(t_{n+1})+\mathcal{E}_{1,3}^{(4)}(t_{n+1})+\mathcal{E}_{1,3}^{(5)}(t_{n+1}). \label{Eqn:E13-12345}
		\end{align}
		Similar to the estimate for $\mathcal{E}_{1,1}^{(1)}$ presented before, we can derive that
		$$
		\tau^{-1} \|\eta(\tfrac{t_{n+1}}{T_1})\mathcal{E}_{1,3}^{(1)}(t_{n+1})\|_{Z_\tau^{\frac{1}{2},-\frac{1}{2}}}  \leq C_T \tau^{\frac{s}{2}-\frac{1}{4}}.
		$$
		For $\mathcal{E}_{1,3}^{(2)}$, if $s\in(\frac{1}{2},\frac{3}{2}]$, following a similar procedure as before, we establish 
		$$
		\tau^{-1} \|\eta(\tfrac{t_{n+1}}{T_1})\mathcal{E}_{1,3}^{(2)}(t_{n+1})\|_{Z_\tau^{\frac{1}{2},-\frac{1}{2}}}\leq C_T \tau^{\frac{s}{2}-\frac{1}{4}}.
		$$
		For $s \in (\frac{3}{2},2]$, noting that $H^{s}\subset W^{1,\infty}$, by using \eqref{eq:XtautoZtao}, \eqref{normdecr-s} and the bilinear estimate we get that
		\begin{align}
			& \tau^{-1} \|\eta(\tfrac{t_{n+1}}{T_1})\mathcal{E}_{1,3}^{(2)}(t_{n+1})\|_{Z_\tau^{\frac{1}{2},-\frac{1}{2}}} \lesssim \tau^{-1} \|\eta(\tfrac{t_{n+1}}{T_1})\mathcal{E}_{1,3}^{(2)}(t_{n+1})\|_{X_\tau^{\frac{1}{2},0}}\lesssim \tau^{-\frac{5}{4}} \|\eta(\tfrac{t_{n+1}}{T_1})\mathcal{E}_{1,3}^{(2)}(t_{n+1})\|_{X_\tau^{0,0}} \notag \\
			&\qquad \lesssim \tau^{\frac{3}{4}} \sup_{\zeta,\xi\in[0,\tau]} \Big(\| v_\tau(t_n)\|^2_{l^\infty L^\infty} \| \overline{v}_\tau(t_n+\xi)\|_{l^\infty H^1}^2  \|  \Pi_\tau {v}_\tau(t_n+\zeta)\|_{X_\tau^{2,0}} \Big) \notag \\
			&\qquad\lesssim \tau^{\frac{s}{2}-\frac{1}{4}}  \sup_{\zeta,\xi\in[0,\tau]} \Big(\| v_\tau(t_n)\|^2_{l^\infty L^\infty} \| \overline{v}_\tau(t_n+\xi)\|_{l^\infty H^1}^2  \|  \Pi_\tau {v}_\tau(t_n+\zeta)\|_{X_\tau^{s,0}} \Big) \leq C_T \tau^{\frac{s}{2}-\frac{1}{4}} . \notag
		\end{align}
		For $s>2$, we have
		\begin{align}
			& \tau^{-1} \|\eta(\tfrac{t_{n+1}}{T_1})\mathcal{E}_{1,3}^{(2)}(t_{n+1})\|_{Z_\tau^{\frac{1}{2},-\frac{1}{2}}} \lesssim \tau^{-1} \|\eta(\tfrac{t_{n+1}}{T_1})\mathcal{E}_{1,3}^{(2)}(t_{n+1})\|_{X_\tau^{\frac{1}{2},0}} \notag \\
			&\qquad \lesssim  \tau  \sup_{\zeta,\xi\in[0,\tau]} \Big(\|v_\tau(t_n)\|^2_{l^\infty {H^1}} \| \overline{v}_\tau(t_n+\xi)\|_{l^\infty {H^2}}^2  \|  \Pi_\tau \overline{v}_\tau(t_n+\zeta)\|_{X^{{5\over 2},0}_\tau}\Big)    \notag \\
			&\qquad\lesssim  \tau^{\frac{s}{2}-\frac{1}{4}}  \sup_{\zeta,\xi\in[0,\tau]} \Big(\|v_\tau(t_n)\|^2_{l^\infty {H^1} } \| \overline{v}_\tau(t_n+\xi)\|_{l^\infty {H^2}}^2    \|\Pi_\tau  \overline{v}_\tau(t_n+\zeta)\|_{X^{s,0}_\tau}\Big)  \leq C_T \tau^{\frac{s}{2}-\frac{1}{4}} . \notag
		\end{align}
		A combination of the above estimates gives
		$$
		\tau^{-1} \|\eta(\tfrac{t_{n+1}}{T_1})\mathcal{E}_{1,3}^{(2)}(t_{n+1})\|_{Z_\tau^{\frac{1}{2},-\frac{1}{2}}}  \leq C_T \tau^{\frac{s}{2}-\frac{1}{4}} ,\quad s \in (\tfrac{1}{2},\tfrac{5}{2}].
		$$
		Proceeding along the same lines, we can bound the remaining terms in \eqref{Eqn:E13-12345} to obtain
		$$
		\tau^{-1} \|\eta(\tfrac{t_{n+1}}{T_1})\mathcal{E}_{1,3}(t_{n+1})\|_{Z_\tau^{\frac{1}{2},-\frac{1}{2}}}    \leq C_T \tau^{\frac{s}{2}-\frac{1}{4}} . 
		$$
		In conclusion, we arrive at
		$$
		\tau^{-1} \|\eta(\tfrac{t_{n+1}}{T_1})\mathcal{E}_{1} (t_{n+1})\|_{Z_\tau^{\frac{1}{2},-\frac{1}{2}}}    \leq C_T \tau^{\frac{s}{2}-\frac{1}{4}} ,\quad s \in  (\tfrac{1}{2},\tfrac{5}{2}].
		$$
		
		\medskip
		
		\noindent {(b)~Estimate of} $\tau^{-1}\|\eta(\frac{t_{n+1}}{T_1})\mathcal{E}_i(t_{n+1})\|_{Z_\tau^{\frac{1}{2},-\frac{1}{2}}}$, $i=2,3$.
		
		The related estimates for $\mathcal{E}_{i}$, $i=2,3$ can be obtained in a similar way. %We have
		%$$
		%\tau^{-1}\|\eta(\tfrac{t_{n+1}}{T_1})\mathcal{E}_{2}(t_{n+1})\|_{Z_{\tau}^{\frac{1}{2},-\frac{1}{2}}}+\tau^{-1}\|\eta(\tfrac{t_{n+1}}{T_1})\mathcal{E}_{3}(t_{n+1})\|_{Z_{\tau}^{\frac{1}{2},-\frac{1}{2}}} \leq C_T \tau^{\frac{s}{2}-\frac{1}{4}}.
		%$$
		
		\medskip
		
		\noindent  {(c)~Estimate of} $\tau^{-1}\|\eta(\frac{t_{n+1}}{T_1})\mathcal{E}_4(t_{n+1})\|_{Z_\tau^{\frac{1}{2},-\frac{1}{2}}}$.
		
		We first rewrite $\mathcal{E}_4$ as
		$$
		\begin{aligned}
			\mathcal{E}_{4}(t_{n+1})&=i  \int_0^\tau\fe^{-i\zeta\p_x^2}\Big[\psi(\Pi_\tau v_\tau(t_n+\zeta))-\psi(\Pi_\tau v_\tau(t_n))   \Big]\Pi_\tau v_\tau(t_n+\zeta)\dd\zeta \\
			&\quad+i \int_0^\tau\fe^{-i\zeta\p_x^2} \psi(\Pi_\tau v_\tau(t_n)) \Big[\Pi_\tau v_\tau(t_n+\zeta)-\Pi_\tau v_\tau(t_n)\Big]\dd\zeta =:\mathcal{E}_{4,1}(t_{n+1})+\mathcal{E}_{4,2}(t_{n+1}).
		\end{aligned}
		$$
		%Thanks to \eqref{eq:linfty_embdis}-\eqref{eq:x_in_ydis} and \eqref{bounddis}, it holds 
		%$$
		%\|v_\tau(t_n)\|_{l^\infty H^{1\over2}}     \lesssim  \tau^{-\epsilon} \|v_{\tau}(t_{n})  \|_{X^{s, {1 \over 2}+\epsilon}_{\tau}}\lesssim_T \tau^{-\epsilon}.
		%$$
		Thanks to Proposition~\ref{diffpj}, we have $\|v_\tau\|_{L_t^\infty H^{1\over2}}\leq C_T$. Noting the mass conservation property \eqref{Eqn:mass-vtau} of \eqref{pjgdnls}, we derive from \eqref{Eeqn:recall-psi} that
		$$
		\begin{aligned}
			|\psi(\Pi_\tau v_\tau(t_n+\zeta)) - \psi(\Pi_\tau v_\tau(t_n))| &\lesssim   ( 1 + \|\Pi_\tau v_\tau(t_n+\zeta)\|_{H^{\frac{1}{2}}} + \|\Pi_\tau v_\tau(t_n)\|_{H^{\frac{1}{2}}} )^3 \\
			&\qquad \times \|\Pi_\tau v_\tau(t_n+\zeta)-\Pi_\tau v_\tau(t_n)\|_{H^{\frac{1}{2}}}.
		\end{aligned}
		$$
		For $\mathcal{E}_{4,1}$ and $\mathcal{E}_{4,2}$, using \eqref{eq:XtautoZtao}, the above inequality, and the fact that $\psi(0)=0$, we have
		\begin{align}
			&\tau^{-1}\|\eta(\tfrac{t_{n+1}}{T_1})\mathcal{E}_{4,1}(t_{n+1})\|_{Z_{\tau}^{\frac{1}{2},-\frac{1}{2}}} \lesssim \tau^{-1}\|\eta(\tfrac{t_{n+1}}{T_1})\mathcal{E}_{4,1}(t_{n+1})\|_{X_{\tau}^{\frac{1}{2},0}} \notag \\
			&\qquad \lesssim \sup_{\zeta\in[0,\tau]} \Big\| \Big[\psi(\Pi_\tau v_\tau(t_n+\zeta))-\psi(\Pi_\tau v_\tau(t_n))   \Big]\Pi_\tau v_\tau(t_n+\zeta)\Big\|_{X_{\tau}^{\frac{1}{2},0}}   \notag \\
			& \qquad 
			\lesssim    \sup_{\zeta\in[0,\tau]} \Big(\| \psi(\Pi_\tau v_\tau(t_n+\zeta))-\psi(\Pi_\tau v_\tau(t_n))   \|_{l_\tau^2} \| \Pi_\tau v_\tau(t_n+\zeta)\|_{l^\infty H^{\frac{1}{2}}}    \Big) \notag\\
			& \qquad 
			\lesssim   \sup_{\zeta\in[0,\tau]}  \Big( \Big\| \big( 1+ \| \Pi_\tau v_\tau(t_n+\zeta) \|_{H^{\frac{1}{2}}} +   \| \Pi_\tau v_\tau(t_n) \|_{H^{\frac{1}{2}}}\big)^3 \| \Pi_\tau v_\tau(t_n+\zeta)- \Pi_\tau v_\tau(t_n)   \|_{H^{\frac{1}{2}}} \Big\|_{l_\tau^2} \notag \\
			&\qquad\qquad\qquad\quad \times \| \Pi_\tau v_\tau(t_n+\zeta)\|_{l^\infty H^{\frac{1}{2}}}    \Big) \notag\\
			&\qquad\leq  C_T \sup_{\zeta\in[0,\tau]} \|\Pi_\tau v_\tau(t_n+\zeta)-\Pi_\tau v_\tau(t_n)\|_{X_\tau^{{1\over2},0}},\notag\\
			&\tau^{-1}\|\eta(\tfrac{t_{n+1}}{T_1})\mathcal{E}_{4,2}(t_{n+1})\|_{Z_{\tau}^{\frac{1}{2},-\frac{1}{2}}} \lesssim \tau^{-1}\|\eta(\tfrac{t_{n+1}}{T_1})\mathcal{E}_{4,2}(t_{n+1})\|_{X_{\tau}^{\frac{1}{2},0}} \notag \\
			&\qquad\lesssim    \sup_{\zeta\in[0,\tau]} \|\psi(\Pi_\tau v_\tau(t_n))\|_{l^\infty} \|\Pi_\tau v_\tau(t_n+\zeta)-\Pi_\tau v_\tau(t_n)\|_{X_\tau^{{1\over2},0}}\notag \\
			&\qquad\lesssim  \sup_{\zeta\in[0,\tau]}  \Big( (1+ \|\Pi_\tau v_\tau(t_n)\|_{l^\infty H^{\frac{1}{2}}} )^3 \|\Pi_\tau v_\tau(t_n)\|_{l^\infty H^{\frac{1}{2}}} \|\Pi_\tau v_\tau(t_n+\zeta)-\Pi_\tau v_\tau(t_n)\|_{X_\tau^{{1\over2},0}} \Big) \notag \\
			&\qquad\leq C_T \sup_{\zeta\in[0,\tau]}  \|\Pi_\tau v_\tau(t_n+\zeta)-\Pi_\tau v_\tau(t_n)\|_{X_\tau^{{1\over2},0}} .\notag
		\end{align}
		Consequently, we only need to estimate $\|\Pi_\tau v_\tau(t_n+\zeta)-\Pi_\tau v_\tau(t_n)\|_{X_\tau^{{1\over2},0}}$. By Duhamel's formula, we arrive at
		\begin{align}
			&\Pi_\tau v_\tau(t_n+\zeta)-\Pi_\tau v_\tau(t_n) \notag\\
			&\qquad= \Pi_\tau(\fe^{i\zeta\p_x^2}-1)v_\tau(t_n) -\int_0^{\zeta}\fe^{i(\zeta-\xi)\p_x^2}\Pi_\tau \Big[(\Pi_\tau v_\tau(t_n+\xi))^2\partial_x \Pi_\tau\overline{v}_\tau(t_n+\xi)\Big]\dd\xi \notag\\
			&\qquad\quad+\frac{i}{2}\int_0^{\zeta}\fe^{i(\zeta-\xi)\p_x^2} \Pi_\tau(|\Pi_\tau v_\tau(t_n+\xi)|^4\Pi_\tau v_\tau(t_n+\xi)) \dd\xi \notag\\
			&\qquad\quad- i \mu \int_0^{\zeta}\fe^{i(\zeta-\xi)\p_x^2} \Pi_\tau(|\Pi_\tau v_\tau(t_n+\xi)|^2\Pi_\tau v_\tau(t_n+\xi))\dd\xi \notag\\
			&\qquad\quad+i \int_0^{\zeta}\fe^{i(\zeta-\xi)\p_x^2} \psi(\Pi_\tau v_\tau(t_n+\xi))\Pi_\tau v_\tau(t_n+\xi) \dd\xi \notag\\
			&\qquad=:\mathcal{E}_4^{(1)} (t_{n+1})+\mathcal{E}_4^{(2)}(t_{n+1})+\mathcal{E}_4^{(3)}(t_{n+1})+\mathcal{E}_4^{(4)}(t_{n+1})+\mathcal{E}_4^{(5)}(t_{n+1}). \notag
		\end{align}
		We focus on $\mathcal{E}_4^{(2)}$, as the remaining terms can be estimated by employing the techniques used in part (a). By \eqref{normdecr-b} and Theorem~\ref{Thm:multiplier1}, for $s>\frac{1}{2}$ we have
		$$
		\|\eta(\tfrac{t_{n+1}}{T_1})\mathcal{E}_{4}^{(2)}(t_{n+1})\|_{X_\tau^{{1\over2},0}} \lesssim \tau^{-\frac{1}{2}} \|\eta(\tfrac{t_{n+1}}{T_1})\mathcal{E}_{4}^{(2)}(t_{n+1})\|_{X_\tau^{{1\over2},-\frac{1}{2}}} \lesssim \tau^{\frac{1}{2}}  \sup_{\xi\in[0,\tau]}\Big( \|v_\tau(t_n+\xi)\|^2_{X_\tau^{\frac{1}{2},\frac{1}{2}}} \|{v}_\tau(t_n+\xi)\|_{X_\tau^{\frac{1}{2},\frac{1}{2}}} \Big).
		$$
		For $s>\frac{3}{2}$, noting that $H^{s}\subset W^{1,\infty}$ and using \eqref{normdecr-s}, one has
		$$
		\|\eta(\tfrac{t_{n+1}}{T_1})\mathcal{E}_{4}^{(2)}(t_{n+1})\|_{X_\tau^{{1\over2},0}} \lesssim \tau^{-\frac{1}{4}} \|\eta(\tfrac{t_{n+1}}{T_1})\mathcal{E}_{4}^{(2)}(t_{n+1})\|_{X_\tau^{0,0}} \lesssim \tau^{\frac{3}{4}}  \sup_{\xi\in[0,\tau]}\Big( \|{v}_\tau(t_n+\xi)\|^2_{l^\infty L^\infty} \|\partial_x \overline{v}_\tau(t_n+\xi)\|_{l_\tau^2L^2}\Big).
		$$
		For $s>2$, we arrive at
		$$
		\|\eta(\tfrac{t_{n+1}}{T_1})\mathcal{E}_{4}^{(2)}(t_{n+1})\|_{X_\tau^{{1\over2},0}} \lesssim  \tau  \sup_{\xi\in[0,\tau]}\Big( \|{v}_\tau(t_n+\xi)\|^2_{l^\infty { H^{1}}} \|\partial_x  \overline{v}_\tau(t_n+\xi)\|_{l_\tau^2 H^{1\over2}}\Big).
		$$
		In summary, it holds
		$$
		\|\eta(\tfrac{t_{n+1}}{T_1})\mathcal{E}_{4}^{(2)}(t_{n+1})\|_{X_\tau^{{1\over2},0}} \leq C_T \tau^{\frac{s}{2}-\frac{1}{4}},\quad s\in(\tfrac{1}{2},\tfrac{5}{2}].
		$$
		Thus, the proof is completed.
	\end{proof}

	%%%%%%%%%%%%%%%%%%%%%%%%%%%
	\section{Global error analysis}\label{sectionglobal}
	%%%%%%%%%%%%%%%%%%%%%%%%%%%
	Let $e^n=v_\tau(t_n)-v^n$ denote the global error for the projected dNLS equation \eqref{pjgdnls}. Thanks to Proposition~\ref{diffpj}, we have
	$$
	\|v(t_n)-v^n\|_{H^{1\over2}} \leq \|v(t_n)-v_\tau(t_n)\|_{H^{1\over2}}+\|v_\tau(t_n)-v^n\|_{H^{1\over2}} \leq C_T\tau^{\frac{s}{2}-\frac{1}{4}} + \|e^n\|_{H^{1\over2}}.
	$$
	It remains to estimate $e^n$, which can be rewritten as
	$$
	\begin{aligned}
		e^n&=  \Phi^\tau(v_\tau(t_{n-1}))  - \Phi^\tau(v^{n-1}) +\mathcal{E}_{loc}(t_{n}) \\
		&=\e^{i\tau \partial_x^2} e^{n-1} +\tau \e^{i \tau \partial_x^2}\varphi_1(-i\tau \partial_x^2) (F_\tau(v_\tau(t_{n-1}))- F_\tau(v^{n-1}) ) 
		+\mathcal{E}_{loc}(t_{n}) \\
		&=\cdots = \tau\sum_{k=0}^{n-1} {\rm e}^{i(n-k)\tau\p_x^2}\varphi_1(-i\tau\p_x^2) ( F_\tau(v_\tau(t_{k})) -F_\tau(v^{k})  ) +
		\sum_{k=0}^{n-1} {\rm e}^{i(n-k-1)\tau\p_x^2} \mathcal{E}_{loc}(t_{k+1}),
	\end{aligned}
	$$
	in which $e^0=0$ was used.
	Using the properties of $\eta$ we find that $e^n$ satisfies: for $0\leq n \leq N_1= \lfloor \frac{T_1}{\tau} \rfloor$ with $T_1\leq \min\{1,T\}$,
	\begin{equation}\label{Eqn:en-formula}
		e^n= \tau\eta(t_n)\sum_{k=0}^{n-1} {\rm e}^{i(n-k)\tau \p_x^2}\eta\bigg(\frac{t_k}{T_1}\bigg) \varphi_1(-i\tau\p_x^2) ( F_\tau(v_\tau(t_{k})) -F_\tau(v^{k})  ) +\mathcal{R}_n,
	\end{equation}
	where
	$$
	\mathcal{R}_n=   \eta(t_n)\sum_{k=0}^{n-1} {\rm e}^{i(n-k-1)\tau \p_x^2} \eta\bigg(\frac{t_{k+1}}{T_1}\bigg) \mathcal{E}_{loc}(t_{k+1}).
	$$
	
	By using \eqref{Eqn:discrete-convolution} and Lemma~\ref{Lem:local-error}, we derive from \eqref{Eqn:en-formula} that
	\begin{equation}\label{Eqn:error-conclusion-1}
		\|e^n\|_{Z_\tau^{\frac{1}{2},\frac{1}{2}}} \leq  C_T   \big\| \eta (\tfrac{t_n}{T_1} ) \big( F_\tau(v_\tau(t_{n})) -F_\tau(v^{n})  \big)  \big\|_{Z_\tau^{\frac{1}{2},-\frac{1}{2}}}  + C_T \tau^{\frac{s}{2}-\frac{1}{4}},
	\end{equation}
	where the first term on the right-hand side is bounded by
	\begin{equation}\label{Eqn:Z1/21/2toZ1/21/2}
		\begin{aligned}
			& \big\| \eta (\tfrac{t_n}{T_1} ) \big( F_\tau(v_\tau(t_{n})) -F_\tau(v^{n})  \big)  \big\|_{Z_\tau^{\frac{1}{2},-\frac{1}{2}}} \\
			&\qquad\lesssim \Big\|\eta(\tfrac{t_n}{T_1}) \Pi_\tau \Big[(\Pi_\tau v_\tau(t_n))^2\p_x\Pi_\tau\overline{v}_\tau(t_n)-(\Pi_\tau v^n)^2\p_x\Pi_\tau\overline{v}^n\Big] \Big\|_{Z_{\tau}^{\frac{1}{2},-\frac{1}{2}}} \\
			&\qquad \quad +  \Big\|\eta(\tfrac{t_n}{T_1})  \Pi_{\tau} \Big[|\Pi_\tau v_\tau(t_n)|^4\Pi_\tau v_\tau(t_n)-|\Pi_\tau v^n|^4\Pi_\tau v^n)\Big] \Big\|_{Z_{\tau}^{\frac{1}{2},-\frac{1}{2}}} \\
			&\qquad \quad +   \Big\|\eta(\tfrac{t_n}{T_1}) \Pi_{\tau}\Big[|\Pi_\tau v_\tau(t_n)|^2\Pi_\tau v_\tau(t_n)-|\Pi_\tau v^n)|^2\Pi_\tau v^n\Big]  \Big\|_{Z_{\tau}^{\frac{1}{2},-\frac{1}{2}}} \\
			&\qquad \quad +   \Big\|\eta(\tfrac{t_n}{T_1})  \Big[\psi(\Pi_\tau v_\tau(t_n))\Pi_\tau v_\tau(t_n)-\psi(\Pi_\tau v^n)\Pi_\tau v^n\Big]  \Big\|_{Z_{\tau}^{\frac{1}{2},-\frac{1}{2}}} .
		\end{aligned}
	\end{equation}
	For the first term on the right-hand side of \eqref{Eqn:Z1/21/2toZ1/21/2}, by noting that 
	$$
	\begin{aligned}
		(\Pi_\tau v^n)^2\p_x\Pi_\tau\overline{v}^n &  =  (\Pi_\tau v_\tau(t_n)-\Pi_\tau e^n)^2 (   \p_x\Pi_\tau\overline{v}_\tau(t_n) - \p_x\Pi_\tau\overline{e}^n) \\
		& =(\Pi_\tau v_\tau(t_n))^2 \p_x \Pi_\tau v_\tau(t_n) - (\Pi_\tau v_\tau(t_n))^2 \p_x\Pi_\tau\overline{e}^n \\
		&\quad - \big[ 2\Pi_\tau v_\tau(t_n) \Pi_\tau e^n -(\Pi_\tau e^n)^2\big](   \p_x\Pi_\tau\overline{v}_\tau(t_n) - \p_x\Pi_\tau\overline{e}^n) ,
	\end{aligned}
	$$
	we obtain from \eqref{eq:gen_tri} that
	$$
	\begin{aligned}
		&\Big\|\eta(\tfrac{t_n}{T_1}) \Pi_\tau \Big[(\Pi_\tau v_\tau(t_n))^2\p_x\Pi_\tau\overline{v}_\tau(t_n)-(\Pi_\tau v^n)^2\p_x\Pi_\tau\overline{v}^n\Big] \Big\|_{Z_{\tau}^{\frac{1}{2},-\frac{1}{2}}} \\
		&\qquad \leq C_T T_1^{\varepsilon}  \Big( \|e^n\|_{Z_\tau^{\frac{1}{2},\frac{1}{2}}} +  \|e^n\|_{Z_\tau^{\frac{1}{2},\frac{1}{2}}}^2 +  \|e^n\|_{Z_\tau^{\frac{1}{2},\frac{1}{2}}}^3 \Big).
	\end{aligned}
	$$
	For the second term on the right-hand side of \eqref{Eqn:Z1/21/2toZ1/21/2}, by noting the relation
	$$
	\begin{aligned}
		|a-b|^4 (a-b) & = (a-b)^3(\overline{a}-\overline{b})^2 = (a^3-3a^2b+3ab^2-b^3)(\overline{a}^2-2\overline{a}\overline{b}+\overline{b}^2)
	\end{aligned}
	$$
	and taking $a= \Pi_\tau{v}_\tau(t_n)$, $b= \Pi_\tau e^n$, it follows from \eqref{eq:XtautoZtao} and \eqref{eq:gen_quint_pol_est}  that
	$$
	\begin{aligned}
		& \Big\|\eta(\tfrac{t_n}{T_1})  \Pi_{\tau} \Big[|\Pi_\tau v_\tau(t_n)|^4\Pi_\tau v_\tau(t_n)-|\Pi_\tau v^n|^4\Pi_\tau v^n)\Big] \Big\|_{Z_{\tau}^{\frac{1}{2},-\frac{1}{2}}} \\
		&\qquad\lesssim \Big\|\eta(\tfrac{t_n}{T_1})  \Pi_{\tau} \Big[|\Pi_\tau v_\tau(t_n)|^4\Pi_\tau v_\tau(t_n)-|\Pi_\tau v^n|^4\Pi_\tau v^n)\Big] \Big\|_{X_{\tau}^{\frac{1}{2},-\frac{3}{8}}} \\
		&\qquad\leq C_T  T_1^\varepsilon \Big( \|e^n\|_{Z_\tau^{\frac{1}{2},\frac{1}{2}}} +  \|e^n\|_{Z_\tau^{\frac{1}{2},\frac{1}{2}}}^2 +  \|e^n\|_{Z_\tau^{\frac{1}{2},\frac{1}{2}}}^3 +  \|e^n\|_{Z_\tau^{\frac{1}{2},\frac{1}{2}}}^4  +  \|e^n\|_{Z_\tau^{\frac{1}{2},\frac{1}{2}}}^5\Big).
	\end{aligned}
	$$
	The third term can be similarly treated. For the last term, using \eqref{eq:XtautoZtao} and \eqref{Eqn:T-inverse}, we obtain
	\begin{align}
		&  \Big\|\eta(\tfrac{t_n}{T_1})  \Big[\psi(\Pi_\tau v_\tau(t_n))\Pi_\tau v_\tau(t_n)-\psi(\Pi_\tau v^n)\Pi_\tau v^n)\Big]  \Big\|_{Z_{\tau}^{\frac{1}{2},-\frac{1}{2}}} \notag \\
		&\qquad \lesssim  T_1^\varepsilon \Big\|\eta(\tfrac{t_n}{T_1})  \Big[\psi(\Pi_\tau v_\tau(t_n))-\psi\big(\Pi_\tau v_\tau(t_n) - \Pi_\tau e^n\big)\Big]\Pi_\tau v_\tau(t_n)  \Big\|_{X_{\tau}^{\frac{1}{2},0}} \notag \\
		&\qquad\quad+ T_1^\varepsilon \Big\|\eta(\tfrac{t_n}{T_1})  \psi\big(\Pi_\tau v_\tau(t_n) - \Pi_\tau e^n\big)  \Pi_\tau e^n \Big\|_{X_{\tau}^{\frac{1}{2},0}} .\notag 
	\end{align}
	Applying \eqref{Eqn:psil4} and noting that $\psi(0)=0$, we further get
	\begin{align}
		&  \Big\|\eta(\tfrac{t_n}{T_1})  \Big[\psi(\Pi_\tau v_\tau(t_n))\Pi_\tau v_\tau(t_n)-\psi(\Pi_\tau v^n)\Pi_\tau v^n)\Big]  \Big\|_{Z_{\tau}^{\frac{1}{2},-\frac{1}{2}}} \notag \\
		&\qquad \lesssim  T_1^\varepsilon 
		\|\psi(\Pi_\tau v_\tau(t_n)) - \psi(\Pi_\tau v_\tau(t_n) - \Pi_\tau e^n)\|_{l^4_\tau} \|\Pi_\tau v_\tau(t_n)\|_{l^4_\tau H^{\frac{1}{2}}} \notag \\
		&\qquad\quad+ T_1^\varepsilon \big\|\psi\big(\Pi_\tau v_\tau(t_n) - \Pi_\tau e^n\big)\big\|_{l_\tau^4} \|\Pi_\tau e^n \|_{l_\tau^4 H^\frac{1}{2}} \notag \\
		&\qquad \lesssim   T_1^\varepsilon \Big( 1+ \|v_\tau(t_n)\|_{Z_\tau^{\frac{1}{2},\frac{1}{2}}} +  \|\Pi_\tau v_\tau(t_n) - \Pi_\tau e^n\|_{Z_\tau^{\frac{1}{2},\frac{1}{2}}}  \Big)^3  \|\Pi_\tau e^n\|_{Z_\tau^{\frac{1}{2},\frac{1}{2}}}  \|\Pi_\tau v_\tau(t_n)\|_{Z_\tau^{\frac{1}{2},\frac{1}{2}}} \notag \\
		&\qquad \quad +   T_1^\varepsilon \Big( 1+ \|\Pi_\tau v_\tau(t_n) - \Pi_\tau e^n\|_{Z_\tau^{\frac{1}{2},\frac{1}{2}}}   \Big)^3   \|\Pi_\tau v_\tau(t_n) - \Pi_\tau e^n\|_{Z_\tau^{\frac{1}{2},\frac{1}{2}}}  \|\Pi_\tau e^n\|_{Z_\tau^{\frac{1}{2},\frac{1}{2}}} \notag \\
		&\qquad\leq  C_T T_1^\varepsilon\Big(  \|e^n\|_{Z_\tau^{\frac{1}{2},\frac{1}{2}}}+\|e^n\|_{Z_\tau^{\frac{1}{2},\frac{1}{2}}}^2+\|e^n\|_{Z_\tau^{\frac{1}{2},\frac{1}{2}}}^3+\|e^n\|_{Z_\tau^{\frac{1}{2},\frac{1}{2}}}^4+\|e^n\|_{Z_\tau^{\frac{1}{2},\frac{1}{2}}}^5\Big).     \notag
	\end{align}
	Combining the above results, we arrive at
	$$
	\big\| \eta (\tfrac{t_n}{T_1} ) \big( F_\tau(v_\tau(t_{n})) -F_\tau(v^{n})  \big)  \big\|_{Z_\tau^{\frac{1}{2},-\frac{1}{2}}} 
	\leq C_T  T_1^\varepsilon \Big( \|e^n\|_{Z_\tau^{\frac{1}{2},\frac{1}{2}}} +  \|e^n\|_{Z_\tau^{\frac{1}{2},\frac{1}{2}}}^2 +  \|e^n\|_{Z_\tau^{\frac{1}{2},\frac{1}{2}}}^3 +  \|e^n\|_{Z_\tau^{\frac{1}{2},\frac{1}{2}}}^4 +  \|e^n\|_{Z_\tau^{\frac{1}{2},\frac{1}{2}}}^5 \Big),
	$$
	which together with \eqref{Eqn:error-conclusion-1} implies
	$$
	\|e^n\|_{Z_\tau^{\frac{1}{2},\frac{1}{2}}} \leq C_T  T_1^\varepsilon \Big( \|e^n\|_{Z_\tau^{\frac{1}{2},\frac{1}{2}}} +  \|e^n\|_{Z_\tau^{\frac{1}{2},\frac{1}{2}}}^2 +  \|e^n\|_{Z_\tau^{\frac{1}{2},\frac{1}{2}}}^3 +  \|e^n\|_{Z_\tau^{\frac{1}{2},\frac{1}{2}}}^4  +  \|e^n\|_{Z_\tau^{\frac{1}{2},\frac{1}{2}}}^5\Big)+ C_T \tau^{\frac{s}{2}-\frac{1}{4}}.
	$$
	By selecting $T_1$ sufficiently small and using \eqref{eq:linfty_embdis}, we conclude  that
	$$ 
	\|e^n\|_{l^\infty H^{1\over2}} \lesssim \|e^n\|_{Z_\tau^{\frac{1}{2},\frac{1}{2}}} \leq C_T \tau^{\frac{s}{2}-\frac{1}{4}},
	$$
	which together with Proposition~\ref{diffpj} actually gives  
	$$
	\|v(t_n)-v^n\|_{H^{1\over2}} \leq C_T \tau^{\frac{s}{2}-\frac{1}{4}}
	$$
	for $0\leq n \leq N_1=\lfloor {T_1\over  \tau}\rfloor$. Observing that the choice of $T_1$ depends only on $T$ and  $v_0$, we can  repeat the argument for $\lfloor{T_1\over \tau}\rfloor\leq n \leq 2\lfloor{ T_1\over \tau}\rfloor$ and so on to obtain the estimate for $0\leq m \leq {T\over \tau}$. 
	
	Using the locally bi-Lipschitz continuity of $\mathcal{G}_0$ from Lemma~\ref{lem:gauge_trafo_est} and the fact that $\mu(u)=\mu(v)$, and observing that for fixed $t$ a translation in $x$ is an isometric isomorphism on $H^s(\T)$, we have
	$$
	\begin{aligned}
		&\|u(t_n)-  \mathcal{G}^{-1}(v^n)\|_{H^{\frac{1}{2}}} = \| \mathcal{G}_0^{-1}(\tau_{\mu(v)} v(t_n)) -\mathcal{G}_0^{-1}(\tau_{\mu(v)}  v^n)\|_{H^{\frac{1}{2}}} \\
		& \qquad \lesssim \| \tau_{\mu(v)}v(t_n)  - \tau_{\mu(v)}v^n \|_{H^{\frac{1}{2}}} = \|v(t_n)-v^n\|_{H^{1\over2}} \leq C_T \tau^{\frac{s}{2}-\frac{1}{4}}.
	\end{aligned}
	$$

	%%%%%%%%%%%%%%%%%%%%%%%%%%%%%%%%%%%%%
	\section{Numerical experiments}\label{sectionnumerexp}
	%%%%%%%%%%%%%%%%%%%%%%%%%%%%%%%%%%%%%
	In this section, we perform some numerical tests to validate our main result (Theorem~\ref{Thm:main}). We take the initial data
	$$
	u_0(x)=\sum_{k\in \Z} \langle k \rangle^{-(s+\frac{1}{2}+\varepsilon)} \widehat{g}_k \mathrm{e}^{ikx}\in H^{s},
	$$
	where  $\varepsilon$ is a positive constant and $\widehat{g}_k$ are random variables uniformly distributed in the square $[-1,1]+i[-1,1]$. In our experiments, we can take $\varepsilon=0$. For spatial discretization, we use a standard Fourier pseudospectral method and choose the largest Fourier mode as $K = 2^{17}$. We normalize the $H^{1/2}$ norm of the initial data to 0.1. As a reference solution, we use the integrator proposed in \cite{ji2026low} with a small time step size, $\tau= 2^{-24}$, computed with a standard Fourier pseudospectral method and largest Fourier mode $K = 2^{17}$. We set $T = 1$ as the final time. The convergence of the filtered exponential integrator \eqref{filteredsch} for rough initial data $u_0\in H^{s}$, $s=0.7,\,1,\,1.5,\,2.5$ is presented in Figure~\ref{Fig:Hs}. The figure confirms our theoretical analysis. 
	
	\begin{figure}[h]
		\centering
		\includegraphics[width=0.44\textwidth]{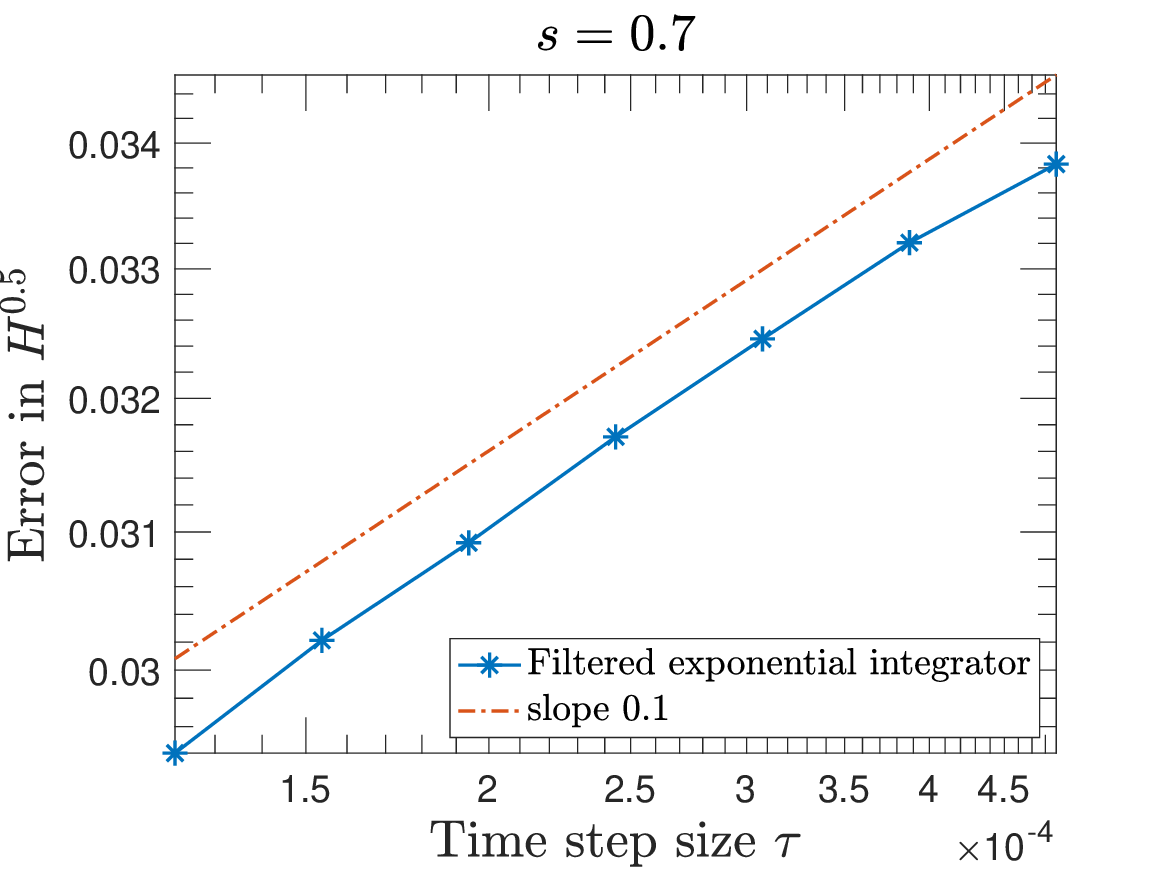}
		\includegraphics[width=0.44\textwidth]{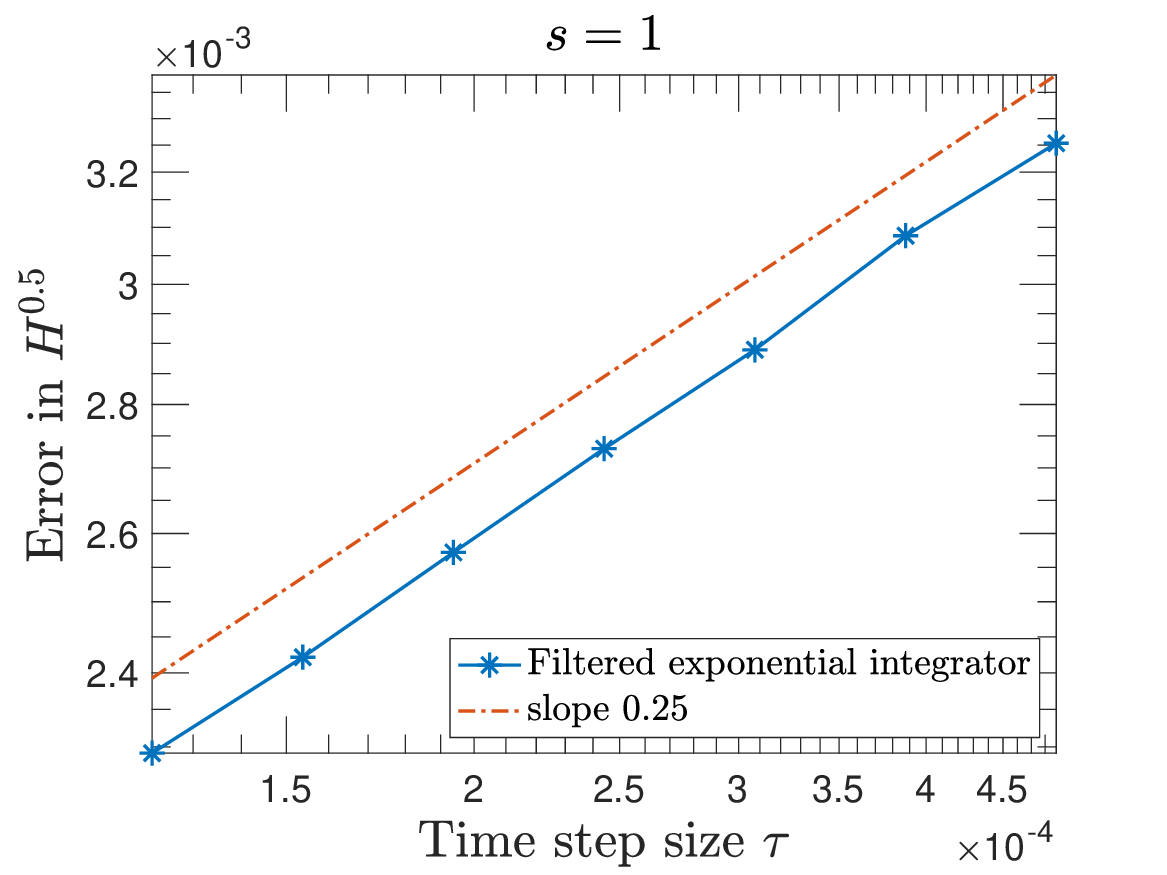} \\[2mm]
		\includegraphics[width=0.44\textwidth]{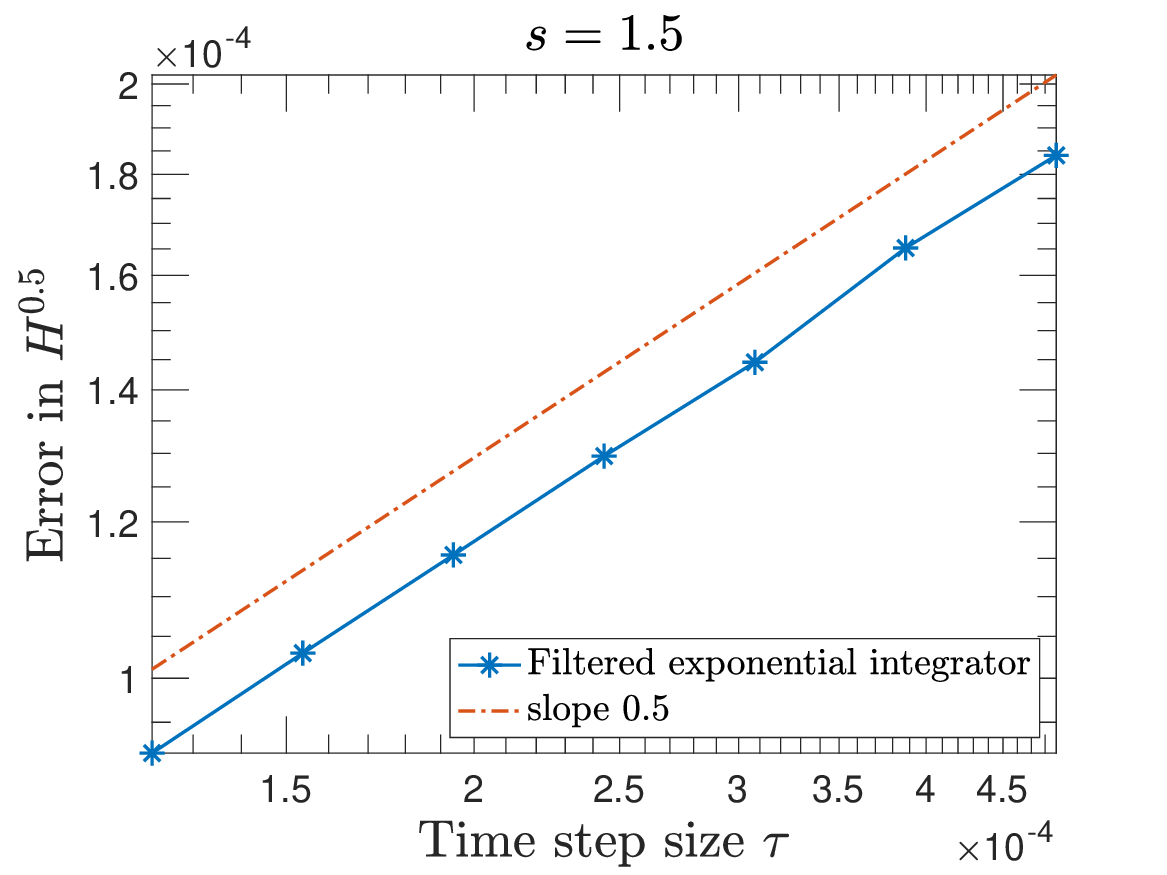}
		\includegraphics[width=0.44\textwidth]{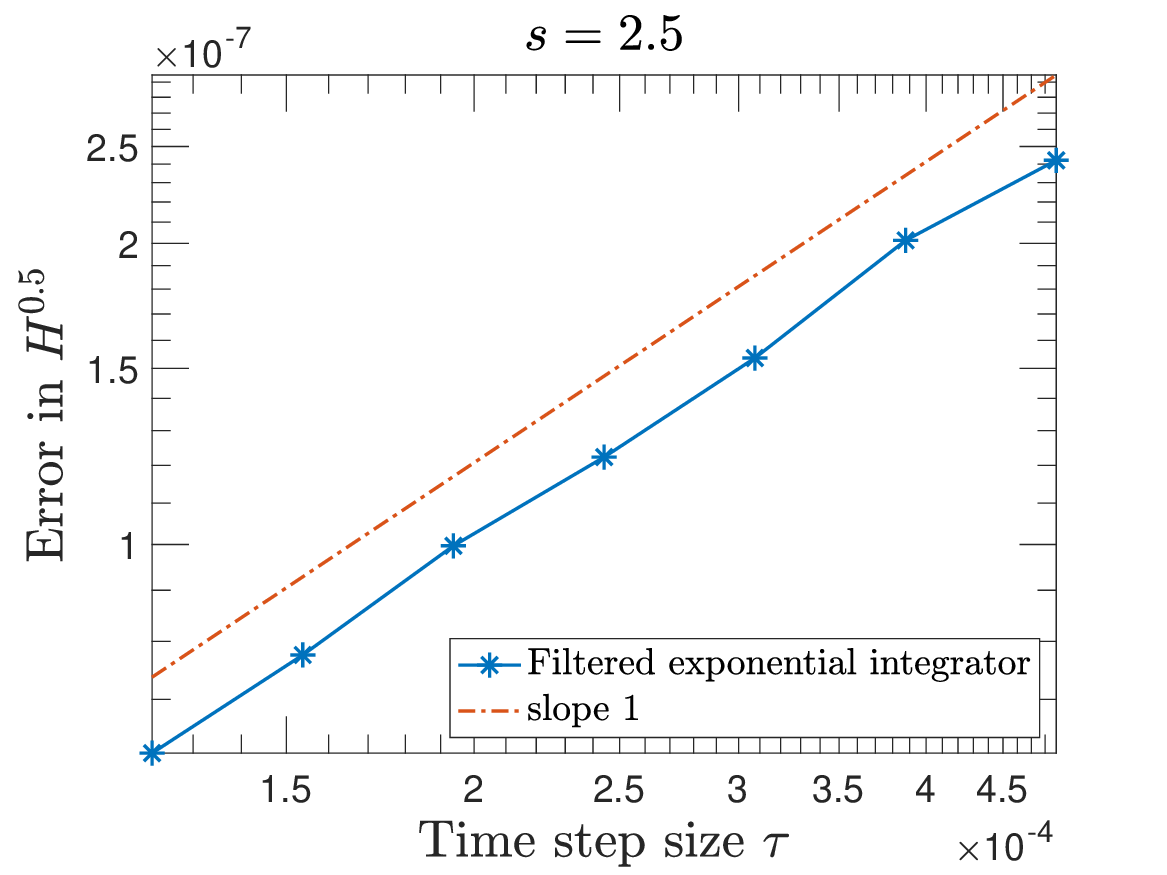}
		\caption{$H^{0.5}$ error of the filtered exponential integrator \eqref{filteredsch} for rough initial data $u_0 \in H^s$ for various values of $s$.}\label{Fig:Hs}
	\end{figure}

	Given the application of the gauge transformation \eqref{Eqn:gauge}, conservation properties, including mass conservation and energy conservation, are important for the dNLS equation. Accordingly, we conduct numerical experiments to evaluate the mass \eqref{Eqn:mass} and energy \eqref{Eqn:energy} conservation of the filtered exponential integrator \eqref{filteredsch} for rough initial data. We set the initial data in $H^1$, and take $K=2^{10}$, $\tau=2^{-12}$, $\|u_0\|_{H^{1\over 2}}=0.1$ and $T=1000$. Figure~\ref{Fig:conserve} illustrates that the filtered method exhibits robust long-term behavior at low regularity.

	\begin{figure}[h]
		\centering
		\includegraphics[width=0.44\textwidth]{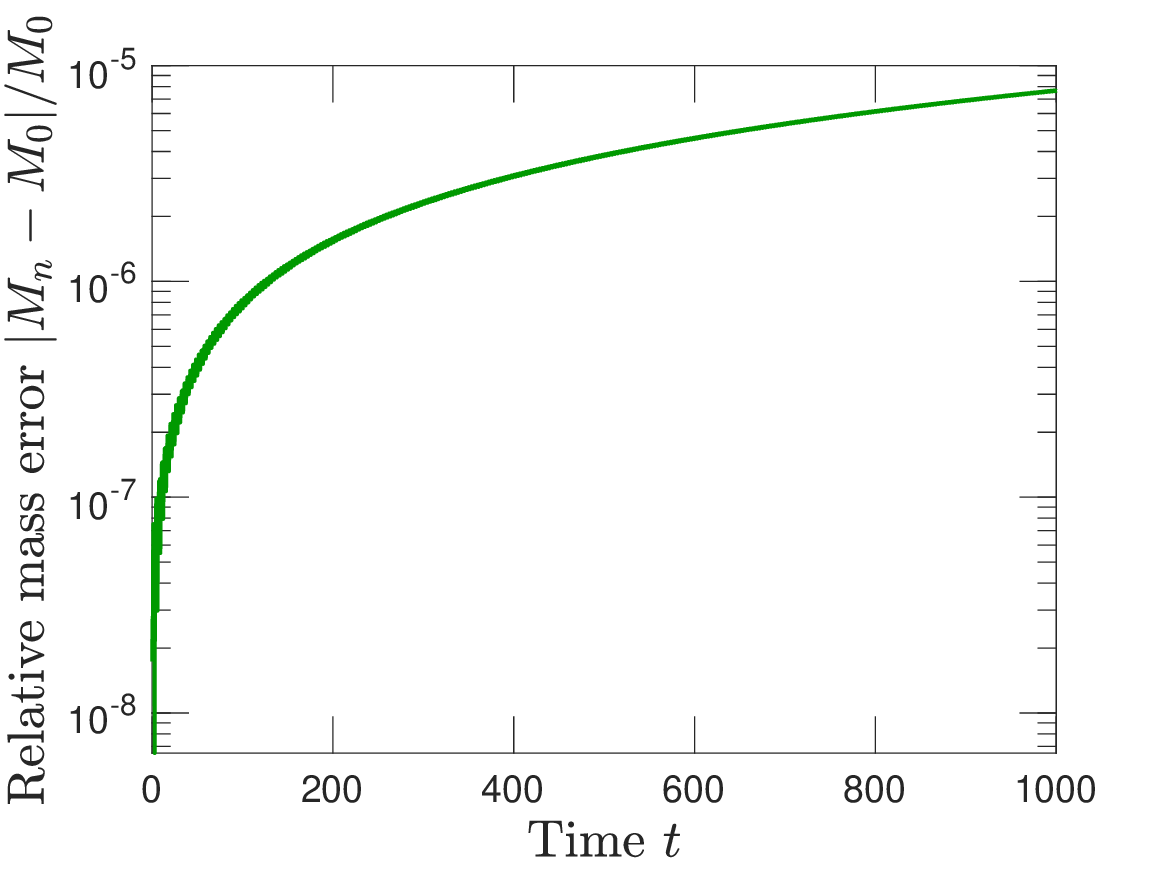}
		\includegraphics[width=0.44\textwidth]{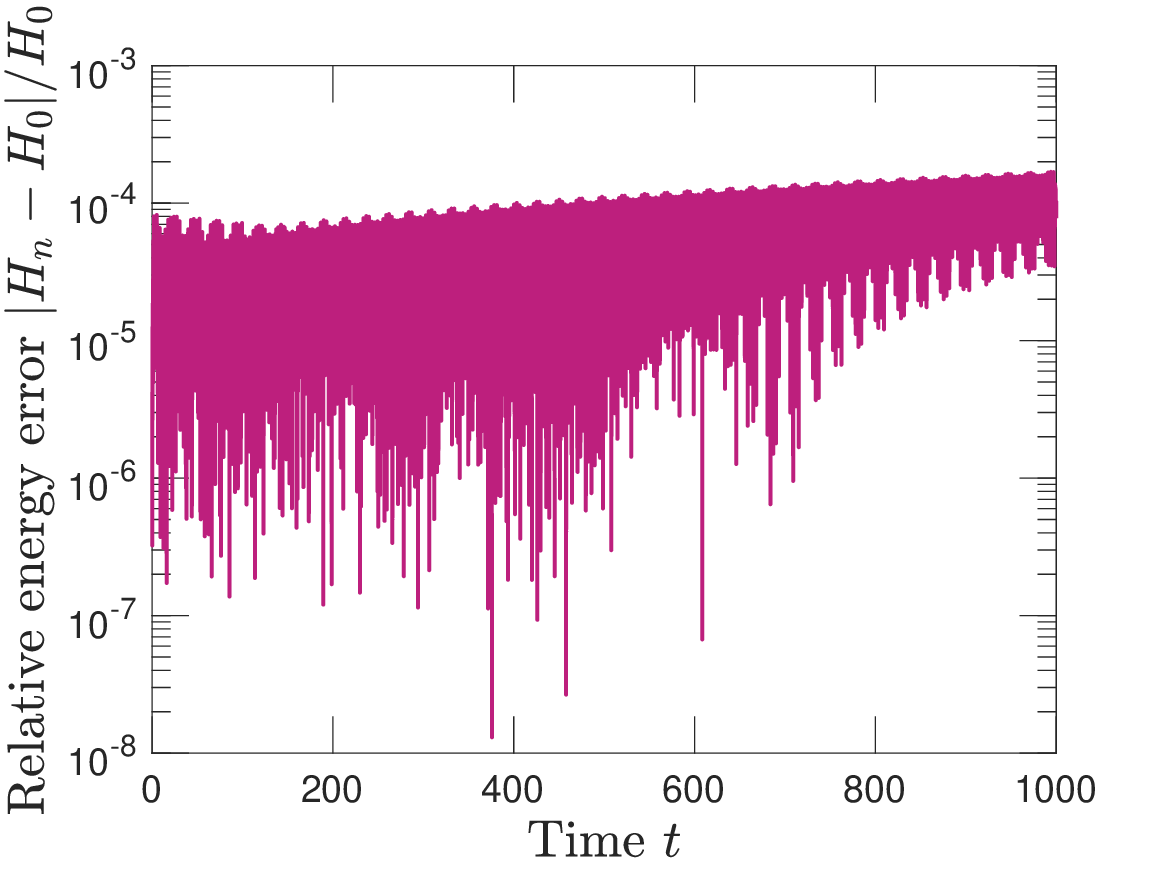}
		\caption{Conservation behavior of the filtered exponential integrator \eqref{filteredsch} for rough initial data $u_0 \in H^1$. Left: relative mass error; Right: relative energy error.}\label{Fig:conserve}
	\end{figure}

	\appendix
	%%%%%%%%%%%%%%%%%%%%%%%%%%%
	\section{Multilinear estimates}\label{multilinear}
	%%%%%%%%%%%%%%%%%%%%%%%%%%%
	We first present some useful embedding theorems.
	\begin{lemma}\label{sobolevdiscrete}
		For a sequence $\{v^n\}_{n\in \Z}$, we have
		\begin{align}
			\|v^n\|_{l^p_\tau H^s}&\lesssim \|v^n\|_{X_\tau^{s,b}},\quad 2 \leq p<\infty, ~ b>
			\tfrac{1}{2}-\tfrac{1}{p}, ~ s \in \R, \label{eq:sobdis}\\
			\|v^n\|_{l^p_\tau L^q}&\lesssim \|v^n\|_{X_\tau^{s,b}},\quad 2 \leq p,q<\infty, ~  b>
			\tfrac{1}{2}-\tfrac{1}{p},~ s>   \tfrac{1}{2}-\tfrac{1}{q},  \label{eq:sob2dis}\\
			\|v^n\|_{X_\tau^{s,b}} & \lesssim
			\|v^n\|_{l^p_\tau H^s},\quad 1 < p \leq 2 ,~ b> \tfrac{1}{2}-\tfrac{1}{p},~ s\in \R,\label{eq:dual_sobdis}\\
			\|v^n\|_{Y_\tau^{s,b^\prime}}  & \lesssim  \|v^n\|_{X_\tau^{s,b}},\quad  b>b^\prime+\tfrac{1}{2}, \label{eq:XtautoYtao} \\
			\|v^n\|_{l_\tau^{\infty}H^s} &\lesssim \|v^n\|_{Z_\tau^{s, \frac{1}{2}}},\quad  s\in \R. \label{eq:linfty_embdis}
		\end{align}
	\end{lemma}
	\begin{proof}
		We set $f^n =\langle \partial_x \rangle^s \e^{-i n \tau \partial_x^2}  v^n$. According to \eqref{Xtausb}, it holds $\|v^n\|_{X_\tau^{s,b}} \sim \|f^n\|_{H_\tau^b L^2}$.
		For \eqref{eq:sobdis}, since $\|f^n\|_{l_\tau^p L^2} = \|v^n\|_{l_\tau^p H^s}$, it suffices to show $\|f^n\|_{l_\tau^p L^2} \lesssim \|f^n\|_{H_\tau^b L^2}$. Observing that $\widehat{f^m}(k)=\frac{1}{2\pi}\int_{-\frac{\pi}{\tau}}^{\frac{\pi}{\tau}} \widetilde{f^n}(\sigma,k)\e^{ i m\tau \sigma}\dd \sigma$ and using the Cauchy--Schwarz inequality, we obtain that, for each $m\in \Z$,
		$$
		\begin{aligned}
			\|f^m\|_{L^2} &  \lesssim  \|\widehat{f^m}(k)\|_{l^2(k)}   \lesssim \int_{-\frac{\pi}{\tau}}^{\frac{\pi}{\tau}} \|   \widetilde{f^n}(\sigma,k)\|_{l^2(k)} \dd \sigma
			\\ 
			&
			\lesssim \|\langle d_\tau (\sigma) \rangle^{-\frac{1}{2}-\varepsilon}\|_{L^2({\sigma})} \| \langle d_\tau (\sigma) \rangle^{\frac{1}{2}+\varepsilon} \widetilde{f^n}(\sigma,k)\|_{L^2l^2}\lesssim \|f^n\|_{H_\tau^{\frac{1}{2}+\varepsilon} L^2},
		\end{aligned}
		$$
		which yields $\|f^n\|_{l^\infty L^2}\lesssim \|f^n\|_{H_\tau^{\frac{1}{2}+\varepsilon} L^2}$ with arbitrarily small $\varepsilon>0$. Moreover, it holds $\|f^n\|_{l_\tau^2 L^2}\lesssim \|f^n\|_{H_\tau^0 L^2}$. Interpolating between the above two embeddings leads to the desired result. Subsequently, \eqref{eq:sob2dis} follows from the combination of \eqref{eq:sobdis} and $H^{s}\subset L^q$ ($s>   \tfrac{1}{2}-\tfrac{1}{q}$). The relation \eqref{eq:dual_sobdis} follows by duality from \eqref{eq:sobdis}. 
		For the proof of \eqref{eq:XtautoYtao}, using the Cauchy--Schwarz inequality, we obtain
		$$
		\begin{aligned}
			\|v^n\|_{Y_\tau^{s,b^\prime}}^2 & = 
			\sum_{k\in \Z}\left(\int_{-\frac{\pi}{\tau}}^{\frac{\pi}{\tau}} \lb k\rb^s\lb d_\tau(\sigma+k^2)\rb^{b^\prime} |\widetilde{v^n}(\sigma, k)| \dd \sigma \right)^2 \\
			& \leq \sum_{k\in \Z}\int_{-\frac{\pi}{\tau}}^{\frac{\pi}{\tau}} \lb d_\tau(\sigma+k^2)\rb^{2b^\prime-2b} \dd \sigma \int_{-\frac{\pi}{\tau}}^{\frac{\pi}{\tau}}  \lb k\rb^{2s}\lb d_\tau(\sigma+k^2)\rb^{2b} |\widetilde{v^n}(\sigma, k)|^2 \dd \sigma.
		\end{aligned} 
		$$
		Since $d_{\tau}$ is $2\pi/\tau$ periodic and satisfies $|d_{\tau}(\sigma)| \sim | \sigma |$ for $|\tau \sigma | \leq \pi$, we have, for each $k\in\Z$,
		$$
		\int_{-\frac{\pi}{\tau}}^{\frac{\pi}{\tau}} \lb d_\tau(\sigma+k^2)\rb^{2b^\prime-2b} \dd \sigma=\int_{-\frac{\pi}{\tau}}^{\frac{\pi}{\tau}} \lb d_\tau(\sigma)\rb^{2b^\prime-2b} \dd \sigma\lesssim \int_{-\frac{\pi}{\tau}}^{\frac{\pi}{\tau}}  \lb  \sigma \rb^{2b^\prime-2b} \dd \sigma\lesssim 1,
		$$
		which gives the desired result. Finally, we prove \eqref{eq:linfty_embdis}. For each $m\in \Z$, we have
		$$
		|\widehat{f^m}(k)|= \bigg|\frac{1}{2\pi}\int_{-\frac{\pi}{\tau}}^{\frac{\pi}{\tau}} \widetilde{f^n}(\sigma,k)\e^{i m\tau \sigma}\dd \sigma\bigg|\leq \frac{1}{2\pi} \|\widetilde{f^n}(\sigma,k)\|_{L^1(\sigma)}.
		$$
		Taking the $l^2$ norm in $k$ yields
		$$
		\|v^m\|_{H^s}= \|f^m\|_{L^2} \lesssim \|\widetilde{f^n}(\sigma,k)\|_{L^1l^2} \leq \|v^n\|_{Z_\tau^{s, \frac{1}{2}}},
		$$
		which completes the proof.
	\end{proof}

	For the multilinear estimates, we follow the proof of the continuous case in \cite{Herr2006}. We begin with an elementary bound for the multilinear multiplier.
	
	\begin{lemma}\label{Lem:multiplier1}
		Let $(\sigma,k)=\sum_{i=1}^3 (\sigma_i,k_i)$ with $\sigma_i\in \R$ and $k_i \in \Z$. Define
		$$
		\begin{aligned}
			&M(\sigma_1,\sigma_2,\sigma_3,k_1,k_2,k_3) \\
			&\qquad=\frac{\langle k \rangle^{\frac{1}{2}} i k_3}{\langle k_1 \rangle^{\frac{1}{2}} \langle k_2 \rangle^{\frac{1}{2}} \langle k_3 \rangle^{\frac{1}{2}}\langle d_\tau(\sigma+k^2) \rangle^{\frac{1}{2}} \langle d_\tau(\sigma_1+k_1^2) \rangle^{\frac{1}{2}} \langle d_\tau(\sigma_2+k_2^2) \rangle^{\frac{1}{2}} \langle d_\tau(\sigma_3-k_3^2) \rangle^{\frac{1}{2}}},
		\end{aligned}
		$$
		and
		\begin{align}
			M_0(\sigma_1,\sigma_2,\sigma_3,k_1,k_2,k_3)
			&=\frac{\chi_{A_0}}{\langle k_1 \rangle^{\frac{1}{2}} \langle k_2 \rangle^{\frac{1}{2}} \langle d_\tau(\sigma_1+k_1^2) \rangle^{\frac{1}{2}} \langle d_\tau(\sigma_2+k_2^2) \rangle^{\frac{1}{2}} \langle d_\tau(\sigma_3-k_3^2) \rangle^{\frac{1}{2}}}, \notag \\
			M_1(\sigma_1,\sigma_2,\sigma_3,k_1,k_2,k_3)
			&=\frac{\chi_{A_1}}{\langle k_1 \rangle^{\frac{1}{2}} \langle k_2 \rangle^{\frac{1}{2}}  \langle d_\tau(\sigma+k^2) \rangle^{\frac{1}{2}}   \langle d_\tau(\sigma_2+k_2^2) \rangle^{\frac{1}{2}} \langle d_\tau(\sigma_3-k_3^2) \rangle^{\frac{1}{2}}},\notag\\
			M_2(\sigma_1,\sigma_2,\sigma_3,k_1,k_2,k_3)
			&=\frac{\chi_{A_2}}{\langle k_1 \rangle^{\frac{1}{2}} \langle k_2 \rangle^{\frac{1}{2}}  \langle d_\tau(\sigma+k^2) \rangle^{\frac{1}{2}}   \langle d_\tau(\sigma_1+k_1^2) \rangle^{\frac{1}{2}} \langle d_\tau(\sigma_3-k_3^2) \rangle^{\frac{1}{2}}},\notag\\
			M_3(\sigma_1,\sigma_2,\sigma_3,k_1,k_2,k_3)
			&=\frac{\chi_{A_3}}{\langle k_1 \rangle^{\frac{1}{2}} \langle k_2 \rangle^{\frac{1}{2}}  \langle d_\tau(\sigma+k^2) \rangle^{\frac{1}{2}}   \langle d_\tau(\sigma_1+k_1^2) \rangle^{\frac{1}{2}}\langle d_\tau(\sigma_2+k_2^2) \rangle^{\frac{1}{2}} },\notag\\
			N(\sigma_1,\sigma_2,\sigma_3,k_1,k_2,k_3)
			&=\frac{1}{\langle d_\tau(\sigma+k^2) \rangle^{\frac{1}{2}} \langle d_\tau(\sigma_1+k_1^2) \rangle^{\frac{1}{2}} \langle d_\tau(\sigma_2+k_2^2) \rangle^{\frac{1}{2}} \langle d_\tau(\sigma_3-k_3^2) \rangle^{\frac{1}{2}}},\notag
		\end{align}
		where $\chi_{A_j}$ for $j=0,1,2,3$ is the characteristic function of the subregion $A_j\subset \R^3\times \Z^3$, where 
		$$
		\begin{aligned}
			A_0:~\langle d_\tau(\sigma+k^2)\rangle &  \geq \langle d_\tau(\sigma_1+k_1^2)\rangle,\,\langle d_\tau(\sigma_2+k_2^2)\rangle,\,\langle d_\tau(\sigma_3-k_3^2)\rangle ,\\
			A_1: \langle d_\tau(\sigma_1+k_1^2)\rangle &  \geq \langle d_\tau(\sigma+k^2)\rangle,\,\langle d_\tau(\sigma_2+k_2^2)\rangle,\,\langle d_\tau(\sigma_3-k_3^2)\rangle , \\
			A_2: \langle d_\tau(\sigma_2+k_2^2)\rangle  & \geq \langle d_\tau(\sigma+k^2)\rangle,\,\langle d_\tau(\sigma_1+k_1^2)\rangle,\,\langle d_\tau(\sigma_3-k_3^2)\rangle ,\\
			A_3: \langle d_\tau(\sigma_3-k_3^2)\rangle &  \geq \langle d_\tau(\sigma+k^2)\rangle,\,\langle d_\tau(\sigma_1+k_1^2)\rangle,\,\langle d_\tau(\sigma_2+k_2^2)\rangle , \\
		\end{aligned}
		$$
		If $|k|,|k_j|\leq \tau^{-\frac{1}{2}}$, then the estimate
		$$
		|M| \lesssim M_0+M_1+M_2+M_3+N
		$$
		holds true.
	\end{lemma}
	\begin{proof}
		Observing that $\sigma+k^2-(\sigma_1+k_1^2+\sigma_2+k_2^2+\sigma_3-k_3^2)=2(k-k_1)(k-k_2)$, we have
		$$
		|d_\tau ( (k-k_1)(k-k_2) )| \leq \bigg|d_\tau \bigg(\frac{\sigma+k^2}{2}\bigg)\bigg| + \bigg|d_\tau \bigg(\frac{\sigma_1+k_1^2}{2} \bigg)\bigg|+\bigg|d_\tau \bigg(\frac{\sigma_2+k_2^2}{2}\bigg)\bigg| +\bigg|d_\tau \bigg(\frac{\sigma_3-k_3^2}{2}\bigg)\bigg|.
		$$
		Thanks to $|k|,|k_j|\leq \tau^{-\frac{1}{2}}$, we derive that 
		$$
		|(k-k_1)(k-k_2)|\leq 4\tau^{-1}\quad \mbox{and}\quad \bigg|\frac{\sigma+k^2}{2}\bigg|,\bigg| \frac{\sigma_j\pm k_j^2}{2} \bigg|\leq \frac{\pi+1}{2\tau}~ (j=1,2,3).
		$$ 
		In addition, we note that $\frac{\sin x}{x}$ is strictly decreasing over $[0,2]$. Then we have
		\begin{align}
			|d_\tau ( (k-k_1)(k-k_2) )|  
			&  = |(k-k_1)(k-k_2)| \cdot \frac{2}{\tau|(k-k_1)(k-k_2)|}\cdot\bigg|\sin\bigg(\frac{\tau(k-k_1)(k-k_2)}{2}\bigg)\bigg| \notag \\
			&  \geq \frac{\sin 2}{2}|(k-k_1)(k-k_2)| , \notag\\
			\bigg|d_\tau \bigg(\frac{\sigma+k^2}{2}\bigg)\bigg|  
			&  =\frac{2}{\tau} \bigg| \sin \bigg(\frac{\tau (\sigma+k^2)}{4}\bigg) \bigg| \leq  \frac{2}{\tau} \cdot 2\bigg| \sin \bigg(\frac{\tau (\sigma+k^2)}{4}\bigg) \bigg|  \cos \bigg(\frac{\tau (\sigma+k^2)}{4}\bigg) \notag\\
			&  = \frac{2}{\tau} \bigg| \sin \bigg(\frac{\tau (\sigma+k^2)}{2}\bigg) \bigg| =|d_\tau(\sigma+k^2)|, \notag\\
			\bigg|d_\tau \bigg(\frac{\sigma_j\pm k_j^2}{2}\bigg)\bigg|  & \leq  |d_\tau (\sigma_j\pm k_j^2)|, \quad \mbox{for}~j=1,2,3. \label{Eqn:dtau/2<=dtau}
		\end{align}
		It follows that
		$$
		\begin{aligned}
			\langle (k-k_1)&(k-k_2) \rangle^{\frac{1}{2}}    \leq \frac{\sqrt{2}}{\sqrt{\sin 2}}  \langle d_\tau ( (k-k_1)(k-k_2) ) \rangle^{\frac{1}{2}}  \\
			& \leq \frac{3}{2}(\langle d_\tau(\sigma+k^2) \rangle^{\frac{1}{2}} +\langle d_\tau(\sigma_1+k_1^2) \rangle^{\frac{1}{2}} +\langle d_\tau(\sigma_2+k_2^2) \rangle^{\frac{1}{2}} +\langle d_\tau(\sigma_3-k_3^2) \rangle^{\frac{1}{2}} )\\
			& \leq 6(\chi_{A_0}\langle d_\tau(\sigma+k^2) \rangle^{\frac{1}{2}} +\chi_{A_1}\langle d_\tau(\sigma_1+k_1^2) \rangle^{\frac{1}{2}} +\chi_{A_2}\langle d_\tau(\sigma_2+k_2^2) \rangle^{\frac{1}{2}} +\chi_{A_3}\langle d_\tau(\sigma_3-k_3^2) \rangle^{\frac{1}{2}} ) \\
			&=:6   (\chi_{A_0}L_0+\chi_{A_1}L_1+\chi_{A_2}L_2+\chi_{A_3}L_3).
		\end{aligned}
		$$
		
		The analysis is divided into four distinct cases.
		
		\medskip
		
		(1)  $|k|>2|k_1|$ and $|k|>2|k_2|$: In this case we have $|k_3|\leq 2|k|$ and $\langle(k-k_1)(k-k_2)\rangle\geq \frac{1}{4}\langle k \rangle^2$. It follows that
		$$
		\langle k \rangle \leq 2\langle(k-k_1)(k-k_2)\rangle^{\frac{1}{2}} \leq 12\max_{j=0,1,2,3} L_j.
		$$
		This further gives, for $(\sigma_1,\sigma_2,\sigma_3,k_1,k_2,k_3)\in A_0$,
		$$
		|M|=\frac{\langle k \rangle^{\frac{1}{2}} |k_3|}{\langle k_1 \rangle^{\frac{1}{2}} \langle k_2 \rangle^{\frac{1}{2}} \langle k_3 \rangle^{\frac{1}{2}} L_0  L_1 L_2 L_3} 
		\leq \frac{12 |k_3|}{\langle k_1 \rangle^{\frac{1}{2}} \langle k_2 \rangle^{\frac{1}{2}} \langle k_3 \rangle^{\frac{1}{2}} \langle k \rangle^{\frac{1}{2}} L_1 L_2 L_3} \lesssim M_0 \lesssim \sum_{j=0}^3 M_j.
		$$
		Following an analogous reasoning for $(\sigma_1,\sigma_2,\sigma_3,k_1,k_2,k_3)\in A_j$ ($j=1,2,3$), we arrive at $|M|\lesssim \sum_{j=0}^3 M_j$.
		
		\medskip
		
		(2) $|k|\leq 2 |k_1|$ and $|k|\leq 2|k_2|$: In this case we can verify that $|k_3|\leq 4\max\{ |k_1|,|k_2|\}$ and $|k|\leq 2\min\{|k_1|,|k_2|\}$. This implies $|{M}|\lesssim {N}$.
		
		\medskip
		
		In the remaining two cases: (3) $|k|>2|k_1|$ and $|k|\leq 2|k_2|$; as well as (4) $|k|\leq 2|k_1|$ and $|k|> 2|k_2|$, we can proceed in the same way as in \cite{Herr2006} to show that  
		$$
		|M| \lesssim M_0+M_1+M_2+M_3+N.
		$$
		This completes the proof.
	\end{proof}
	
	Inspired by \cite{Herr2006}, we shall use the abbreviation
	$$
	\begin{aligned}
		&\int_{*} \sum_{*} \prod_{j=1}^m f_j(\sigma_j, k_j) \\
		&\qquad:= \underbrace{\int_{-\frac{\pi}{\tau}}^{\frac{\pi}{\tau}} \dots \int_{-\frac{\pi}{\tau}}^{\frac{\pi}{\tau}}}_{m-1} \sum_{k_1, \dots, k_{m-1}}  \prod_{j=1}^{m-1} f_j(\sigma_j, k_j)   f_m\bigg(\sigma - \sum_{j=1}^{m-1} \sigma_j, k - \sum_{j=1}^{m-1} k_j\bigg) \dd\sigma_1 \dots \dd\sigma_{m-1},
	\end{aligned}
	$$
	which is actually the convolution $f_1 * \ldots * f_m(\sigma,k)$. In particular, for given sequences of functions $\{v^n_j\}_{n\in\Z}$, $j=1,2,3$, we have
	$$
	\begin{aligned}
		\widetilde{v_1^nv_2^n}(\sigma,k) &= \tau\sum_{m \in \Z}\widehat{v_1^mv_2^m}(k)\e^{-i m \tau \sigma} = \tau \sum_{m\in\Z}\sum_{k_1\in\Z} \widehat{v_1^m}(k_1)\widehat{v_2^m}(k-k_1) \e^{- im\tau \sigma}  \\
		& = \tau \sum_{m\in\Z}\sum_{k_1\in\Z} \bigg(\frac{1}{2\pi}
		\int_{-\frac{\pi}{\tau}}^{\frac{\pi}{\tau}} \widetilde{v_1^n}(\sigma_1,k_1)\e^{im\tau\sigma_1}\dd\sigma_1
		\bigg) \widehat{v_2^m}(k-k_1) \e^{- im\tau \sigma}  \\
		& =\frac{1}{2\pi} \sum_{k_1\in\Z} 
		\int_{-\frac{\pi}{\tau}}^{\frac{\pi}{\tau}} \widetilde{v_1^n}(\sigma_1,k_1) \bigg( \tau \sum_{m\in\Z} \widehat{v_2^m}(k-k_1) \e^{- im\tau (\sigma-\sigma_1)}  \bigg)
		\dd\sigma_1 \\
		& =\frac{1}{2\pi} \sum_{k_1\in \Z} \int_{-\frac{\pi}{\tau}}^{\frac{\pi}{\tau}} \widetilde{v_1^n}(\sigma_1,k_1) \widetilde{v_2^n}(\sigma-\sigma_1,k-k_1)\dd \sigma_1 = \frac{1}{2\pi} 
		\int_{*} \sum_{*} \prod_{j=1}^2 \widetilde{v_j^n}(\sigma_j, k_j) ,\\
		\widetilde{v_1^nv_2^nv_3^n}(\sigma,k) & =  \frac{1}{(2\pi)^2} 
		\int_{*} \sum_{*} \prod_{j=1}^3 \widetilde{v_j^n}(\sigma_j, k_j).
	\end{aligned}
	$$

	\begin{theorem}\label{Thm:multiplier1}
		For all sequences $\{v^n_j\}_{n \in \mathbb{Z}}$ $(j = 1,2,3)$ supported on the time grid $n\tau \in [-2T_1, 2T_1]$ for $\tau,T_1  \in (0,1]$, there exists $\varepsilon>0$ such that
		$$
		\|\Pi_\tau(\Pi_\tau v_1^n\Pi_\tau v_2^n \, \partial_x \Pi_\tau v_3^n)\|_{X_\tau^{\frac{1}{2},-\frac{1}{2}}}    \lesssim  T_1^\varepsilon \|v_1^n\|_{X_\tau^{\frac{1}{2},\frac{1}{2}}} \|v_2^n\|_{X_\tau^{\frac{1}{2},\frac{1}{2}}} \|\overline{v}_3^n\|_{X_\tau^{\frac{1}{2},\frac{1}{2}}}.
		$$
	\end{theorem}
	\begin{proof}
		We define 
		$$
		\begin{aligned}
			f_1(\sigma_1,k_1)& =  \langle d_\tau(\sigma_1+k_1^2) \rangle^{\frac{1}{2}} \langle k_1 \rangle^{\frac{1}{2}} \widetilde{\Pi_\tau v^n_1}(\sigma_1,k_1), \\
			f_2(\sigma_2,k_2)& =  \langle d_\tau(\sigma_2+k_2^2) \rangle^{\frac{1}{2}} \langle k_2 \rangle^{\frac{1}{2}} \widetilde{\Pi_\tau v^n_2}(\sigma_2,k_2), \\
			f_3(\sigma_3,k_3)& =  \langle d_\tau(\sigma_3-k_3^2) \rangle^{\frac{1}{2}} \langle k_3 \rangle^{\frac{1}{2}} \widetilde{\Pi_\tau v^n_3}(\sigma_3,k_3), \\
		\end{aligned}
		$$
		and note that 
		$$
		\|\Pi_\tau(\Pi_\tau v_1^n\Pi_\tau v_2^n \partial_x \Pi_\tau v_3^n)\|_{X_\tau^{\frac{1}{2},-\frac{1}{2}}}   =\bigg\|\frac{{1}_{|k| \leq \tau^{-1/2}}}{(2\pi)^{2}} \int_*\sum_* M(\sigma_1,\sigma_2,\sigma_3,k_1,k_2,k_3)\prod_{j=1}^3 f_j(\sigma_j,k_j)\bigg\|_{L^2l^2}. 
		$$
		Thanks to Lemma~\ref{Lem:multiplier1}, it suffices to estimate the above term with $f_j$ replaced by $|f_j|$ and $M$ replaced by $M_l$ ($l=0,1,2,3$) and $N$.
		
		We first consider the estimate for $M_0$:
		$$
		\bigg\|\frac{{1}_{|k| \leq \tau^{-1/2}}}{(2\pi)^{2}} \int_*\sum_* M_0(\sigma_1,\sigma_2,\sigma_3,k_1,k_2,k_3)\prod_{j=1}^3 |f_j(\sigma_j,k_j)|\bigg\|_{L^2l^2}=:J(M_0).
		$$
		By H\"older's inequality and using \eqref{eq:sob2dis} and \eqref{Eqn:discreteL4bound}, we get
		$$
		\begin{aligned}
			J(M_0)& = \|\Pi_\tau \big(\mathcal{F}_{\tau,x}^{-1} ( |\widetilde{\Pi_\tau v_1^n}|) \mathcal{F}_{\tau,x}^{-1} ( |\widetilde{\Pi_\tau v_2^n}|)  \mathcal{F}_{\tau,x}^{-1} (\langle k_3\rangle^{\frac{1}{2}} |\widetilde{\Pi_\tau v_3^n}|)  \big)\|_{l_\tau^2L^2} \\
			& \lesssim  \|\mathcal{F}_{\tau,x}^{-1} ( |\widetilde{\Pi_\tau v_1^n}|)\|_{l_\tau^8L^8}\|\mathcal{F}_{\tau,x}^{-1} ( |\widetilde{\Pi_\tau v_2^n}|)  \|_{l_\tau^8L^8}\| \mathcal{F}_{\tau,x}^{-1} (\langle k_3\rangle^{\frac{1}{2}} |\widetilde{\Pi_\tau \overline{v}_3^n}|)\|_{l_\tau^4L^4}  \\
			&\lesssim  \|v_1^n\|_{X_\tau^{\frac{1}{2},\frac{7}{16}}} \|v_2^n\|_{X_\tau^{\frac{1}{2},\frac{7}{16}}} \|\overline{v}_3^n\|_{X_\tau^{\frac{1}{2},\frac{3}{8}}} .
		\end{aligned}
		$$
		We recall that  $\{v_j^n\}_{n\in \Z}$ is supported on the time grid $n\tau \in [-2T_1,2T_1]$ and that $\eta$ is a smooth and compactly supported function, which is one on $[-1,1]$ and supported in $[-2,2]$. Therefore, we have
		$$
		\begin{aligned}
			\|v_1^n\|_{X_\tau^{\frac{1}{2},\frac{7}{16}}} \|v_2^n\|_{X_\tau^{\frac{1}{2},\frac{7}{16}}} \|\overline{v}_3^n\|_{X_\tau^{\frac{1}{2},\frac{3}{8}}} 
			& =\|\eta(\tfrac{t_n}{2T_1}) v_1^n\|_{X_\tau^{\frac{1}{2},\frac{7}{16}}} \|\eta(\tfrac{t_n}{2T_1})v_2^n\|_{X_\tau^{\frac{1}{2},\frac{7}{16}}} \|\eta(\tfrac{t_n}{2T_1}) \overline{v}_3^n\|_{X_\tau^{\frac{1}{2},\frac{3}{8}}} \\
			&\lesssim T_1^{\varepsilon} \|v_1^n\|_{X_\tau^{\frac{1}{2},\frac{1}{2}}} \|v_2^n\|_{X_\tau^{\frac{1}{2},\frac{1}{2}}}  \|\overline{v}_3^n\|_{X_\tau^{\frac{1}{2},\frac{1}{2}}},
		\end{aligned}
		$$	
		where \eqref{Eqn:T-inverse} was used in the last estimate. Combining the above inequalities gives
		$$
		J(M_0)\lesssim T_1^{\varepsilon} \|v_1^n\|_{X_\tau^{\frac{1}{2},\frac{1}{2}}} \|v_2^n\|_{X_\tau^{\frac{1}{2},\frac{1}{2}}}  \|\overline{v}_3^n\|_{X_\tau^{\frac{1}{2},\frac{1}{2}}}.
		$$
		
		We now turn to the estimates for $M_l$ ($l=1,2,3$):
		$$
		\bigg\|\frac{{1}_{|k| \leq \tau^{-1/2}}}{(2\pi)^{2}} \int_*\sum_* M_l(\sigma_1,\sigma_2,\sigma_3,k_1,k_2,k_3)\prod_{j=1}^3 |f_j(\sigma_j,k_j)|\bigg\|_{L^2l^2}=:J(M_l).
		$$
		With the help of \eqref{eq:dual_sobdis}, \eqref{eq:sob2dis}, \eqref{Eqn:discreteL4bound} and \eqref{Eqn:discreteX-38bound}, we obtain
		\begin{align}
			J(M_1) & \lesssim \big\|\mathcal{F}_{\tau,x}^{-1} ( \langle d_\tau(\sigma_1+k_1^2) \rangle^{\frac{1}{2}}  |\widetilde{\Pi_\tau v^n_1}|  )  \mathcal{F}_{\tau,x}^{-1} ( |\widetilde{\Pi_\tau v_2^n}|)  \mathcal{F}_{\tau,x}^{-1}  ( \langle k_3\rangle^{\frac{1}{2}}|\widetilde{\Pi_\tau v_3^n|} )\big\|_{X_{\tau}^{0,-\frac{3}{8}}} \notag \\
			& \lesssim \big\|\mathcal{F}_{\tau,x}^{-1} ( \langle d_\tau(\sigma_1+k_1^2) \rangle^{\frac{1}{2}}  |\widetilde{\Pi_\tau v^n_1}|  )  \mathcal{F}_{\tau,x}^{-1} ( |\widetilde{\Pi_\tau v_2^n}|)   \mathcal{F}_{\tau,x}^{-1} ( \langle k_3\rangle^{\frac{1}{2}}|\widetilde{\Pi_\tau v_3^n|} )\big\|_{l_\tau^{\frac{15}{14}}L^2}  \notag \\
			& \lesssim \big\|\mathcal{F}_{\tau,x}^{-1} ( \langle d_\tau(\sigma_1+k_1^2) \rangle^{\frac{1}{2}}  |\widetilde{\Pi_\tau v^n_1}|  )\big\|_{ l_\tau^{2}L^8}  \big\| \mathcal{F}_{\tau,x}^{-1} ( |\widetilde{\Pi_\tau v_2^n}|)   \mathcal{F}_{\tau,x}^{-1} ( \langle k_3\rangle^{\frac{1}{2}}|\widetilde{\Pi_\tau v_3^n|} )\big\|_{ l_\tau^{\frac{30}{13}}L^{\frac{8}{3}}} \notag \\
			& \lesssim \big\|\mathcal{F}_{\tau,x}^{-1} ( \langle d_\tau(\sigma_1+k_1^2) \rangle^{\frac{1}{2}}  |\widetilde{\Pi_\tau v^n_1}|  )\big\|_{l_\tau^2 H^{\frac{1}{2}}}  \|\mathcal{F}_{\tau,x}^{-1} ( |\widetilde{\Pi_\tau v_2^n}|) \|_{l_\tau^{\frac{60}{11}} L^8} \big\|\mathcal{F}_{\tau,x}^{-1} ( \langle k_3\rangle^{\frac{1}{2}}|\widetilde{\Pi_\tau \overline{v}_3^n|} )\big\|_{l_\tau^{4}L^{4}} \notag \\
			& \lesssim \big\|\mathcal{F}_{\tau,x}^{-1} ( \langle d_\tau(\sigma_1+k_1^2) \rangle^{\frac{1}{2}}  |\widetilde{\Pi_\tau v^n_1}|  )\big\|_{X_\tau^{\frac{1}{2},0}}   \|\mathcal{F}_{\tau,x}^{-1} ( |\widetilde{\Pi_\tau v_2^n}|)  \|_{X_\tau^{\frac{1}{2},\frac{1}{2}}} \big\|\mathcal{F}_{\tau,x}^{-1} ( \langle k_3\rangle^{\frac{1}{2}}|\widetilde{\Pi_\tau \overline{v}_3^n|} )\big\|_{X_\tau^{0,\frac{3}{8}}}    \notag \\
			& \lesssim \|v_1^n\|_{X_\tau^{\frac{1}{2},\frac{1}{2}}}\|v_2^n\|_{X_\tau^{\frac{1}{2},\frac{1}{2}}}  \|\overline{v}_3^n\|_{X_\tau^{\frac{1}{2},\frac{3}{8}}},\notag \\
			J(M_2) & \lesssim \|v_1^n\|_{X_\tau^{\frac{1}{2},\frac{1}{2}}}  \|v_2^n\|_{X_\tau^{\frac{1}{2},\frac{1}{2}}}\|\overline{v}_3^n\|_{X_\tau^{\frac{1}{2},\frac{3}{8}}},\notag \\ 
			J(M_3) & \lesssim  \big\| \Pi_\tau \big(  \mathcal{F}_{\tau,x}^{-1} (|\widetilde{\Pi_\tau v_1^n }|)  \mathcal{F}_{\tau,x}^{-1} (|\widetilde{\Pi_\tau v_2^n }|) \mathcal{F}_{\tau,x}^{-1} ( \langle d_\tau(\sigma_3-k_3^2) \rangle^{\frac{1}{2}} \langle k_3 \rangle^{\frac{1}{2}} \widetilde{|\Pi_\tau v^n_3|}) \big) \big\|_{X_\tau^{0,-\frac{3}{8}}} \notag \\
			& \lesssim \big\|\mathcal{F}_{\tau,x}^{-1} (|\widetilde{\Pi_\tau v_1^n }|)  \mathcal{F}_{\tau,x}^{-1} (|\widetilde{\Pi_\tau v_2^n }|)  \mathcal{F}_{\tau,x}^{-1} ( \langle d_\tau(\sigma_3-k_3^2) \rangle^{\frac{1}{2}} \langle k_3 \rangle^{\frac{1}{2}} \widetilde{|\Pi_\tau {v}^n_3|}) \big\|_{l_\tau^{\frac{4}{3}} L^\frac{4}{3} } \notag \\
			&\lesssim  \| \Pi_\tau v_1^n\|_{l_\tau^8 L^8} \| \Pi_\tau v_2^n \|_{l_\tau^8 L^8} \big\|\mathcal{F}_{\tau,x}^{-1} ( \langle d_\tau(\sigma_3-k_3^2) \rangle^{\frac{1}{2}} \langle k_3 \rangle^{\frac{1}{2}} \widetilde{|\Pi_\tau {v}^n_3|}) \big\|_{l_\tau^{2} L^2} \notag \\
			& \lesssim \|v_1^n\|_{X_\tau^{\frac{1}{2},\frac{7}{16}}} \|v_2^n\|_{X_\tau^{\frac{1}{2},\frac{7}{16}}}  \|\overline{v}_3^n\|_{X_\tau^{\frac{1}{2},\frac{1}{2}}},\notag
		\end{align} 
		where \eqref{Eqn:equivalent-conjugate} was used for the estimate of $J(M_3)$.
		By \eqref{Eqn:T-inverse}, we further have
		$$
		J(M_1)+J(M_2)+J(M_3)\lesssim T_1^{\varepsilon} \|v_1^n\|_{X_\tau^{\frac{1}{2},\frac{1}{2}}} \|v_1^n\|_{X_\tau^{\frac{1}{2},\frac{1}{2}}} \|\overline{v}_3^n\|_{X_\tau^{\frac{1}{2},\frac{1}{2}}}.
		$$

		Finally, we turn to the estimate for $N$:
		$$
		\bigg\|\frac{{1}_{|k| \leq \tau^{-1/2}}}{(2\pi)^{2}} \int_*\sum_* N(\sigma_1,\sigma_2,\sigma_3,k_1,k_2,k_3)\prod_{j=1}^3 |f_j(\sigma_j,k_j)|\bigg\|_{L^2l^2}=:J(N).
		$$
		Again, using \eqref{Eqn:discreteX-38bound} and \eqref{Eqn:discreteL4bound}, we arrive at
		$$
		\begin{aligned}
			J(N) & \lesssim  \big\| \Pi_\tau ( \mathcal{F}_{\tau,x}^{-1} ( \langle k_1\rangle^{\frac{1}{2}}|\widetilde{\Pi_\tau v_1^n|} ) \mathcal{F}_{\tau,x}^{-1} ( \langle k_2\rangle^{\frac{1}{2}}|\widetilde{\Pi_\tau v_2^n|} )  \mathcal{F}_{\tau,x}^{-1} ( \langle k_3\rangle^{\frac{1}{2}}|\widetilde{\Pi_\tau v_3^n|} ) \big) \big\|_{X_{\tau}^{0,-\frac{3}{8}}}  \\
			&\lesssim  \big\| \mathcal{F}_{\tau,x}^{-1} ( \langle k_1\rangle^{\frac{1}{2}}|\widetilde{\Pi_\tau v_1^n|} ) \mathcal{F}_{\tau,x}^{-1} ( \langle k_2\rangle^{\frac{1}{2}}|\widetilde{\Pi_\tau v_2^n|} )  \mathcal{F}_{\tau,x}^{-1} ( \langle k_3\rangle^{\frac{1}{2}}|\widetilde{\Pi_\tau \overline{v}_3^n|} )\big\|_{l_\tau^{\frac{4}{3}} L^\frac{4}{3}}  \\
			&\lesssim  \big\| \mathcal{F}_{\tau,x}^{-1} ( \langle k_1\rangle^{\frac{1}{2}}|\widetilde{\Pi_\tau v_1^n|} )\big\|_{l_\tau^4 L^4}\big\| \mathcal{F}_{\tau,x}^{-1} ( \langle k_2\rangle^{\frac{1}{2}}|\widetilde{\Pi_\tau v_2^n|} )\big\|_{l_\tau^4 L^4} \big\| \mathcal{F}_{\tau,x}^{-1} ( \langle k_3\rangle^{\frac{1}{2}}|\widetilde{\Pi_\tau \overline{v}_3^n|} )\big\|_{l_\tau^4 L^4} \\
			&\lesssim  T_1^{\varepsilon} \|v_1^n\|_{X_\tau^{\frac{1}{2},\frac{1}{2}}} \|v_2^n\|_{X_\tau^{\frac{1}{2},\frac{1}{2}}} \|\overline{v}_3^n\|_{X_\tau^{\frac{1}{2},\frac{1}{2}}}.
		\end{aligned}
		$$
		This finally shows the desired result.
	\end{proof}

	\begin{lemma}\label{Lem:multiplier1-Y}
		We adopt the notation from Lemma~\ref{Lem:multiplier1} and define
		$$
		\widetilde{M}(\sigma_1,\sigma_2,\sigma_3,k_1,k_2,k_3)= \langle d_\tau(\sigma+k^2)\rangle^{-\frac{1}{2}} M(\sigma_1,\sigma_2,\sigma_3,k_1,k_2,k_3) , 
		$$
		and, for $\beta \in (0,\frac{1}{6})$,
		$$
		\begin{aligned}
			&\widetilde{M}_0(\sigma_1,\sigma_2,\sigma_3,k_1,k_2,k_3)\\
			&\qquad=\frac{\chi_{A_0}}{\langle k \rangle^{\frac{1}{2}-3\beta} \langle k_1 \rangle^{\frac{1}{2}} \langle k_2 \rangle^{\frac{1}{2}} \langle k_3 \rangle^{\frac{1}{2}-3\beta} \langle d_\tau(\sigma_1+k_1^2)\rangle^{\frac{1}{2}+\beta} \langle d_\tau(\sigma_2+k_2^2)\rangle^{\frac{1}{2}+\beta}\langle d_\tau(\sigma_3-k_3^2)\rangle^{\frac{1}{2}+\beta} }, \\
			&\widetilde{M}_j(\sigma_1,\sigma_2,\sigma_3,k_1,k_2,k_3)= \langle d_\tau(\sigma+k^2)\rangle^{-\frac{1}{2}} M_j(\sigma_1,\sigma_2,\sigma_3,k_1,k_2,k_3) ,\quad j=1,2,3, \\
			&\widetilde{N}(\sigma_1,\sigma_2,\sigma_3,k_1,k_2,k_3)= \langle d_\tau(\sigma+k^2)\rangle^{-\frac{1}{2}}  N(\sigma_1,\sigma_2,\sigma_3,k_1,k_2,k_3).
		\end{aligned}
		$$
		If $|k|,|k_j|\leq \tau^{-\frac{1}{2}}$, then the estimate
		$$
		|\widetilde{M}| \lesssim \widetilde{M}_0+\widetilde{M}_1+\widetilde{M}_2+\widetilde{M}_3+\widetilde{N}
		$$
		holds true.
	\end{lemma}
	\begin{proof}
		As an example, we consider the region $A_0$ and verify that
		$$
		\langle d_\tau(\sigma+k^2)\rangle^{-\frac{1}{2}}M_0 \lesssim \widetilde{M}_0+\widetilde{N}.
		$$
		The other regions are treated similarly. As before, the analysis is divided into four distinct cases.
		
		\medskip
		
		(1)  $|k|>2|k_1|$ and $|k|>2|k_1|$: In this case we have $|k_3|\leq 2|k|$. From the proof of Lemma~\ref{Lem:multiplier1}, we recall that
		$$
		\begin{aligned}
			|d_\tau( (k-k_1)(k-k_2) )| \geq \frac{\sin 2}{2}|(k-k_1)(k-k_2)| \geq \frac{\sin 2}{8}|k|^2 
		\end{aligned}
		$$
		and
		$$
		|d_\tau ( (k-k_1)(k-k_2) )| \leq |d_\tau ({\sigma+k^2})| + |d_\tau ({\sigma_1+k_1^2} )|+|d_\tau ({\sigma_2+k_2^2})| +|d_\tau ({\sigma_3-k_3^2})|.
		$$
		Therefore, we obtain
		$$
		\langle k \rangle^2 \lesssim \langle d_\tau(\sigma+k^2)\rangle.
		$$
		This further implies
		$$
		\langle d_\tau( \sigma +k^2) \rangle^{-\frac{1}{2}} \lesssim \langle d_\tau(\sigma_1+ k_1^2) \rangle^{-\beta} \langle d_\tau(\sigma_2+ k_2^2) \rangle^{-\beta} \langle d_\tau(\sigma_3-k_3^2) \rangle^{-\beta} \langle k \rangle^{-\frac{1}{2}+3\beta}\langle k_3\rangle^{-\frac{1}{2}+3\beta},
		$$
		which shows
		$$
		\langle d_\tau(\sigma+k^2)\rangle^{-\frac{1}{2}}M_0 \lesssim \widetilde{M}_0.
		$$
		
		\medskip
		
		(2) $|k|\leq 2 |k_1|$ and $|k|\leq 2|k_2|$: In this case we can verify that $|k_3|\leq 4\max\{ |k_1|,|k_2|\}$ and $|k|\leq 2\min\{|k_1|,|k_2|\}$. This implies $|\widetilde{M}|\lesssim \widetilde{N}$.
		
		\medskip
		
		(3) $|k|>2|k_1|$ and $|k|\leq 2|k_2|$: In this case we have $|k|\leq 2|k-k_1|$ and $|k| |k-k_2|\leq 2\langle (k-k_1)(k-k_2)\rangle$. We assume $k_3 \neq 0$, as the term $\widetilde{M}$ vanishes otherwise. In the subregion where $|k_1|\leq |k-k_2|$ we obtain $|k_3|\leq |k-k_2|+|k_1|\leq 2|k-k_2|$. Then, by noting that 
		$$
		|(k-k_1)(k-k_2)|\leq |k|^2+|kk_1|+|kk_2|+|k_1k_2|<
		|k|^2+\frac{1}{2}|k|^2+|kk_2|+\frac{1}{2}|kk_2|\leq 3\tau^{-1} 
		$$ 
		where $|k|,|k_j|\leq \tau^{-\frac{1}{2}}$ was used, we have 
		$$
		\langle k \rangle \langle k_3 \rangle \lesssim       \langle (k-k_1)(k-k_2) \rangle \lesssim   \langle d_\tau( (k-k_1)(k-k_2) )\rangle \lesssim \langle d_\tau(\sigma+k^2)\rangle.
		$$
		This further gives
		$$
		\begin{aligned}
			\langle d_\tau(\sigma+k^2)\rangle^{-\frac{1}{2}} & \lesssim  \langle d_\tau(\sigma+k^2)\rangle^{-\frac{1}{2}+3\beta} \langle d_\tau(\sigma_1+k_1^2)\rangle^{-\beta} \langle d_\tau(\sigma_2+k_2^2)\rangle^{-\beta}  \langle d_\tau(\sigma_3-k_3^2)\rangle^{-\beta} \\
			& \lesssim   \langle k\rangle^{-\frac{1}{2}+3\beta}  \langle k_3\rangle^{-\frac{1}{2}+3\beta} \langle d_\tau(\sigma_1+k_1^2)\rangle^{-\beta} \langle d_\tau(\sigma_2+k_2^2)\rangle^{-\beta}  \langle d_\tau(\sigma_3-k_3^2)\rangle^{-\beta}.
		\end{aligned}
		$$
		From this, it follows that
		$$
		\langle d_\tau(\sigma+k^2)\rangle^{-\frac{1}{2}}M_0 \lesssim \widetilde{M}_0.
		$$
		On the other hand, we have $|k_3|\leq 2|k_1|$ in the subregion where $|k_1|>|k-k_2|$. Then we obtain $\widetilde{M} \lesssim \widetilde{N}$.
		
		\medskip
		
		(4) $|k|\leq2|k_1|$ and $|k| > 2|k_2|$: The desired estimate follows from the symmetry of $\widetilde{M}$ in $k_1$ and $k_2$.
	\end{proof}

	\begin{theorem}\label{Thm:multiplier1-Y}
		For all sequences $\{v^n_j\}_{n \in \mathbb{Z}}$ $(j = 1,2,3)$ supported on the time grid $n\tau \in [-2T_1, 2T_1]$ for $\tau,T_1  \in (0,1]$, there exists $\varepsilon>0$ such that
		$$
		\|\Pi_\tau(\Pi_\tau v_1^n\Pi_\tau v_2^n \, \partial_x \Pi_\tau v_3^n)\|_{Y_\tau^{\frac{1}{2},-1}}  \lesssim T_1^\varepsilon \|v_1^n\|_{X_\tau^{\frac{1}{2},\frac{1}{2}}} \|v_2^n\|_{X_\tau^{\frac{1}{2},\frac{1}{2}}} \|\overline{v}_3^n\|_{X_\tau^{\frac{1}{2},\frac{1}{2}}}.
		$$
	\end{theorem}
	\begin{proof}
		We adopt the notation from Theorem~\ref{Thm:multiplier1} and note that
		$$
		\|\Pi_\tau(\Pi_\tau v_1^n\Pi_\tau v_2^n \partial_x \Pi_\tau v_3^n)\|_{Y_\tau^{\frac{1}{2},-1}}=\bigg\|\frac{{1}_{|k| \leq \tau^{-1/2}}}{(2\pi)^{2}} \int_*\sum_* \widetilde{M}(\sigma_1,\sigma_2,\sigma_3,k_1,k_2,k_3)\prod_{j=1}^3 f_j(\sigma_j,k_j)\bigg\|_{L^1l^2}. 
		$$
		Thanks to Lemma~\ref{Lem:multiplier1-Y},  it suffices to estimate the above term with $f_j$ replaced by $|f_j|$ and $\widetilde{M}$ replaced by $\widetilde{M}_l$ ($l=0,1,2,3$) and $\widetilde{N}$.
		
		We first consider the estimate for $\widetilde{M}_0$: 
		$$
		\bigg\|\frac{{1}_{|k| \leq \tau^{-1/2}}}{(2\pi)^{2}} \int_*\sum_* \widetilde{M}_0 (\sigma_1,\sigma_2,\sigma_3,k_1,k_2,k_3)\prod_{j=1}^3 |f_j(\sigma_j,k_j)|\bigg\|_{L^1l^2}=:J(\widetilde{M}_0).
		$$
		By Young's and H\"older's inequalities, we have
		\begin{align}
			& \bigg\|\frac{{1}_{|k| \leq \tau^{-1/2}}}{(2\pi)^{2}} \int_*  \widetilde{M}_0 (\sigma_1,\sigma_2,\sigma_3,k_1,k_2,k_3)\prod_{j=1}^3 |f_j(\sigma_j,k_j)|\bigg\|_{L^1(\sigma)} \notag \\
			&\qquad ={1}_{|k| \leq \tau^{-1/2}} \langle k \rangle^{-\frac{1}{2}+3\beta}\notag \\
			&\qquad\quad\cdot\bigg\|\int_{-\frac{\pi}{\tau}}^{\frac{\pi}{\tau}}  \int_{-\frac{\pi}{\tau}}^{\frac{\pi}{\tau}}  \prod_{j=1}^{2} \frac{|f_j(\sigma_j, k_j)|}{\langle k_j\rangle^{\frac{1}{2}}\langle d_\tau( \sigma_j+k_j^2) \rangle^{\frac{1}{2}+\beta}}  \frac{ |f_3(\sigma -  \sigma_1-\sigma_2, k_3)|}{\langle k_3 \rangle^{\frac{1}{2}-3\beta} \langle d_\tau(\sigma -  \sigma_1-\sigma_2 -k_3^2) \rangle^{\frac{1}{2}+\beta} }\dd\sigma_1\dd\sigma_2\bigg\|_{L^1(\sigma)}\notag \\
			&\qquad\lesssim {1}_{|k| \leq \tau^{-1/2}} \langle k \rangle^{-\frac{1}{2}+3\beta}  \prod_{j=1}^{2} 
			\bigg\| \frac{|f_j(\sigma_j, k_j)|}{\langle k_j\rangle^{\frac{1}{2}}\langle d_\tau( \sigma_j+k_j^2) \rangle^{\frac{\beta}{2}}} \bigg\|_{L^2(\sigma_j)} \bigg\| \frac{ |f_3(\sigma_3, k_3)|}{\langle k_3 \rangle^{\frac{1}{2}-3\beta} \langle d_\tau(\sigma_3 -k_3^2) \rangle^{\frac{\beta}{2}} }\bigg\|_{L^2(\sigma_3)} \notag\\
			&\qquad\quad\cdot  \prod_{j=1}^{2}   \|\langle d_\tau( \sigma_j+k_j^2) \rangle^{-\frac{1}{2}-\frac{\beta}{2}}\|_{L^2(\sigma_j)}   \|\langle d_\tau( \sigma_3-k_3^2) \rangle^{-\frac{1}{2}-\frac{\beta}{2}}\|_{L^2(\sigma_3)}.\notag
		\end{align}
		Thanks to \eqref{Eqn:dtau/2<=dtau} and the fact that $|d_{\tau}( \tfrac{\sigma_j\pm k_j^2}{2})| \sim | \tfrac{\sigma_j\pm k_j^2}{2} |$ since $|\tfrac{\tau (\sigma_j\pm k_j^2)}{2} | \leq \pi$, we have
		$$
		\|\langle d_\tau( \sigma_j\pm k_j^2) \rangle^{-\frac{1}{2}-\frac{\beta}{2}}\|_{L^2(\sigma_j)} \leq  \bigg\|\bigg\langle d_\tau\bigg( \frac{\sigma_j\pm k_j^2}{2}\bigg) \bigg\rangle^{-\frac{1}{2}-\frac{\beta}{2}}\bigg\|_{L^2(\sigma_j)} \lesssim \bigg\|\bigg\langle  \frac{\sigma_j\pm k_j^2}{2} \bigg\rangle^{-\frac{1}{2}-\frac{\beta}{2}}\bigg\|_{L^2(\sigma_j)} \lesssim 1.
		$$
		
		Denote $g_j(\sigma_j,k_j)=|f_j(\sigma_j,k_j)|\langle d_\tau( \sigma_j+k_j^2) \rangle^{-\frac{\beta}{2}}$ for $j=1,2$, and $g_3(\sigma_3,k_3)=|f_j(\sigma_3,k_3)|\langle d_\tau( \sigma_3-k_3^2) \rangle^{-\frac{\beta}{2}}$. Taking $\beta=\frac{1}{24}$ and using H\"older's and Young's inequalities, we arrive at
		$$
		\begin{aligned}
			J(\widetilde{M}_0) &\lesssim \| \langle k \rangle^{-\frac{3}{8}}\|_{l^4} \bigg\| \sum_{k=k_1+k_2+k_3} \langle k_1 \rangle^{-\frac{1}{2}} \langle k_2 \rangle^{-\frac{1}{2}} \langle k_3 \rangle^{-\frac{3}{8}} \prod_{j=1}^{3} \|g_j(\cdot,k_j)\|_{L^2}\bigg\|_{l^4}\\
			&\lesssim  \prod_{j=1}^{3} \big\|  \langle k_j \rangle^{-\frac{3}{8}}  \|g_j(\cdot,k_j)\|_{L^2} \big\|_{l^\frac{4}{3}} \lesssim \prod_{j=1}^{3}   \|g_j\|_{L^2 l^2} \lesssim \|v_1^n\|_{X_\tau^{\frac{1}{2},\frac{23}{48}}} \|v_2^n\|_{X_\tau^{\frac{1}{2},\frac{23}{48}}} \|\overline{v}_3^n\|_{X_\tau^{\frac{1}{2},\frac{23}{48}}}.
		\end{aligned}
		$$
		By \eqref{Eqn:T-inverse}, we have
		$$
		J(\widetilde{M}_0) \lesssim T_1^\varepsilon\|v_1^n\|_{X_\tau^{\frac{1}{2},\frac{1}{2}}} \|v_2^n\|_{X_\tau^{\frac{1}{2},\frac{1}{2}}} \|\overline{v}_3^n\|_{X_\tau^{\frac{1}{2},\frac{1}{2}}}.
		$$
		
		For the estimates of $\widetilde{M}_j$ ($j=1,2,3$) and $\widetilde{N}$, noting that
		$$
		\|\langle d_\tau(\sigma + k^2) \rangle^{-\frac{1}{2}} g (\sigma,k)\|_{L^1(\sigma)}
		\lesssim \|\langle d_\tau(\sigma + k^2) \rangle^{\frac{1}{8}} g (\sigma,k)\|_{L^2(\sigma)},
		$$
		we derive that 
		$$
		\begin{aligned}
			J(\widetilde{M}_j)& :=\bigg\|\frac{{1}_{|k| \leq \tau^{-1/2}}}{(2\pi)^{2}} \int_*\sum_* \widetilde{M}_j (\sigma_1,\sigma_2,\sigma_3,k_1,k_2,k_3)\prod_{j=1}^3 |f_j(\sigma_j,k_j)|\bigg\|_{L^1l^2} \\
			&\lesssim \bigg\|\langle d_\tau(\sigma + k^2) \rangle^{\frac{1}{8}} \frac{{1}_{|k| \leq \tau^{-1/2}}}{(2\pi)^{2}} \int_*\sum_*  {M}_j (\sigma_1,\sigma_2,\sigma_3,k_1,k_2,k_3)\prod_{j=1}^3 |f_j(\sigma_j,k_j)|\bigg\|_{L^2l^2},\\
			J(\widetilde{N})& :=\bigg\|\frac{{1}_{|k| \leq \tau^{-1/2}}}{(2\pi)^{2}} \int_*\sum_* \widetilde{N}(\sigma_1,\sigma_2,\sigma_3,k_1,k_2,k_3)\prod_{j=1}^3 |f_j(\sigma_j,k_j)|\bigg\|_{L^1 l^2} \\
			& \lesssim \bigg\|\langle d_\tau(\sigma + k^2) \rangle^{\frac{1}{8}}\frac{{1}_{|k| \leq \tau^{-1/2}}}{(2\pi)^{2}} \int_*\sum_*  {N}(\sigma_1,\sigma_2,\sigma_3,k_1,k_2,k_3)\prod_{j=1}^3 |f_j(\sigma_j,k_j)|\bigg\|_{L^2 l^2}.
		\end{aligned}
		$$
		Consequently, the estimates for $M_j$ and $N$ directly imply those for $\widetilde{M}_j$ and $\widetilde{N}$. Specifically, for $J(\widetilde{M}_1)$, we have
		$$
		J(\widetilde{M}_1) \lesssim \big\|\mathcal{F}_{\tau,x}^{-1} ( \langle d_\tau(\sigma_1+k_1^2) \rangle^{\frac{1}{2}}  |\widetilde{\Pi_\tau v^n_1}|  )  \mathcal{F}_{\tau,x}^{-1} ( |\widetilde{\Pi_\tau v_2^n}|)    \mathcal{F}_{\tau,x}^{-1}( \langle k_3\rangle^{\frac{1}{2}}|\widetilde{\Pi_\tau v_3^n|} )\big\|_{X_{\tau}^{0,-\frac{3}{8}}}.
		$$
		This allows us to utilize the previously established estimate for $J(M_1)$, which was
		$$
		\begin{aligned}
			J(M_1) & \lesssim \big\|\mathcal{F}_{\tau,x}^{-1} ( \langle d_\tau(\sigma_1+k_1^2) \rangle^{\frac{1}{2}}  |\widetilde{\Pi_\tau v^n_1}|  )  \mathcal{F}_{\tau,x}^{-1} ( |\widetilde{\Pi_\tau v_2^n}|)  \mathcal{F}_{\tau,x}^{-1}  ( \langle k_3\rangle^{\frac{1}{2}}|\widetilde{\Pi_\tau v_3^n|} )\big\|_{X_{\tau}^{0,-\frac{3}{8}}} \\
			& \lesssim T_1^{\varepsilon} \|v_1^n\|_{X_\tau^{\frac{1}{2},\frac{1}{2}}}\|v_2^n\|_{X_\tau^{\frac{1}{2},\frac{1}{2}}}  \|\overline{v}_3^n\|_{X_\tau^{\frac{1}{2},\frac{1}{2}}}.
		\end{aligned}
		$$
		This eventually gives the desired result.
	\end{proof}


\begin{thebibliography}{10}
		\bibitem{bahouri2022global}
		H.~Bahouri and G.~Perelman.
		\newblock Global well-posedness for the derivative nonlinear {S}chr{\"o}dinger
		equation.
		\newblock {\em Invent. Math.}, 229:639--688, 2022.
		
		\bibitem{BiagioniLinares}
	    H.~Biagioni and F.~Linares.
		\newblock Ill-posedness for the derivative {S}chr{\"o}dinger and generalized
		{Benjamin-Ono} equations.
		\newblock {\em Trans. Amer. Math. Soc.}, 353:3649--3659, 2001.
		
		\bibitem{bourgain1993}
		J.~Bourgain.
		\newblock Fourier transform restriction phenomena for certain lattice subsets
		and applications to nonlinear evolution equations {I}. {S}chr{\"o}dinger
		equations.
		\newblock {\em Geom. Funct. Anal}, 3:107--156, 1993.
		
		\bibitem{caonew}
		J.~Cao, B.~Li, and Y.~Lin. 
		\newblock A new second-order low-regularity integrator for the cubic nonlinear {S}chr{\"o}dinger equation. 
		\newblock {\em IMA J. Numer. Anal.}, 44:1313--1345, 2024.
		
		\bibitem{caolly}
		J.~Cao, B.~Li, Y.~Lin, and F.~Yao.
		\newblock Numerical approximation of discontinuous solutions of the semilinear wave equation.
		\newblock {\em SIAM J. Numer. Anal.}, 63:214--238, 2025.
		
		
		\bibitem{colliander2001global}
		J.~Colliander, M.~Keel, G.~Staffilani, H.~Takaoka, and T.~Tao.
		\newblock Global well-posedness for {S}chr{\"o}dinger equations with
		derivative.
		\newblock {\em SIAM J. Math. Anal.}, 33:649--669, 2001.
		
		\bibitem{colliander2002refined}
		J.~Colliander, M.~Keel, G.~Staffilani, H.~Takaoka, and T.~Tao.
		\newblock A refined global well-posedness result for {S}chr{\"o}dinger
		equations with derivative.
		\newblock {\em SIAM J. Math. Anal.}, 34:64--86, 2002.
		
		\bibitem{deng2021optimal}
		Y.~Deng, A.~R. Nahmod, and H.~Yue.
		\newblock Optimal local well-posedness for the periodic derivative nonlinear
		{S}chr{\"o}dinger equation.
		\newblock {\em Commun. Math. Phys.}, 384:1061--1107, 2021.
		
		\bibitem{Dysthe1979-waterNLS}
		K.~B. Dysthe.
		\newblock Note on a modification to the nonlinear {S}chr{\"o}dinger equation
		for application to deep water waves.
		\newblock {\em Proc. R. Soc. Lond. A}, 369:105--114, 1979.
		
		
		\bibitem{feng_maierhofer_schratz_2023}
		Y.~Feng, G.~Maierhofer, and K.~Schratz.
		\newblock Long-time error bounds of low-regularity integrators for nonlinear
		{S}chr{\"o}dinger equations.
		\newblock {\em Math. Comp.}, 93:1569--1598, 2023.
		
		\bibitem{feng2025explicit}
		Y.~Feng, G.~Maierhofer, and C.~Wang.
		\newblock Explicit symmetric low-regularity integrators for the nonlinear
		{S}chr{\"o}dinger equation.
		\newblock {\em SIAM J. Sci. Comput.}, 47:A2154--A2179, 2025.
		
		\bibitem{fengschratz2024sg}
		Y.~Feng and K.~Schratz.
		\newblock Improved uniform error bounds on a {L}awson-type exponential
		integrator for the long-time dynamics of sine-{G}ordon equation.
		\newblock {\em Numer. Math.}, 156:1455--1477, 2024.
		
		
		\bibitem{Hayashi1993}
		N.~Hayashi.
		\newblock The initial value problem for the derivative nonlinear
		{S}chr{\"o}dinger equation in the energy space.
		\newblock {\em Nonlinear Anal.}, 20:823--833, 1993.
		
		\bibitem{HayashiOzawa1992}
		N.~Hayashi and T.~Ozawa.
		\newblock On the derivative nonlinear {S}chr{\"o}dinger equation.
		\newblock {\em Physica D}, 55:14--36, 1992.
		
		
		\bibitem{Herr2006}
		S.~Herr.
		\newblock On the Cauchy problem for the derivative nonlinear {S}chr{\"o}dinger
		equation with periodic boundary condition.
		\newblock {\em Int. Math. Res. Not. IMRN}, 2006:Art. ID~96763, 33 pp., 2006. 
		
		\bibitem{Hochacta}
		M.~Hochbruck and A.~Ostermann.
		\newblock Exponential integrators.
		\newblock {\em Acta Numer.}, 19:209--286, 2010.
		
	
		
		\bibitem{Jimcom}
		L.~Ji, H.~Li, A.~Ostermann, and C.~Su.
		\newblock Filtered {Lie-Trotter} splitting for the “good” {B}oussinesq
		equation: low regularity error estimates.
		\newblock {\em Math. Comp.}, 94:2345--2365, 2025.
		
		\bibitem{ji2026low}
		L.~Ji, H.~Li, and A.~Ostermann.
		\newblock A low regularity exponential-type integrator for the derivative nonlinear {S}chr{\"o}dinger equation.
		\newblock {\em Preprint arXiv:2601.20212}, 2026.
		
		
		
		\bibitem{Jisiam}
		L.~Ji, A.~Ostermann, F.~Rousset, and K.~Schratz.
		\newblock Low regularity full error estimates for the cubic nonlinear
		{S}chr{\"o}dinger equation.
		\newblock {\em SIAM J. Numer. Anal.}, 62:2071--2086, 2024.
		
		\bibitem{Jiima}
		L.~Ji, A.~Ostermann, F.~Rousset, and K.~Schratz.
		\newblock Low regularity error estimates for the time integration of 2{D}
		{NLS}.
		\newblock {\em IMA J. Numer. Anal.}, 45:2023--2059, 2025.
		
		\bibitem{JiZhao}
		L.~Ji and X.~Zhao.
		\newblock Error estimates of time-splitting schemes for nonlinear {Klein--Gordon} equation with rough data.
		\newblock To appear in {\em SIAM J. Numer. Anal.}, 2026.
		
		\bibitem{kaup1978exact}
		D.~J. Kaup and A.~C. Newell.
		\newblock An exact solution for a derivative nonlinear {S}chr{\"o}dinger
		equation.
		\newblock {\em J. Math. Phys.}, 19:798--801, 1978.
		
		
		\bibitem{lima2022ns}
		B.~Li, S.~Ma, and K.~Schratz.
		\newblock A semi-implicit exponential low-regularity integrator for the
		{N}avier--{S}tokes equations.
		\newblock {\em SIAM J. Numer. Anal.}, 60:2273--2292, 2022.
		
		\bibitem{lischratzzivcovich2023kg}
		B.~Li, K.~Schratz, and F.~Zivcovich.
		\newblock A second-order low-regularity correction of {L}ie splitting for the
		semilinear {K}lein--{G}ordon equation.
		\newblock {\em ESAIM Math. Model. Numer. Anal.}, 57:899--919, 2023.
		
		\bibitem{liwunls}
		B.~Li and Y.~Wu.
		\newblock A fully discrete low-regularity integrator for the 1{D} periodic cubic
		nonlinear {S}chr{\"o}dinger equation.
		\newblock {\em Numer. Math.}, 149:151--183, 2021.
		
		\bibitem{liwu2026kdv}
		B.~Li and Y.~Wu.
		\newblock An unfiltered low-regularity integrator for the {K}d{V} equation
		with solutions below $H^1$.
		\newblock {\em Found. Comput. Math.}, 26:1321--1380, 2026.
		
		
		\bibitem{li2018numerical}
		S. Li, X. Li, and F. Shi.
		\newblock Numerical methods for the derivative nonlinear {S}chr{\"o}dinger
		equation.
		\newblock {\em Int. J. Nonlinear Sci. Numer. Simul.}, 19:239--249, 2018.
		
		\bibitem{mio1976modified}
		K.~Mio, T.~Ogino, K.~Minami, and S.~Takeda.
		\newblock Modified nonlinear {S}chr{\"o}dinger equation for {A}lfv{\'e}n waves
		propagating along the magnetic field in cold plasmas.
		\newblock {\em J. Phys. Soc. Japan}, 41:265--271, 1976.
		
		\bibitem{mjolhus1976modulational}
		E.~Mj{\o}lhus.
		\newblock On the modulational instability of hydromagnetic waves parallel to
		the magnetic field.
		\newblock {\em J. Plasma Phys.}, 16:321--334, 1976.
		
		\bibitem{mjolhus1989nonlinear}
		E.~Mj{\o}lhus.
		\newblock Nonlinear {A}lfv{\'e}n waves and the {DNLS equation}: oblique
		aspects.
		\newblock {\em Phys. Scr.}, 40:227--237, 1989. 
		
		\bibitem{mjolhus1986alfven}
		E.~Mj{\o}lhus and J.~Wyller.
		\newblock Alfv{\'e}n solitons.
		\newblock {\em Phys. Scr.}, 33:442--451, 1986. 
		
		\bibitem{mosesself2007}
		J.~Moses, B.~A. Malomed, and F.~W. Wise.
		\newblock Self-steepening of ultrashort optical pulses without
		self-phase-modulation.
		\newblock {\em Phys. Rev. A}, 76:021802, 2007.
		
		
		\bibitem{Ostjems}
		A.~Ostermann, F.~Rousset, and K.~Schratz.
		\newblock Fourier integrator for periodic {NLS}: low regularity estimates via
		discrete {Bourgain} spaces.
		\newblock {\em J. Eur. Math. Soc.}, 25:3913--3952, 2023.
		
		\bibitem{ostfocm}
		A.~Ostermann and K.~Schratz.
		\newblock Low regularity exponential-type integrators for semilinear {S}chr{\"o}dinger equations.
		\newblock {\em Found. Comput. Math.}, 18:731--755, 2018.
		
		\bibitem{Roupaa}
		F.~Rousset and K.~Schratz.
		\newblock Convergence error estimates at low regularity for time
		discretizations of {KdV}.
		\newblock {\em Pure Appl. Anal.}, 4:127--152, 2022.
		
		
		\bibitem{Takaoka1999}
		H.~Takaoka.
		\newblock Well-posedness for the one-dimensional nonlinear {S}chr{\"o}dinger
		equation with the derivative nonlinearity.
		\newblock {\em Adv. Differential Equations}, 4:561--680, 1999.
		
		\bibitem{Tao2006}
		T.~Tao. 
		\newblock Nonlinear Dispersive Equations: Local and Global Analysis.
		\newblock Amer. Math. Soc., Providence RI, 2006.
		
		\bibitem{wu2014global}
		Y.~Wu.
		\newblock Global well-posedness for the nonlinear {S}chr{\"o}dinger equation
		with derivative in energy space.
		\newblock {\em Anal. PDE}, 6:1989--2002, 2014.
		
		\bibitem{wu2015global}
		Y.~Wu.
		\newblock Global well-posedness on the derivative nonlinear {S}chr{\"o}dinger
		equation.
		\newblock {\em Anal. PDE}, 8:1101--1112, 2015. 
		
		\bibitem{WuandYao2022}
		Y.~Wu and F.~Yao.
		\newblock A first-order {F}ourier integrator for the nonlinear {S}chr{\"o}dinger equation on $\mathbb T$ without loss of regularity.
		\newblock {\em Math. Comp.}, 91:1213--1235, 2022.
		
		\bibitem{Yang1974-intFormalism}
		C.~N. Yang.
		\newblock Integral formalism for gauge fields.
		\newblock {\em Phys. Rev. Lett.}, 33:445--447, 1974. 
		
		\bibitem{YangMills1954}
		C.~N. Yang and R.~L. Mills.
		\newblock Conservation of isotopic spin and isotopic gauge invariance.
		\newblock {\em Phys. Rev.}, 96:191--195, 1954.
		
	\end{thebibliography}
\end{document}